\documentclass[reqno,11pt]{amsart}

\usepackage[a4paper,left=23mm,right=23mm,top=30mm,bottom=30mm,marginpar=25mm]{geometry}
\usepackage{amssymb}
\usepackage{mathtools}
\usepackage{url}
\usepackage{color}
\usepackage{bbm}
\usepackage[noadjust]{cite}
\usepackage{enumitem}
\usepackage[hidelinks]{hyperref}
\renewcommand{\epsilon}{\varepsilon}

\numberwithin{equation}{section}

\newtheoremstyle{thmlemcorr}{10pt}{10pt}{\itshape}{}{\bfseries}{.}{10pt}{{\thmname{#1}\thmnumber{ #2}\thmnote{ (#3)}}}
\newtheoremstyle{thmlemcorr*}{10pt}{10pt}{\itshape}{}{\bfseries}{.}\newline{{\thmname{#1}\thmnumber{ #2}\thmnote{ (#3)}}}
\newtheoremstyle{defi}{10pt}{10pt}{\itshape}{}{\bfseries}{.}{10pt}{{\thmname{#1}\thmnumber{ #2}\thmnote{ (#3)}}}
\newtheoremstyle{remexample}{10pt}{10pt}{}{}{\bfseries}{.}{10pt}{{\thmname{#1}\thmnumber{ #2}\thmnote{ (#3)}}}
\newtheoremstyle{ass}{10pt}{10pt}{}{}{\bfseries}{.}{10pt}{{\thmname{#1}\thmnumber{ A#2}\thmnote{ (#3)}}}

\theoremstyle{thmlemcorr}
\newtheorem{theorem}{Theorem}
\numberwithin{theorem}{section}
\newtheorem{lemma}[theorem]{Lemma}
\newtheorem{corollary}[theorem]{Corollary}
\newtheorem{proposition}[theorem]{Proposition}

\theoremstyle{thmlemcorr*}
\newtheorem{theorem*}{Theorem}
\newtheorem{lemma*}[theorem]{Lemma}
\newtheorem{corollary*}[theorem]{Corollary}
\newtheorem{proposition*}[theorem]{Proposition}
\newtheorem{problem*}[theorem]{Problem}
\newtheorem{conjecture*}[theorem]{Conjecture}

\theoremstyle{defi}
\newtheorem{definition}[theorem]{Definition}

\theoremstyle{remexample}
\newtheorem{remarkx}[theorem]{Remark}
\newenvironment{remark}
  {\pushQED{\qed}\remarkx}
  {\popQED\endremarkx}
\newtheorem{example}[theorem]{Example}

\theoremstyle{ass}

\newcommand{\Gcal}{\mathcal{G}}
\newcommand{\Hcal}{\mathcal{H}}

\newcommand{\Scal}{\mathcal{S}}

\DeclareMathOperator{\diam}{diam}

\DeclareMathOperator{\dist}{dist}

\DeclareMathOperator{\supp}{supp}

\newcommand{\floor}[1]{\lfloor #1 \rfloor}

\newcommand{\normB}[1]{\Bigl\|#1\Bigr\|}

\newcommand{\abslr}[1]{\left|#1\right|}

\newcommand{\dprlr}[1]{\left\langle #1 \right\rangle}

\newcommand{\N}{\mathbb{N}}
\newcommand{\R}{\mathbb{R}}

\newcommand{\C}{\mathbb{C}}

\newcommand{\Sn}{\mathbb{S}^{n-1}}

\newcommand{\eps}{\epsilon}

\DeclareMathOperator{\Div}{div}

\def\XXint#1#2#3{{\setbox0=\hbox{$#1{#2#3}{\int}$}
\vcenter{\hbox{$#2#3$}}\kern-.5\wd0}}

\DeclarePairedDelimiter\abs{\lvert}{\rvert}
\DeclarePairedDelimiter{\norm}{\lVert}{\rVert}
\DeclarePairedDelimiter{\inner}{\langle}{\rangle}

\newcommand{\parent}[1]{\left( #1 \right)} %large parentheses
\newcommand{\Rn}{\R^{n}}

\renewcommand{\phi}{\varphi}

\newcommand{\dy}{\, dy}

\newcommand{\weakto}{\rightharpoonup}

\def\XXint#1#2#3{{\setbox0=\hbox{$#1{#2#3}{\int}$}
     \vcenter{\hbox{$#2#3$}}\kern-.5\wd0}}

\makeatletter
\g@addto@macro\bfseries{\boldmath}
\makeatother

\DeclareMathOperator{\diver}{div}

\DeclareMathOperator{\loc}{loc}

\renewcommand{\O}{\Omega}
\renewcommand{\a}{\alpha}

\renewcommand{\d}{\delta}

\newcommand{\s}{\sigma}

\newcommand{\p}{\partial}
\newcommand{\e}{\varepsilon}
\newcommand{\mc}{\mathcal}

\def\XXint#1#2#3{{\setbox0=\hbox{$#1{#2#3}{\int}$}
		\vcenter{\hbox{$#2#3$}}\kern-.5\wd0}}

\title[Nonlocal gradients]{Nonlocal gradients: fundamental theorem of calculus, Poincar\'e inequalities and embeddings}

\author{J. C. Bellido, C. Mora-Corral and H. Sch\"onberger}

\address[J. C. Bellido]{E.T.S.I.\ Industriales, Department of Mathematics, Universidad de Castilla-La Mancha,
13071-Ciudad Real, Spain.}
\email{JoseCarlos.Bellido@uclm.es}

\address[C. Mora-Corral]{Departamento de Matem\'aticas, Universidad Aut\'onoma de Madrid, 28049 Madrid, Spain and Instituto de Ciencias Matem\'aticas, CSIC-UAM-UC3M-UCM, 28049 Madrid, Spain.}
\email{carlos.mora@uam.es}

\address[H. Sch\"onberger]{Mathematisch-Geographische Fakult\"{a}t, Katholische Universit\"at Eichst\"att-Ingolstadt, 85072 Eichst\"att, Germany.}
\email{hidde.schoenberger@ku.de}

\begin{document}

\begin{abstract}
We address the study of nonlocal gradients defined through general radial kernels $\rho$.
Our investigation focuses on the properties of the associated function spaces, which
% are defined as the set of $p$-integrable functions whose distributional nonlocal gradient is also $p$-integrable.
%We explore the main structural characteristics of these spaces, which
 depend on the characteristics of the kernel function.
Specifically, even with minimal assumptions on $\rho$, we establish Poincar\'e inequalities and compact embeddings into Lebesgue spaces.
%\textcolor{red}{These assumptions serve as a characterization for the case when $p=2$.}
Additionally, we present a fundamental theorem of calculus that enables us to recover a function from its nonlocal gradient through a convolution. This is used to demonstrate embeddings into Orlicz spaces and spaces of continuous functions that mirror the well-known Sobolev and Morrey inequalities for classical gradients.
Finally, we establish conditions for inclusions and equality of spaces associated to different kernels.

\end{abstract}


% Submit to Journal of Functional Analysis, editor Brezis or De Philippis.

\maketitle
\thispagestyle{empty}

%\tableofcontents

\section{Introduction}

Many phenomena in nature, science and engineering are modeled with differential equations or local variational principles.
Locality in this context means that the behaviour of an object depends only on its immediate neighborhood.
However, there are situations where long-range interactions have to be taken into account.
This gives rise to nonlocal models involving integro-differential equations or integral operators.
The study of nonlocal models has proliferated in the last decades, as they provide effective ways to bridge between different length scales and lead to refined predictions.
Areas that have benefited from nonlocal modeling include materials science, diffusion processes, imaging and machine learning.

%An initial motivation for us was 
 An initial motivation for this work is \emph{peridynamics}, which is a new nonlocal approach to solid mechanics \cite{Silling2000} that has experienced huge progress and has led to a substantial literature; see, e.g., the books \cite{madenci_oterkus, Gerstle, BoFoGeSi17, DEDuGu19}, or the more mathematically-oriented articles \cite{BoBrMi2001, BeMC14, MeD16}.
Even though bond-based peridynamics, based on double-integral energies, is among the most widespread nonlocal models in mechanics, it was demonstrated in~\cite{BeCuMCBondBased, MeSc23} that it is incompatible with the classical model of nonlinear elasticity in the nonlocal-to-local limit.

To remedy this shortcoming of bond-based peridynamics, the authors of \cite{BeCuMC} adopted a model similar to the classical one (\cite{Ball77}) that involves, instead of the classical gradient, the \emph{Riesz fractional gradient} $D^s$ for $s\in (0,1)$, which
is defined for smooth functions $u : \Rn \to \R$ as
\begin{equation}\label{eq:Dsu}
 D^s u (x) = c_{n,s} \int_{\Rn} \frac{u(x)-u(y)}{|x-y|^{n+s}} \frac{x-y}{|x-y|} \, dy .
\end{equation}
%with $c_{n,s}$ a normalising constant. Note that we consider vector functions $u$ and that is why the tensor product $\otimes$ appears in the definition of $D^su$. The operator $D^s$ is related to the Riesz potential $I_{1-s}\sim  \left| \cdot \right|^{-(n-1+s)}$ via the formula 
%\begin{equation*}%\label{eq:DsuIDu}
%D^s u = I_{1-s} * \nabla u .
%\end{equation*} 
%In terms of the Fourier transform, one has $\widehat{D^s u}(\xi)=\frac{2\pi i \xi}{|2\pi \xi|^{1-s}} \widehat{u}(\xi)$ for $\xi\in \R^n$.  
%
%
%In \cite{Silhavy20}, \v{S}ilhav\'{y} identified $D^s$ as the unique operator determined (up to a constant) by the following properties: 1) invariance under translations and rotations;
%2) homogeneity of order $s$ under dilations; 3)
%continuity in a distributional sense.
The Riesz fractional gradient is a paradigmatic example of a nonlocal gradient.
%Hence, the Riesz fractional gradient is nowadays acknowledged as the natural fractional derivative.
It was Shieh \& Spector \cite{ShS2015, ShS2018} who brought it to the attention of the PDE and calculus of variations communities.
They introduced the function spaces associated to the Riesz gradient, with the key observation that they are equivalent to the \emph{Bessel potential spaces} $H^{s,p}(\R^n)$. Moreover, a series of useful inequalities and embeddings hold.
These fundamental insights laid the basis for an analysis of the equations and minimization problems related to the fractional gradient, which has led to the understanding of fractional function spaces \cite{Ponce16, COMI2019}, the existence of solutions in variational problems \cite{KrSc22, BeCuMC}, and the relationship with local models \cite{bellido2020convergence}.

%In search for an improved nonlocal description of elasticity and effects beyond, the \emph{fractional hyperelasticity} model is based on energy functionals of the form
%\begin{equation}\label{eq:EnergyDs}
% \int_{\Rn} W (x, D^s u (x)) \, dx 
%\end{equation}
%defined for $u\in H^{s,p}(\R^n)$. 
%In comparison with~\eqref{eq:funcclas}, they do not involve usual derivatives, and thus require less regularity, whereas they have the advantage over~\eqref{eq:energyPeridynamics} of having a gradient object.
%Yet, there remains the issue that the fractional gradient encodes interactions ranging in all of $\R^n$.

A drawback of the Riesz fractional gradient for certain applications is the fact that it requires integration over the whole space for its computation. To be able to work on bounded domains, as is desirable, for instance, for realistic materials modeling, the previous approach was modified in~\cite{BeCuMo22, CuKrSc23} by incorporating a horizon parameter.
This was implemented by truncating the Riesz fractional gradient. 
To be more explicit, the nonlocal gradient 
for a fractional parameter $s \in (0,1)$, a horizon $\delta>0$ and an appropriate non-negative, smooth, radial cut-off function $w_{\delta}$ supported in the ball $B_{\delta}(0)$ 
is defined as
\begin{equation}\label{Ddeltas}
D^{s}_\delta u (x) = \int_{B_{\delta}(x)} \frac{u(x)-u(y)}{\abs{x-y}} \frac{x-y}{\abs{x-y}}\frac{w_{\delta}(x-y)}{\abs{x-y}^{n+s-1}}\,dy \quad \text{for $x \in \Omega$.}
\end{equation}
Note that $u$ needs to be defined in the larger region $\Omega_{\delta}=\{x \in \R^n\,:\, \dist
(x,\Omega) < \delta\}$. 
%The associated functionals are then defined as
%\begin{equation}\label{eq:funcfrac}
% \int_{\Omega} W(x,D^s_{\delta}u (x))\,dx . 
%\end{equation}
As in the case of Sobolev spaces (and in the fractional case above), there are two ways to define the space $H^{s,p,\delta} (\Rn)$ associated with nonlocal gradients: as a completion of the smooth compactly supported functions under the norm $\left\| u \right\|_{L^p (\Rn)} + \left\| D^s_\delta u \right\|_{L^p (\Rn)}$, and through a distributional definition, which is based on a suitable integration by parts formula.
%For a given domain $\Omega\subset \Rn$, $H^{s,p,\delta}_0(\Omega)$ stands for the subset of functions in $H^{s,p,\delta} (\Rn)$ vanishing in $\O^c$.

 Analyses of the variational problems based on $D^s_\d u$ can be found in \cite{BeCuMo22, CuKrSc23, BellidoCuetoMoraCorral+2023, KrS24}.
%; they contain, in particular, criteria for weak lower semicontinuity, a general $\Gamma$-convergence statement for parameter-dependent families, and a derivation of Euler--Lagrange equations. 
Two techniques established in those papers are the \emph{translation method}~\cite{CuKrSc23, BeCuMo22} and the \emph{nonlocal fundamental theorem of calculus}~\cite{BeCuMo22}. The translation property is based on the observation that every nonlocal gradient is a classical gradient. Precisely, 
\begin{equation}\label{eq:translation}
D^s_\delta u= \nabla (Q_\delta^s\ast u),
\end{equation}  
where $Q_\delta^s$ is an integrable kernel supported in the ball $B_\delta(0)$. The nonlocal fundamental theorem of calculus refers to the representation formula
\begin{equation}\label{eq:NFTC}
 u = V_\delta^s \ast D_\delta^s u,  
\end{equation}
where $V_\delta^s$ is a locally integrable function implicitly given via Fourier transform. The identity \eqref{eq:NFTC} can be used to prove various embeddings and inequalities related to the nonlocal gradient $D^s_\d$ and the function spaces $H^{s,p,\d}(\R^n)$. Both \eqref{eq:translation} and \eqref{eq:NFTC} are modifications of analogous results in the fractional setting~\cite{ShS2015, Ponce16, COMI2019, KrSc22}. \smallskip

Expressions \eqref{eq:Dsu} and \eqref{Ddeltas} lead naturally to the central objects in this paper, which are general nonlocal gradients of the form 
\begin{equation*} \label{eq: general nonlocal gradient}
	D_\rho u(x)= \int_{\R^n} \frac{u(x)-u(y)}{|x-y|} \frac{x-y}{|x-y|} \rho(x-y)\, dy 
\end{equation*}
for some kernel $\rho$, typically with a singularity at the origin. 
Despite the fact that the axiomatic properties (invariance by translation and rotation, homogeneity and continuity \cite{Silhavy20}) of the Riesz gradient are desirable in many contexts, in some situations the use of other kernels $\rho$ presents advantages.
Perhaps the most relevant kernels for applications are those with compact support, since they allow for modeling phenomena in bounded domains.
We denote by $H^{\rho,p}(\R^n)$ the set of $L^p (\Rn)$ functions with an $L^p (\Rn)$ distributional nonlocal gradient, while, given an open $\O \subset \Rn$, the set $H^{\rho,p}_0(\Omega)$ comprises those $H^{\rho,p}(\R^n)$ functions vanishing in the complement of $\Omega$.

%Another reason for working with a general $\rho$ is the possibility of having different growths at zero and at infinity.

References on general nonlocal gradients are \cite{DGLZ,Delia, EGM22}, where vector calculus for nonlocal gradients is addressed, and \cite{MeS}, on localization properties of nonlocal gradients.
More precisely, this article benefits from the works \cite{LeDu20} for non-radial kernels, \cite{BeCuMo22, CuKrSc23} for the operator $D^s_{\d}$, and \cite{HaT23} for kernels supported in a half-ball.
In fact, while this article was being written, we became aware of the preprint \cite{Arr23}, which studies general nonlocal operators similar to ours, but focuses on different aspects: nonlocal-to-local estimates, localization, quasiconvexity and fine properties.

%While there is a rich literature on peridynamics, and the research community has intensively studied the fractional Laplacian \cite{CaSi07, MoRaSe16, BuVa16, ChLiMa20} as the major nonlocal operator over the last years, the insights of problems with nonlocal gradients are much less extensive.

In this work we examine radial kernels, which model isotropic interactions, and its aim is to ascertain the assumptions on $\rho$ that are necessary to develop a satisfactory theory for nonlocal gradients and its associated function spaces.
In particular, we establish a set of hypotheses on $\rho$ for which the main structural properties of the Riesz gradient carry over to general gradients. This lays the basis for the study of PDEs and variational problems based on nonlocal gradients. \smallskip
%With those assumptions, we show: translation property, Poincar\'e inequalities, compact embeddings into $L^p$, nonlocal fundamental theorem of calculus, embeddings into Orlicz spaces, embeddings into spaces of continuous functions, and conditions on two kernels $\rho_1$ and $\rho_2$ for which $H^{\rho_1,p} \subset H^{\rho_2,p}$ or $H^{\rho_1,p} = H^{\rho_2,p}$.
%
%In light of the fact that the Riesz transform is bounded on $L^p$, we expect that $\rho$ has to be at least as singular as $\abs{x}^{-n+1}$ at the origin.
%When radial symmetry is dropped, however, the singular behaviour of the kernel is not needed; see \cite{HaT23} for an analysis of kernels supported in a half-ball. 

We give here an overview of our main results. The first key property that we establish is an analogue of the translation method in \eqref{eq:translation}, that is, we identify a locally integrable function $Q_\rho$ such that
\begin{equation}\label{eq:rhotranslation}
D_\rho u = \nabla (Q_\rho * u) \quad \text{for $u \in C_c^{\infty}(\R^n)$},
\end{equation}
see~Proposition~\ref{pr:Qhat}. Beyond the simple representation that \eqref{eq:rhotranslation} provides, we use this formula to gain information about the operator $D_\rho$ from the Fourier perspective by studying the Fourier transform of $Q_\rho$. This enables us to prove Poincar\'{e} inequalities and compact embeddings for the spaces $H^{\rho,p}_0(\Omega)$ with $\Omega \subset \R^n$ open and bounded. Precisely, beyond technical assumptions on $\rho$, we show that if
\begin{equation}\label{eq:cond1}
\liminf_{x \to 0}\, \abs{x}^{n-1}\rho(x) >0,
\end{equation}
then there is a $C>0$ such that
\[
\norm{u}_{L^p(\O)} \leq C \norm{D_\rho u}_{L^p(\R^n,\R^n)} \quad \text{for all $u \in H^{\rho,p}_0(\O)$},
\]
whereas $H^{\rho,p}_0(\O)$ is compactly embedded into $L^p(\R^n)$ if 
\begin{equation}\label{eq:cond2}
\lim_{x \to 0}\, \abs{x}^{n-1}\rho(x) =\infty,
\end{equation}
see Theorem~\ref{th:l2bound} for the case $p=2$ and Theorem~\ref{th:lpbound} for the general case $p \in (1,\infty)$; the setting with $p \in (1,\infty)$ requires an additional smoothness assumption on $\rho$ in order to apply the Mihlin-H\"{o}rmander multiplier theorem. Remarkably, we show in Proposition~\ref{pr:Poincareconverse} that the conditions \eqref{eq:cond1} and \eqref{eq:cond2} are essentially optimal when $p=2$, which indicates that a singularity in the kernel $\rho$ is unavoidable.

Under additional assumptions on $\rho$, we can push the Fourier analysis of $D_\rho$ further and prove a nonlocal fundamental theorem of calculus as in \eqref{eq:NFTC},
\begin{equation}\label{eq:rhoNFTC}
 u = V_\rho \ast D_\rho u \quad \text{for all $u \in H^{\rho,p}_0(\O)$},  
\end{equation}
with a locally integrable function $V_\rho$ whose singularity at the origin is related to that of $\rho$, see~Theorem~\ref{th:NTFC} and Corollary~\ref{cor:ftoc}. As an application, we use \eqref{eq:rhoNFTC} to show embeddings of $H^{\rho,p}_0(\O)$ into Orlicz spaces and spaces with prescribed modulus of continuity, where the behavior of the Orlicz function and modulus of continuity are linked to the singularity of $\rho$ at the origin, cf.~Theorem~\ref{th:HLS} and Theorem~\ref{th:Morrey}. This provides a refinement and generalization of the fractional Sobolev and Morrey inequalities for $D^s$ and $D^s_\d$ \cite{ShS2015, BeCuMo22} that is not restricted to the scale of Lebesgue or H\"{o}lder spaces.

As is apparent from the previous paragraphs, not all results require the same assumptions on $\rho$.
In fact, during the development of the theory, we will increasingly impose more conditions on the kernel. 
We will see several examples of $\rho$ along the paper, but we anticipate a few of them for which the general theory holds:
\[
\frac{\chi(x)}{\abs{x}^{n+s-1}}, \qquad \frac{\chi(x) (- \log \abs{x})}{\abs{x}^{n+s-1}} , \qquad \frac{\chi(x)}{\abs{x}^{n+s-1} (- \log \abs{x})} , \qquad \frac{\chi(x)}{\abs{x}^{n+ s(\abs{x})-1}} ,
\]
with $\chi \in C_c^{\infty}(\R^n)$ a non-negative, non-zero, radial function with some weak property of being decreasing (and supported in the unit ball in the second and third examples), $0 < s < 1$ and, in the last example, $s : (0, \infty) \to (0,1)$ a smooth function.
All those kernels have compact support; we will also see that this can be assumed without loss of generality, since only the behaviour of $\rho$ near zero is relevant for the properties of the function spaces $H^{\rho,p} (\Rn)$ and $H^{\rho,p}_0(\O)$, see~Proposition~\ref{prop:carryover}.

We finish this introduction with the comment that the community has not reached a consensus on the name or the notation for the spaces related to nonlocal gradients.
In the introduction of \cite{campos2023unilateral}, there is an interesting discussion on the origin of the spaces $H^{s,p} (\Rn)$, which are commonly called \emph{Bessel potential spaces}, but for which the authors propose the name \emph{Lions-Calder\'on spaces}.
Of course, the spaces $H^{\rho,p}(\Rn)$ generalize Bessel potential spaces, which, in turn, are nonlocal versions of Sobolev spaces.
Actually, in the case of Bessel potential spaces, those are obtained through complex interpolation between Sobolev $W^{1,p}$ and Lebesgue $L^p$ spaces. Hence, an appropriate name for $H^{s,p}(\Rn)$ could be {\it fractional Sobolev spaces}, but this name is commonly reserved for the \emph{Gagliardo} (or \emph{Sobolev-Slobodeckij}) spaces $W^{s,p}(\Rn)$.
In contrast to $H^{s,p}(\Rn)$, Gagliardo spaces are obtained by real interpolation between Sobolev and Lebesgue spaces.
From our point of view, $H^{s,p}(\Rn)$ deserves the name of fractional Sobolev space in more right than $W^{s,p}(\Rn)$, as for the former there is a fractional differential object (the Riesz fractional gradient), whereas for the latter there is just a fractional seminorm.
Although $\rho$-\emph{Bessel space} could be a sensible name for $H^{\rho,p}(\Rn)$, we propose to call it \emph{$\rho$-nonlocal Sobolev space}, since there is no immediate connection with the Bessel potentials. 

The outline of this article is as follows.
In Section \ref{se:first}, we establish the basic assumptions on $\rho$ and properties of $D_{\rho}$ including the translation method of \eqref{eq:rhotranslation}. An analysis of the associated spaces $H^{\rho,p} (\Rn)$ and $H^{\rho,p}_0(\O)$ is performed in Section \ref{se:function}, providing a distributional definition of nonlocal gradients, a  Leibniz rule and density results. In addition, Proposition \ref{prop:carryover} establishes a simple sufficient condition for the equality of spaces associated to different kernels, and for carrying over Poincar\'e inequality from one gradient to the other. Section~\ref{sec:poinc} is devoted to the proof of the Poincar\'{e} inequalities and compact embeddings under the assumptions \eqref{eq:cond1} and \eqref{eq:cond2}. The nonlocal fundamental theorem of calculus as in \eqref{eq:rhoNFTC} is proven in Section~\ref{se:fundamental}, and the subsequent embeddings into Orlicz spaces and spaces of continuous function are presented in Section~\ref{se:embedding}. Finally, in Section~\ref{se:inclusion}, we establish conditions for inclusions and equality of spaces associated to different kernels, and show that the conditions in \eqref{eq:cond1} and \eqref{eq:cond2} are almost optimal in order to have Poincar\'{e} inequalities and compact embeddings in $L^2$, respectively.

\section*{Notation}%\label{se:notation}

We fix the dimension $n \in \N$ of the space and an open set $\O \subset \Rn$. The vectors of the canonical basis of $\Rn$ are $e_j$, $j = 1, \ldots, n$. The characteristic function of $A \subset \Rn$ is denoted by $\mathbbm{1}_A$. The complement of a subset $A$ in $\Rn$ is denoted by $A^c$, its closure by $\bar{A}$ and its boundary by $\p A$.
We write $B_r (x)$ for the open ball centred at $x \in \Rn$ of radius $r >0$.
We also set $B_r = B_r (0)$, $\Sn = \p B_1$ and $\Sn_+ = \{ z \in \Sn : z_1 > 0 \}$.
The surface area in integrals is indicated by $\mathcal{H}^{n-1}$, while we set $\s_{n-1} = \mathcal{H}^{n-1} (\Sn)$.

We will use an exponent $p \in [1, \infty]$ of integrability; its conjugate exponent $p' = \frac{p}{p-1}$.
The notation for Lebesgue $L^p$ and Sobolev $W^{1,p}$ spaces is standard.
So is the notation for functions that are of class $C^k$, for $k$ an integer or infinity.
Their version of compact support are $C^k_c$.
The support of a function is indicated by $\supp$. 
We will indicate the domain and target of the functions, as in $L^p (\O, \Rn)$.
The target is omitted if it is $\R$. We will use the abbreviation \emph{a.e.} for \emph{almost everywhere} or \emph{almost every}. For $\a \in \N^n$, we give the standard meaning to the partial derivative $\p^{\a}$ and the size $|\a|$; see \cite[Sect.\ 2.2]{Gra14a}. The operation of convolution is denoted by $*$.
%We indicate the duality product between tempered distributions and Schwartz functions as $\langle \cdot , \cdot \rangle$.

The convention for the Fourier transform of a function $f \in L^1 (\Rn)$ is
\begin{equation*}
\widehat{f}(\xi)=\int_{\Rn} f(x) \, e^{-2\pi i x \cdot \xi} \, dx .
\end{equation*}
This definition is extended by continuity and duality to other function and distribution spaces.
The Schwartz space is denoted by $\mathcal{S}$ and the space of tempered distributions by $\mathcal{S}'$.
The variable in the Fourier space is generically taken to be $\xi$.
The inverse Fourier transform is denoted by $f^{\vee}$.

A function $f:\Rn \rightarrow \R$ is \emph{radial} if there exists $\bar{f}:[0,\infty) \to  \R$ (the \emph{radial representation} of $f$) such that $f(x)=\bar{f}(|x|)$ for every $x \in \Rn$.
A radial function $f:\Rn \rightarrow \R$ is \emph{radially decreasing} if its radial representation is decreasing.
A function $\phi:\Rn \rightarrow \Rn$ is \emph{vector radial} if there exists $\bar{\phi}:[0,\infty)\to  \R$ such that $\phi(x)=\bar{\phi}(|x|)x$ for every $x \in \Rn$.
	
%It is known (see, e.g., \cite[App.\ B.5]{Grafakos08a}) that the Fourier transform of a radial (respectively, vector radial) function is radial (respectively, vector radial).

The words \emph{increasing} and \emph{decreasing} are meant in their wide (not strict) sense.
In contrast, we use \emph{positive} and \emph{negative} with their strict meaning.

From Section~\ref{se:fundamental}, we will use the notion of \emph{almost increasing} and \emph{almost decreasing}.
A function $f : I \to \Rn$ is almost increasing in the interval $I$ if there exists a $C>0$ such that for any $x_1, x_2 \in I$ with $x_1 \leq x_2$, we have $f(x_1) \leq C f (x_2)$.
An analogous definition is given for almost decreasing. We will denote by $C, C_k\ldots$ generic positive constants, which may vary from line to line.

For convenience of the reader, we collect here the assumptions made on the radial kernel $\rho:\Rn \to [0,\infty]$ that appear throughout the paper:
\begin{enumerate}[label = (H\arabic*)]
 
\item  The function $f_\rho:(0,\infty)\to \R, \ t \mapsto t^{n-2}\overline{\rho}(t)$ is decreasing, and there is a $0<\mu<1$ such that $\mu f_\rho(t/2) \geq  f_\rho(t)$ for $t \in (0,\epsilon)$;

\item  the function $f_\rho$ is smooth in $(0,\infty)$, and for $t \in (0,\epsilon)$,
\[
-C \frac{d}{dt}f_\rho(t) \geq \frac{f_\rho(t)}{t} \quad \text{and} \quad \abslr{\frac{d^k}{dt^k}f_\rho(t)} \leq C_k \frac{f_\rho(t)}{t^k} \quad \text{for $k \in \N$};
\]

\item  the function $g_\rho:(0,\infty) \to \R, \ t \mapsto t^{n+\sigma-1}\overline{\rho}(t)$ is almost decreasing on $(0,\epsilon)$ for some $\sigma \in (0,1)$; 

\item  the function $h_\rho:(0,\infty) \to \R, \ t \mapsto t^{n+\gamma-1}\overline{\rho}(t)$ is almost increasing on $(0,\epsilon)$ for some $\gamma \in (0,1)$.
\end{enumerate}

\section{First properties of $\mathcal{G}_{\rho}$}\label{se:first}

In this section we show some principal properties of the $\rho$-derivative $\mathcal{G}_{\rho} u$ and its Fourier transform for $u \in C^{\infty}_c (\Rn)$.

We always make the following basic assumptions on $\rho$:
\begin{enumerate}[label = (H\arabic*)]
\setcounter{enumi}{-1}
\item \label{itm:h0} $\begin{cases} \text{$\rho:\R^n \to [0,\infty]$ is radial with $\rho (x) \in \R$ for all $x \in \Rn \setminus \{ 0 \}$,} \\
\text{$\rho \in L^1_{\loc}(\R^n)$ with $\displaystyle \int_{\R^n} \min\{1,\abs{x}^{-1}\}\rho(x)\,dx < \infty$,} \\
\text{$\displaystyle \inf_{B_{\epsilon}}\rho >0$ for some $\epsilon>0$.} \end{cases}$
\end{enumerate}
Similar assumptions to \ref{itm:h0} have appeared in the literature; e.g., \cite{HaT23}.
Note that we consider $\rho$ as a function everywhere defined in $\Rn$, and not as an equivalence class of functions defined a.e.

Clearly, \ref{itm:h0} implies that
\begin{equation}\label{eq:H0cor}
 \int_{B_r} \rho (x) \, dx + \int_{B_r^c} \frac{\rho (x)}{\abs{x}} \, dx < \infty , \qquad r > 0 ,
\end{equation}
which, in terms of the radial representation $\overline{\rho}$, can be written as
\begin{equation}\label{eq:H0corbarrho}
 \int_0^r \overline{\rho} (t) t^{n-1} \, dt + \int_r^{\infty} \overline{\rho} (t) t^{n-2} \, dt < \infty , \qquad r > 0 .
\end{equation}
In fact, under the assumption \ref{itm:h0}, we have the equivalence
\begin{equation}
\begin{split}\label{eq:integrabilityQL1}
 \rho \in L^1 (\Rn) &\iff \int_r^{\infty} \overline{\rho} (t) t^{n-1} \, dt < \infty \text{ for some } r > 0\\
 & \iff \int_r^{\infty} \overline{\rho} (t) t^{n-1} \, dt < \infty \text{ for all } r > 0 .
 \end{split}
\end{equation}

\begin{example}\label{ex:H0}
Classes of kernels $\rho$ satisfying \ref{itm:h0} are:
\begin{enumerate}[label=(\alph*)]
\item\label{item:Riesz} Given $s \in (0,1)$,
\begin{equation}\label{eq:Rieszkernel}
 \rho (x) = \frac{1}{\abs{x}^{n+s-1}} .
\end{equation}

\item Given a continuous $s: [0, \infty) \to (0,1)$ with $\inf_{[0,\infty)} s>0$,
\[
 \rho (x) = \frac{1}{\abs{x}^{n + s(\abs{x}) - 1}} .
\]

\item Given $0 \leq \alpha < n$ and $\beta > n-1$,
\[
 \rho (x) = \frac{\mathbbm{1}_{B_1}(x)}{\abs{x}^{\alpha}} +  \frac{\mathbbm{1}_{B_1^c}(x)}{\abs{x}^{\beta}} .
\]

%\item Given $\alpha >1$,
%\[
% \rho (x) = \frac{\mathbbm{1}_{B_{1/2}}(x)}{\abs{x}^n (-\log |x|)^{\alpha}}  .
%\]

\item If $\rho$ satisfies \ref{itm:h0} and $\chi \in L^{\infty} (\Rn)$ is radial with $\chi \geq 0$ and $\inf_{B_{\e}} \chi > 0$ then $\chi \rho$ satisfies \ref{itm:h0}.

\item If $\rho_1, \rho_2$ satisfy \ref{itm:h0} and $\alpha_1, \alpha_2>0$ then $\alpha_1 \rho_1 + \alpha_2 \rho_2$ satisfies \ref{itm:h0}.

\item If $\rho_1, \rho_2$ satisfy \ref{itm:h0} then any measurable radial $\rho$ with $\rho_1 \leq \rho \leq \rho_2$ satisfies \ref{itm:h0}.

\end{enumerate}
\end{example}

\begin{definition} \label{def:nl gradient}
For $u\in C_c^\infty(\Rn)$, we define the nonlocal gradient of $u$ as
\begin{equation}\label{eq:Grux}
 \mathcal{G}_\rho u(x)=\int_{\R^n} \frac{u(x)-u(y)}{\abs{x-y}}\frac{x-y}{\abs{x-y}}\rho(x-y)\,dy , \qquad x \in \Rn .
\end{equation}
\end{definition}

The following result shows that the nonlocal gradient defines an integrable and bounded function, which can be deduced from the more general statement in~\cite[Prop.\ 1]{EGM22}, but we provide the details for the reader's convenience.

\begin{lemma}\label{le:GuL1Linfty}
Let $u \in C^{\infty}_c (\Rn)$.
Then $\mathcal{G}_{\rho} u \in L^1 (\Rn,\Rn) \cap L^{\infty} (\Rn,\Rn)$ and the integral \eqref{eq:Grux} is absolutely convergent for each $x \in \Rn$.
\end{lemma}
\begin{proof}
Let $L>0$ be a Lipschitz constant of $u$, then we can bound
\begin{align*}
 \int_{\R^n} \left| \frac{u(x)-u(y)}{\abs{x-y}}\frac{x-y}{\abs{x-y}}\rho(x-y) \right| dy &\leq \max\{L,2\norm{u}_\infty\}\int_{\R^n} \min\{1,\abs{x-y}^{-1}\}\rho(x-y)\,dy\\
 &= \max\{L,2\norm{u}_\infty\} \int_{\R^n} \min\{1,\abs{z}^{-1}\}\rho(z)\,dz.
\end{align*}
Thus, $\mathcal{G}_{\rho} u$ is bounded thanks to \ref{itm:h0}.

Next, let $K = \supp u$, $\d>0$, and define $K_{\d} := K + B_{\d}$. In order to prove that $\mathcal{G}_{\rho} u \in L^1 (\Rn,\Rn)$, it is enough to show that $\mathcal{G}_{\rho} u \in L^1 (K_{\d}^c,\Rn)$. For $x \notin K_{\d}$ we have
\[
 \abs{\mathcal{G}_{\rho} u (x)} \, dx \leq \int_K \frac{\abs{u(y)}}{\abs{x-y}}\rho(x-y) \, dy ,
\]
so
\[
 \int_{K_{\d}^c} \abs{\mathcal{G}_{\rho} u (x)} \, dx \leq \int_K \abs{u(y)} \int_{K_{\d}^c} \frac{1}{\abs{x-y}}\rho(x-y) \, dx \, dy \leq \norm{u}_{1} \int_{B_{\d}^c} \frac{\rho (z)}{\abs{z}} \, dz,
\]
which is finite thanks to \eqref{eq:H0cor}.
\end{proof}

Define $\lambda_\rho : \Rn \to \C^n$ as
\[
 \lambda_\rho(\xi) := \int_{\R^n} \frac{\rho(x)x}{\abs{x}^2}(e^{2\pi i \xi \cdot x}-1)\,dx .
\]
Similarly as in \cite[Lemma~2.13]{HaT23} (see also \cite[Lemma~1.2]{LeDu20}), the following result shows that $\lambda_\rho$ is the Fourier multiplier associated to the operator $\Gcal_\rho$.

\begin{lemma}\label{le:lamda}
If $u \in C^{\infty}_c (\Rn)$ then
\begin{equation}\label{eq:Gmultiplier}
 \widehat{\Gcal_\rho u} = \lambda_{\rho} \widehat{u}  .
\end{equation}
Moreover, 
\[
\lambda_\rho(\xi) = \frac{i \xi}{\abs{\xi}}\int_{\R^n}\frac{\rho(x)x_1}{\abs{x}^2}\sin(2\pi\abs{\xi}x_1)\,dx , \qquad \xi \in \Rn \setminus \{ 0 \}
\]
and $\abs{\lambda_\rho(\xi)}<\infty$ for all $\xi \in \Rn$.
\end{lemma}
\begin{proof}
Thanks to Lemma \ref{le:GuL1Linfty}, $\Gcal_\rho u \in L^1 (\R^n,\Rn)$ and the following calculation is valid for each $\xi \in \Rn$:
\begin{align*}
 \widehat{\Gcal_\rho u} (\xi) & = \int_{\Rn} \int_{\Rn} \frac{u(x)-u(y)}{\abs{x-y}}\frac{x-y}{\abs{x-y}}\rho(x-y) e^{-2\pi i x \cdot \xi} \, dy \, dx \\
 & = - \int_{\Rn} \int_{\Rn} \frac{u(y - z) - u(y)}{\abs{z}}\frac{z}{\abs{z}}\rho(z) e^{-2\pi i (y - z) \cdot \xi}  \, dy \, dz \\
 & = - \int_{\Rn} \left[ \int_{\Rn} u(y - z)  e^{-2\pi i (y - z) \cdot \xi} \, dy - e^{2\pi i z \cdot \xi} \int_{\Rn} u(y) e^{-2\pi i y \cdot \xi} \, dy \right] \frac{z}{\abs{z}^2} \rho(z) \, dz \\
  & = - \int_{\Rn} \left[  \widehat{u}(\xi) - e^{2\pi i z \cdot \xi} \widehat{u}(\xi) \right] \frac{z}{\abs{z}^2} \rho(z) \, dz \\
  & = \widehat{u}(\xi) \int_{\Rn} \left( e^{2\pi i z \cdot \xi} - 1 \right) \frac{z}{\abs{z}^2} \rho(z) \, dz ,
\end{align*}
so \eqref{eq:Gmultiplier} is proved.
The alternative expression for $\lambda_{\rho}$ is obtained as follows.
The integral 
\[
  \int_{\R^n} \frac{\rho(x)x}{\abs{x}^2}( \cos( 2\pi \xi \cdot x)-1)\,dx
\]
is zero since it is absolutely convergent with an odd integrand.
Therefore,
\[
 \lambda_\rho(\xi) = i \int_{\R^n}\frac{\rho(x)x}{\abs{x}^2}\sin(2\pi \xi \cdot x)\,dx.
\]
Now, as $\rho$ is radial, if $R \in \mathrm{SO}(n)$ is a rotation then
\[
 \lambda_\rho(R \xi) = i \int_{\R^n}\frac{\rho(x)x}{\abs{x}^2}\sin(2\pi R \xi \cdot x)\,dx = i \int_{\R^n} \frac{\rho(x)x}{\abs{x}^2}\sin(2\pi \xi \cdot R^T x)\,dx =R \lambda_\rho(\xi) .
\]
Now, given $\xi \in \Rn$, choose $R \in \mathrm{SO}(n)$ such that $R \xi = \abs{\xi} e_1$.
Then
\begin{align*}
 \lambda_\rho(\xi) & = R^T \lambda_\rho(\abs{\xi} e_1) = i R^T \int_{\R^n}\frac{\rho(x)x}{\abs{x}^2}\sin(2\pi \abs{\xi} x_1) \,dx = i R^T \int_{\R^n}\frac{\rho(x)x_1 e_1 }{\abs{x}^2}\sin(2\pi \abs{\xi} x_1) \,dx \\
 & = \frac{i \xi}{\abs{\xi}} \int_{\R^n}\frac{\rho(x)x_1}{\abs{x}^2}\sin(2\pi\abs{\xi}x_1)\,dx ,
\end{align*}
as desired.

Finally, using the bound $\abs{\sin t} \leq \min \{ 1, \abs{t}\}$ ($t \in \R$), we find that
\[
 \abs{\lambda_\rho(\xi)} \leq \int_{\R^n} \frac{\rho(x)}{\abs{x}} \abs{\sin(2\pi\abs{\xi}x_1)} \,dx \leq \int_{\R^n} \rho(x) \min \{ \abs{x}^{-1}, 2\pi\abs{\xi} \} \,dx < \infty ,
\]
in light of \eqref{eq:H0cor}.
\end{proof}

Now we employ a strategy described in \cite[Prop.\ 4.3]{BeCuMo22} (itself based on \cite[Lemma 15.9]{Ponce16}) consisting of studying a potential of $x\mapsto - \frac{\overline{\rho}(\abs{x})}{\abs{x}} \frac{x}{\abs{x}}$.
For this, we define the function
\[
 Q_{\rho}(x) := \int_{\abs{x}}^\infty \frac{\overline{\rho}(t)}{t}\,dt , \qquad x \in \Rn \setminus \{ 0 \},
\]
which is well defined and finite due to \eqref{eq:H0corbarrho}. We present some of its immediate properties.

\begin{lemma}\label{le:Q}
The following statements hold:
\begin{enumerate}[label = (\roman*)]
\item\label{item:Q1} For each $0 < a < b$ we have
\[
 Q_{\rho} \in W^{1,1} (B_b \setminus B_a) , \quad Q_{\rho} \in L^1 (B_b) \quad \text{and} \quad \nabla Q_{\rho} \in L^1 (B_a^c,\R^n),
\]
and for a.e.\ $x \in \Rn \setminus\{ 0 \}$,
\[
\nabla Q_{\rho}(x) = - \frac{\overline{\rho}(\abs{x})}{\abs{x}} \frac{x}{\abs{x}}.
\]

\item\label{item:Q2}
For every $M>0$ there is a $C=C(n,M)>0$ such that
\[
 \abs{Q_\rho(x)} \leq \frac{C}{\abs{x}^{n-1}} \quad \text{for $\abs{x} \geq M$}.
\]

\item\label{item:Q3}
If $\rho \in L^1 (\Rn)$ then $Q_{\rho} \in L^1 (\Rn)$.
Moreover, when $\rho$ has compact support then $Q_\rho$ lies in $L^1(\R^n)$ and also has compact support.

\end{enumerate}
\end{lemma}
\begin{proof}
Part \ref{item:Q1}.
%We may compute that
%\begin{align*}
%\int_{B_1(0)} Q_\rho(x)\,dx &= \int_{B_1(0)}\int_{\abs{x}}^\infty \frac{\overline{\rho}(t)}{t}\,dt\,dx = \int_{0}^\infty \int_{B_1(0) \cap B_t(0)} \frac{\overline{\rho}(t)}{t}\,dx\,dt\\
%& =\omega_n\int_0^\infty \min\{t^n,1\}\frac{\overline{\rho}(t)}{t}\,dt =\int_{\R^n}\min\{\abs{x}^n,1\}\frac{\rho(x)}{\abs{x}^n}\,dx < \infty,
%\end{align*}
%where the last inequality uses \ref{itm:h0}.
Let $0 < a < 1$.
Then
\begin{equation}\label{eq:intrhot}
 \int_a^1 \frac{\overline{\rho}(t)}{t}\,dt + \int_1^{\infty} \frac{\overline{\rho}(t)}{t}\,dt \leq \frac{1}{a^n} \int_a^1 \overline{\rho}(t) t^{n-1} \,dt + \int_1^{\infty} \overline{\rho}(t) t^{n-2} \,dt < \infty ,
\end{equation}
in view of \eqref{eq:H0corbarrho}.
Consequently, the radial representation $\overline{Q}_{\rho}$ of $Q_{\rho}$ is locally Sobolev in $(0, \infty)$ (see, e.g., \cite[Thms.\ 3.29 and 7.16]{Leoni09}). 
An argument similar to \cite[Lemma 4.1]{Ball82} shows that $Q_{\rho}$ is locally Sobolev in $\Rn \setminus \{ 0\}$ and its distributional derivative coincides with its classical derivative a.e.

A straightforward calculation based on the coarea formula and Fubini's theorem shows that
\[
 \int_{B_b} Q_{\rho} (x) \, dx = \frac{\s_{n-1}}{n} \left[ \int_0^b \overline{\rho}(t) t^{n-1} \, dt + b^n \int_b^{\infty} \frac{\overline{\rho}(t)}{t} \, dt \right] < \infty ,
\]
thanks to \eqref{eq:H0corbarrho} and \eqref{eq:intrhot}.

As $\overline{Q}_{\rho}$ is absolutely continuous,
\begin{equation*}\label{eq:Qprime}
 \overline{Q}'_{\rho} (t) = - \frac{\overline{\rho}(t)}{t} , \qquad \text{for a.e. } t > 0 ,
\end{equation*}
so
\begin{equation*}\label{eq:nablaQ}
 \nabla Q_{\rho} (x) = \overline{Q}'_{\rho} (\abs{x}) \frac{x}{\abs{x}} , \qquad \text{for a.e. } x \in \Rn \setminus \{ 0 \} .
\end{equation*}
With this we clearly have $\nabla Q_{\rho} \in L^1 (B_a^c,\Rn)$, due to \eqref{eq:H0cor}. \smallskip

Part \ref{item:Q2}. For any $t \geq M>0$ we have
\begin{equation*}
 \overline{Q}_\rho (t) = \int_{t}^\infty \frac{\overline{\rho}(r)}{r}\,dr \leq \frac{1}{t^{n-1}} \int_{t}^\infty \overline{\rho}(r) r^{n-2}\,dr \leq \frac{1}{t^{n-1}} \int_{M}^\infty \overline{\rho}(r) r^{n-2}\,dr 
\end{equation*}
and \ref{item:Q2} is concluded thanks to \eqref{eq:H0corbarrho}. \smallskip

Part \ref{item:Q3}.
Assume $\rho \in L^1 (\Rn)$.
In order to show that $Q_{\rho} \in L^1 (\Rn)$ it is enough to check that $Q_{\rho} \in L^1 (B_b^c)$, due to \ref{item:Q1}.
A straightforward calculation shows that
\[
 \int_{B_b^c} Q_{\rho} (x) \, dx = \s_{n-1} \int_b^{\infty} \frac{\overline{\rho} (r)}{r} \frac{r^n - b^n}{n} \, dr \leq \s_{n-1} \int_b^{\infty} \overline{\rho} (r) r^{n-1} \, dr < \infty ,
\]
in view of \eqref{eq:integrabilityQL1}.
If, in addition, $\rho$ has compact support then so does $Q_\rho$, and, hence, $Q_{\rho} \in L^1(\R^n)$.
\end{proof}

%In fact, the same argument of part \ref{item:Q3} of Lemma \ref{le:Q} shows that $\rho \in L^1 (\Rn)$ is equivalent to $Q_{\rho} \in L^1 (\Rn)$.

As a consequence of part \ref{item:Q3} of Lemma \ref{le:Q}, if $\rho \in L^1 (\Rn)$ then $\widehat{Q}_\rho$ is a continuous function, which is analytic if $\rho$ has compact support.
In fact, in the development of the theory, we will often assume that $\supp \rho = \overline{B_\delta}$ for some $\d>0$.
In this case, for an open set $\O \subset \Rn$ we define $\Omega_\delta=\Omega+B_\d$, and note that $\Gcal_\rho \phi$ is supported in $\Omega_\d$ for $\phi \in C_c^{\infty}(\Omega)$.

We now show that the nonlocal gradient can be written as the convolution of $Q_\rho$ with the classical gradient, and derive a formula for $\widehat{Q}_{\rho}$.

\begin{proposition}\label{pr:Qhat}
The following two statements hold:
\begin{enumerate}[label = (\roman*)]
\item\label{item:Qhat1}
For $u \in C^{\infty}_c (\Rn)$, we have that $\mathcal{G}_\rho u \in C^{\infty} (\Rn, \Rn)$ and 
\begin{equation*}
\mathcal{G}_\rho u = Q_\rho * \nabla u = \nabla (Q_\rho * u).
\end{equation*}

\item\label{item:Qhat2} If, in addition, $\rho \in L^1 (\Rn)$, then
\[
 \widehat{\Gcal_\rho u} (\xi) = 2\pi i \xi \widehat{Q}_\rho(\xi)\widehat{u}(\xi) \quad \text{and} \quad \lambda_\rho(\xi) = 2\pi i \xi \widehat{Q}_{\rho}(\xi) , \qquad \xi \in \Rn
\]
and 
\begin{equation}\label{eq:Qhatrho}
 \widehat{Q}_\rho(\xi) = \frac{1}{2\pi\abs{\xi}}\int_{\R^n}\frac{\rho(x)x_1}{\abs{x}^2}\sin(2\pi\abs{\xi}x_1)\,dx , \qquad \xi \in \Rn \setminus \{ 0 \} .
\end{equation}

\end{enumerate}
\end{proposition}
\begin{proof}
Part \ref{item:Qhat1}.
We consider $x, e \in \Rn$ with $|e| = 1$ and the vector field
	\[
	\beta : \Rn \setminus \{ x\} \rightarrow \Rn , \qquad \beta(y)=(u(x)-u(y)) \, Q_\rho (x-y) \, e.
	\] 
For any $0 < a <b$, by Lemma \ref{le:Q}, $\beta \in W^{1,1} (B_b (x) \setminus B_a (x),\R^n)$.
By the divergence theorem (e.g., \cite[Th.\ 18.1]{Leoni09})
\begin{equation}
\begin{split}\label{eq:divergencebeta}
  \int_{B_b (x) \setminus B_a (x)} \diver \beta (y) \, dy &= \int_{\partial B_b (x)} \beta (y) \cdot \nu_{B_b (x)} (y) \, d \mathcal{H}^{n-1} (y) \\
  &\qquad - \int_{\partial B_a (x)} \beta (y) \cdot \nu_{B_b (x)} (y) \, d \mathcal{H}^{n-1} (y) ,
 \end{split}
\end{equation}
where $\nu_{B_r (x)}$ is the exterior normal to $B_r (x)$, for $r=a,b$.
In fact,
	\begin{equation*} \label{eq: divergence of the integrand}
	\diver \beta(y) = - Q_\rho (x-y) \, \nabla u(y) \cdot e  - (u(x) - u(y)) \nabla Q_\rho (x-y) \cdot e , \qquad \text{a.e. } y \in \Rn .
	\end{equation*}
It turns out that both terms of the right-hand side of the formula above are in $L^1 (\Rn)$.
Indeed, by Lemma \ref{le:Q} and the fact that $u$ has compact support, we have that the map $y \mapsto Q_\rho (x-y) \, \nabla u(y)$ is in $L^1 (\Rn,\Rn)$.
Analogously, the map $y \mapsto (u(x) - u(y)) \nabla Q_\rho (x-y)$ is in $L^1 (B_a (x)^c,\Rn)$, and, as $u$ is Lipschitz, we have
\begin{align*}
 \int_{B_a (x)} \abs{(u(x) - u(y)) \nabla Q_\rho (x-y)} \, dy & \leq \norm{\nabla u}_{\infty} \int_{B_a (x)} \abs{x - y} \abs{\nabla Q_\rho (x-y)} \, dy \\
 & = \norm{\nabla u}_{\infty} \s_{n-1} \int_0^a t^{n-1 }\overline{\rho} (t) \, dt < \infty ,
\end{align*}
in view of Lemma~\ref{le:Q} and \eqref{eq:H0corbarrho}.
In particular,
\begin{equation}\label{eq:limab}
\begin{split}
 \lim_{\substack{a \searrow 0 \\ b \to \infty}} \int_{B_b (x) \setminus B_a (x)} \diver \beta (y) \, dy &= - \int_{\Rn} Q_\rho (x-y) \, \nabla u(y) \cdot e \, dy \\
 &\qquad- \int_{\Rn} (u(x) - u(y)) \nabla Q_\rho (x-y) \cdot e \, dy .
 \end{split}
\end{equation}
By Lemma \ref{le:Q}, $Q_{\rho} \in L^1 (B_r)$ for all $r>0$, so
\[
 \int_0^r \overline{Q}_{\rho} (t) t^{n-1} \, dt < \infty,
\]
which implies that $\liminf_{a \downarrow 0} a^n \overline{Q}_{\rho} (a) = 0$.
Let $\{ a_j \}_{j \in \N}$ be a sequence of positive numbers tending to zero such that $\lim_{j \to \infty} a_j^n \overline{Q}_{\rho} (a_j) = 0$.
As $u$ is Lipschitz,
\begin{align*}
 \left|\int_{\partial B_{a_j} (x)} \beta \cdot \nu_{B_{a_j} (x)} \, d \mathcal{H}^{n-1} \right| \leq \int_{\partial B_{a_j} (x)} \abs{\beta} \, d \mathcal{H}^{n-1} \leq \left\| \nabla u \right\|_{\infty} \s_{n-1} a_j^n \overline{Q}_{\rho} (a_j) \to 0 \quad \text{as } j \to \infty ,
\end{align*}
and, as $u$ has compact support, if $b$ is big enough,
\[
 \int_{\partial B_b (x)} \beta \cdot \nu_{B_b (x)} \, d \mathcal{H}^{n-1} =  u(x) \int_{\partial B_b (x)} Q_\rho (x-y) \, e \cdot \nu_{B_b (x)} \, d \mathcal{H}^{n-1} (y) = 0
\]
by symmetry.  
Together with \eqref{eq:divergencebeta} and \eqref{eq:limab}, this yields
\[
 \int_{\Rn} Q_\rho (x-y) \, \nabla u(y) \cdot e \, dy = - \int_{\Rn} (u(x) - u(y)) \nabla Q_\rho (x-y) \cdot e \, dy .
\]
As this is true for every $e \in \Rn$ with $|e|=1$, we conclude that
\[
 \int_{\Rn} Q_\rho (x-y) \, \nabla u(y) \, dy = - \int_{\Rn} (u(x) - u(y)) \nabla Q_\rho (x-y) \, dy .
\]
In light of Lemma~\ref{le:Q} \ref{item:Q1}, this equality shows that $Q_\rho * \nabla u (x) = \mathcal{G}_\rho u (x)$.
Naturally, we also have $Q_\rho * \nabla u = \nabla (Q_\rho * u)$, since $u \in C^{\infty}_c (\Rn)$ and $Q_{\rho} \in L^1_{\loc} (\Rn)$.
In particular, $\mathcal{G}_\rho u \in C^{\infty} (\Rn, \Rn)$. \smallskip

Part \ref{item:Qhat2}.
If $\rho \in L^1 (\Rn)$ we have $Q_{\rho} \in L^1 (\Rn)$ thanks to Lemma \ref{le:Q}.
Consequently, taking Fourier transforms in the expression $\mathcal{G}_\rho u  = Q_\rho * \nabla u$, we conclude that $\widehat{\Gcal_\rho u} (\xi) = 2\pi i \xi \widehat{Q}_\rho(\xi)\widehat{u}(\xi)$ for all $\xi \in \Rn$.
Comparing this expression with that of \eqref{eq:Gmultiplier}, we obtain the correspondence $\lambda_\rho(\xi) = 2\pi i \xi \widehat{Q}_{\rho}(\xi)$ and, hence, the equality
\[
 \widehat{Q}_\rho(\xi) = \frac{-i\xi}{2\pi\abs{\xi}^2}\cdot\lambda_\rho(\xi)= \frac{1}{2\pi\abs{\xi}}\int_{\R^n}\frac{\rho(x)x_1}{\abs{x}^2}\sin(2\pi\abs{\xi}x_1)\,dx , \qquad \xi \in \Rn \setminus \{ 0 \}
\]
holds thanks to Lemma \ref{le:lamda}.
\end{proof}
\begin{remark}\label{rem:QrhoBessel}
If $Q_\rho \in L^1(\R^n)$ and $\xi \in \R^n\setminus\{0\}$, we can use the formula for the Fourier transform of a radial function, see~\cite[Appendix~B.5]{Gra14a}, to find the following alternative expression
\begin{align}\label{eq:Qrhoalternative}
\begin{split}
\widehat{Q}_\rho(\xi) &=\frac{2\pi}{\abs{\xi}^{\frac{n-2}{2}}}\int_0^\infty \int_r^\infty \frac{\overline{\rho}(t)}{t}\,dt \,J_{\frac{n}{2}-1}(2\pi \abs{\xi}r)r^{\frac{n}{2}}\,dr\\
&=\frac{2\pi}{\abs{\xi}^{\frac{n-2}{2}}}\int_0^\infty \frac{\overline{\rho}(t)}{t} \int_0^t\,J_{\frac{n}{2}-1}(2\pi \abs{\xi}r)r^{\frac{n}{2}}\,dr\,dt= \frac{1}{\abs{\xi}^{\frac{n}{2}}}\int_0^\infty \overline{\rho}(t)t^{\frac{n}{2}-1}J_{\frac{n}{2}}(2\pi\abs{\xi}t)\,dt,
\end{split}
\end{align}
with $J_{\nu}$ for $\nu>0$ the Bessel function of the first kind. In the last line we used
\[
\int_0^t J_{\frac{n}{2}-1}(2\pi \abs{\xi}r)r^{\frac{n}{2}}\,dr = \frac{t^{\frac{n}{2}}}{2\pi\abs{\xi}}J_{\frac{n}{2}}(2\pi t\abs{\xi}),
\]
which follows from the identity in \cite[Appendix~B.3]{Gra14a}. The integral in \eqref{eq:Qrhoalternative} also appears in \cite{Arr23} through different methods.
\end{remark}

Part \ref{item:Qhat2} of Proposition \ref{pr:Qhat} formally shows that, for the nonlocal gradient $\Gcal_\rho$, the fundamental theorem of calculus in Fourier space looks like
\begin{equation}\label{eq:ftocfourier}
\widehat{u}(\xi) = \frac{-i\xi}{2\pi\abs{\xi}^2\widehat{Q}_\rho(\xi)}\cdot\widehat{\Gcal_\rho u}(\xi), \qquad \xi \in \Rn \setminus \{ 0 \},
\end{equation}
which motivates the further study of the Fourier transform of $Q_\rho$. This will be carried out in Section \ref{sec:poinc}.

\section{Function spaces}\label{se:function}

In this section we establish the definition and first properties of the spaces $H^{\rho,p}(\R^n)$ and $H^{\rho,p}_0(\Omega)$, including density results. Then, we show a sufficient condition for the equivalence of spaces associated with different kernels.

\subsection{Definition and first properties}
The aim of this section is to introduce the space of $L^p$-functions whose nonlocal gradient is an $L^p$-function in analogy to Sobolev spaces, and derive some of its general properties. We start by defining the nonlocal divergence.

\begin{definition}\label{def:nl diver} For a vector field $\Phi \in C^{\infty}_c (\Rn, \Rn)$, we define its nonlocal divergence as
\[
 \diver_{\rho} \Phi (x) = \int_{\R^n} \frac{\Phi(x) - \Phi(y)}{\abs{x-y}} \cdot \frac{x-y}{\abs{x-y}}\rho(x-y)\,dy , \qquad x \in \Rn .
\]
\end{definition}

The following proposition states that the nonlocal divergence is the dual operator of the nonlocal gradient $\Gcal_\rho$ in the sense of integration by parts.
Different proofs of analogous results have appeared in the literature, for instance \cite{BeCuMC, COMI2019, MeS, Silhavy20}.
We do not include the proof here and refer the interested reader to any of those references. 

\begin{proposition}\label{prop: IBP I}
Let $u \in L^1(\Rn)+L^{\infty}(\Rn)$ be such that
\begin{equation}\label{eq: hyp IBP}
\int_K\int_{\Rn} \frac{|u(x)-u(y)|}{|x-y|} \rho(x-y)\,dx\dy< \infty
\end{equation}
for any compact set $K\subset \Rn$. Then, for any $\Phi \in C^{\infty}_c (\Rn, \Rn)$, the integration by parts formula
\[
 \int_{\Rn} \mathcal{G}_{\rho} u(x) \cdot \Phi (x) \, dx = - \int_{\Rn} u(x) \diver_{\rho} \Phi (x) \, dx 
\]
holds.
\end{proposition}

Notice that the hypothesis \eqref{eq: hyp IBP} guarantees that $\Gcal_\rho u (x)$ exists as a Lebesgue integral for a.e.\ $x\in \Rn$, and $\mathcal{G}_\rho u \in L^1_{\loc}(\Rn,\Rn)$; the assumption $u \in L^1(\Rn)+L^\infty(\Rn)$ combined with Lemma~\ref{le:GuL1Linfty} ensures that $u\, \diver_{\rho} \Phi$ is integrable. The integration by parts formula leads naturally to the definition of the distributional nonlocal gradient. 

\begin{definition}\label{def: wnl gradient}
Given $u\in L^1(\Rn)+L^\infty(\Rn)$, we define its distributional, or weak, nonlocal gradient $D_\rho u$ as the distribution
\[
 \langle D_\rho u,\Phi\rangle=-\int_{\Rn} u(x) \diver_\rho \Phi(x)\,dx , \qquad \Phi \in C^{\infty}_c (\Rn, \Rn) .
\]
\end{definition}
Checking that $D_\rho u$ is a distribution is elementary, given Lemma~\ref{le:GuL1Linfty}.
Thanks to Proposition \ref{prop: IBP I}, if $u \in L^1(\Rn)+L^\infty(\Rn)$ satisfies \eqref{eq: hyp IBP}, its nonlocal gradient and its distributional nonlocal gradient coincide: $\Gcal_\rho u= D_\rho u$. 

\begin{definition}\label{def: FuncSpace}
Let $p \in [1, \infty]$.
We define the $\rho$-nonlocal Sobolev space $H^{ \rho,p}(\R^n)$ as the set of functions $u \in L^p (\Rn)$ such that $D_{\rho} u \in L^p (\Rn, \Rn)$, equipped with the norm
\[\|u\|_{H^{\rho,p}(\R^n)}=
\begin{cases}
 \left( \norm{u}_{L^p (\Rn)} + \norm{D_{\rho} u}_{L^p (\Rn, \Rn)} \right)^{\frac{1}{p}} & \text{for } p \in [1, \infty) , \\
 \max \{ \norm{u}_{L^\infty (\Rn)}, \norm{D_{\rho} u}_{L^\infty (\Rn, \Rn)} \} & \text{for } p = \infty .
\end{cases}
\]
For an open set $\O\subset \Rn$, we define the closed subspace
\[
H^{\rho,p}_0(\Omega):=\{u \in H^{\rho,p}(\R^n)\,:\, u=0 \ \text{a.e.~in $\Omega^c$}\}.
\]
We also denote $H^\rho(\Rn)= H^{\rho,2}(\Rn)$ and $H^{\rho}_0(\Omega) = H^{\rho,2}_0(\Omega)$.
%subspace $H^{\rho,p}_0(\Omega)$ as the closure of $C_c^{\infty}(\Omega)$ in $H^{\rho,p}(\Rn)$ when $p < \infty$.
%We also denote $H^\rho(\Rn)= H^{\rho,2}(\Rn)$ and $H^{\rho}_0(\Omega) = H^{\rho,2}_0(\Omega)$.
%The subspace $H^{\rho,\infty}_0 (\Omega)$ is defined as the set of $u \in H^{\rho,\infty}(\Rn)$ for which there exists a sequence $\{ u_j \}_{j \in \N} \subset C_c^\infty (\O)$ such that
%\begin{equation*}%\label{eq:defconvergenceinfty}
%\begin{split}
% & u_j \to u \text{ a.e.,} \quad \Gcal_{\rho} u_j \to D_{\rho} u \text{ a.e.,} \\
% & \norm{u_j}_{L^{\infty} (\Rn)} \to \norm{u}_{L^{\infty} (\Rn)} \quad \text{and} \quad \norm{\Gcal_{\rho} u_j}_{L^{\infty} (\Rn, \Rn)} \to \norm{D_{\rho }u}_{L^{\infty} (\Rn, \Rn)} \qquad \text{as } j \to \infty .
%\end{split} 
%\end{equation*}
\end{definition}

Standard arguments show that $H^{\rho,p}(\Rn)$ is a Banach space, which is separable for $p \in [1, \infty)$, reflexive for $p \in (1, \infty)$ and Hilbert for $p=2$ (see, e.g., \cite[Prop.\ 8.1]{Brezis}, \cite[Th.\ 2.1]{MeD} or \cite[Prop.\ 3.4]{BeCuMo22}). We note that the choice $\rho (x) = \frac{c_{n,s}}{|x|^{n+s-1}}$ for $s\in (0,1)$ and $p \in (1,\infty)$ gives rise to the usual Bessel-potential space $H^{s,p} (\Rn)$, see~e.g., \cite[Th.\ 1.7]{ShS2015} and \cite[Th.\ A.1]{Comi3}. Moreover, we deduce immediately from the definition that $H^{\rho,p}_0(\Rn)=H^{\rho,p}(\Rn)$.

%\textcolor{red}{Other possible definition of $H^{\rho,\infty}_0 (\Omega)$: closure of $C_c^\infty (\O)$ with respectto weak$^*$ convergence (Ball, Currie, Olver 1981, Section 2)}

The following result shows the embedding of classical Sobolev spaces into $H^{\rho,p}(\Rn)$.
\begin{proposition}\label{prop: W1p embed}
Let $p \in [1 , \infty]$.
Assume $\rho$ satisfies \ref{itm:h0}.
Then, $W^{1,p}(\Rn)\subset H^{\rho,p}(\Rn)$ with continuous embedding, i.e., there exists a constant $C=C(n,p,\rho)>0$ such that for any $u\in W^{1,p}(\Rn)$,
\[
 \| u\|_{H^{\rho,p}(\Rn)} \leq C \| u\|_{W^{1,p}(\Rn)}.
\]
Moreover, $D_\rho u= \Gcal_\rho u$. 
\end{proposition}
\begin{proof}
%We show first that there exists $C=C(n,p,\rho)$ such that
%\[ \| \Gcal_\rho u\|_{L^p(\Rn,\Rn)} \leq C \| u\|_{W^{1,p}(\Rn)},\] 
%for any $u\in W^{1,p}(\Rn)$.
Let $u\in W^{1,p}(\Rn)$.
In the case $p=\infty$, we notice that for any $x \in \Rn$,
\begin{align*}
 | \Gcal_\rho u(x) | & \leq %\int_{\Rn} \frac{|u(x+h)-u(x)|}{|h|} \rho(h)\,dh \\ & = 
 \int_{B_1}\frac{|u(x+h)-u(x)|}{|h|} \rho(h)\,dh+ \int_{B_1^c}\frac{|u(x+h)-u(x)|}{|h|} \rho(h)\,dh \\
&\leq \|\nabla u\|_{L^\infty(\Rn,\Rn)} \int_{B_1} \rho(h)\,dh+2\|u\|_{L^\infty(\Rn)} \int_{B_1^c} \frac{\rho(h)}{|h|}\,dh % \\ &
 \leq C \|u\|_{W^{1,\infty}(\Rn)},
\end{align*}
where we have used \ref{itm:h0}. Thus, $\| \Gcal_\rho u\|_{L^\infty(\Rn,\Rn)} \leq C \|u\|_{W^{1,\infty}(\Rn)}$.

For $p\in [1,\infty)$ we follow the proof of \cite[Prop.\ 2.7]{bellido2020convergence}, which we include here for the reader's convenience.
Clearly, 
\begin{align*} \|\Gcal_\rho u\|_{L^p(\Rn,\Rn)} \leq &\left\|  \int_{B_1} \frac{\abs{u(\cdot +h)-u(\cdot)}}{|h|} \rho(h)\,dh\right\|_{L^p(\Rn)} %\\ & 
+ \left\| \int_{B_1^c} \frac{\abs{u(\cdot+h)-u(\cdot)}}{|h|} \rho(h)\,dh\right\|_{L^p(\Rn)}.
\end{align*}
Applying Minkowski's integral inequality to the first term on the right-hand side of this inequality yields
\begin{align*}
 \left\|  \int_{B_1} \frac{\abs{u(\cdot+h)-u(\cdot)}}{|h|} \rho(h)\,dh \right\|_{L^p(\Rn)} & % \leq \left(\int_{\Rn} \left| \int_{B_1} |u(x+h)-u(x)| \frac{\rho(h)}{|h|} \, dh \right|^p\,dx\right)^\frac{1}{p} \\ & \leq 
 \leq \int_{B_1} \frac{\rho(h)}{|h|} \left(\int_{\Rn} |u(x+h)-u(x)|^p\, dx\right)^\frac{1}{p} dh \\
  & \leq C\|\nabla u\|_{L^p(\Rn,\Rn)} \int_{B_1}\rho(h)\,dh\leq C \| \nabla u \|_{L^p(\Rn,\Rn)} ,
  \end{align*}
where we have used \ref{itm:h0} and \cite[Prop.\ 9.3]{Brezis}.
For the second term, applying Fubini's theorem and \ref{itm:h0}, we find
\begin{align*}
 \left\|  \int_{B_1^c} \frac{\abs{u(\cdot+h)-u(\cdot)}}{|h|} \rho(h)\,dh \right\|_{L^p(\Rn)} \! & \leq \left\| \int_{B_1^c} |u(\cdot+h)| \frac{\rho(h)}{|h|} \,dh\right\|_{L^p(\Rn)} \! + \left\|\int_{B_1^c} |u(\cdot)| \frac{\rho(h)}{|h|} \,dh\right\|_{L^p(\Rn)} \\
& \leq 2 \left(\int_{B_1^c} \frac{\rho(h)}{|h|}\,dh \right)  \| u\|_{L^p(\Rn)}  \leq C\|u\|_{L^p(\Rn)}.
\end{align*}
Consequently, 
$\|\Gcal_\rho u\|_{L^p(\Rn,\Rn)}\leq C \| u\|_{W^{1,p}(\Rn)}$.

Notice that the previous arguments show that, in both cases $p = \infty$ and $p < \infty$, any $u\in W^{1,p}(\Rn)$ satisfies the hypothesis of Proposition \ref{prop: IBP I}, and, hence, $\Gcal_\rho u=D_\rho u$.
\end{proof}

The next result shows that the gradient operator commutes with convolution, see also~\cite[Lemma~3.7]{HaT23}.
\begin{lemma}\label{le:convolutiongradient} 
Let $p \in [1, \infty]$ and let $\rho$ satisfy \ref{itm:h0}.
Let $\phi \in C^{\infty}_c (\Rn)$ and $u \in H^{\rho, p}(\Rn)$.
Then, $\phi * u \in  C^{\infty} (\Rn)$ and $D_{\rho} (\phi * u) = \phi * D_{\rho} u$.
\end{lemma}
\begin{proof}
Let $\Phi\in C_c^\infty(\Rn,\Rn)$ and $x \in \Rn$.
A straightforward calculation shows that
\begin{align*}
 \phi * \diver_{\rho} \Phi (x)  &= \int_{\Rn} \int_{\Rn} \phi (x-y) \frac{\Phi(y) -\Phi(z)}{|y-z|} \cdot\frac{y-z}{|y-z|} \rho (y-z) \, dz \, dy \\
& =  \int_{\Rn} \int_{\Rn} \phi (z') \frac{\Phi (x-z') - \Phi (y'-z')}{|x-y'|} \cdot\frac{x-y'}{|x-y'|} \rho (x-y') \, dz' \, dy' \\
 %& =  \int_{\Rn} \int_{\Rn} \phi (x-y) \frac{\Phi (y) -\Phi (z)}{|y-z|}\cdot \frac{y-z}{|y-z|} \rho (y-z) \, dz \, dy , \\
  & = \diver_{\rho} (\phi * \Phi) (x),
\end{align*}
after applying the changes of variables $y' = x-y + z$ and $z' = x-y$, and having in mind that all integrals involved are absolutely convergent since $\Phi \in C^{\infty}_c (\Rn,\Rn)$ and $\rho$ satisfies \ref{itm:h0}.

Let $\tilde{\phi}$ be the reflection of $\phi$, i.e.,  $\tilde{\phi} (x) = \phi (-x)$.
Now, for any $u\in H^{\rho,p}(\Rn)$, by standard properties of convolution, $\phi * u \in  C^{\infty} (\Rn)$, and, by the previous computation, the definition of the distributional nonlocal gradient and Fubini's theorem (in particular, \cite[Prop.\ 4.16]{Brezis}), we have
\begin{align*}
\int_{\Rn} D_\rho (\phi *  u)(x) \cdot \Phi(x)\,dx &= -\int_{\Rn} \phi * u (x) \diver_\rho \Phi(x)\,dx \\
%& = -\int_{\Rn}\int_{\Rn}\phi(y) u(x-y) \diver_\rho \Phi(x)\,dy\,dx \\
%& = -\int_{\Rn}u(z) \int_{\Rn}\phi(x-z) \diver_\rho \Phi(x)\,dx\,dz \\
& = -\int_{\Rn} u(x) (\tilde{\phi}* \diver_\rho \Phi)(x)\,dx \\
&= -\int_{\Rn} u(x)\diver_\rho(\tilde{\phi} * \Phi) (x)\,dx\\
&= \int_{\Rn} D_\rho u(x)\cdot (\tilde{\phi} *\Phi)(x) \,dx \\
%&= \int_{\Rn} \int_{\Rn} D_\rho u(x) \cdot (\phi(y) \Phi(x-y))\,dy\,dx\\
%&=  \int_{\Rn}\Phi(z) \cdot  \left( \int_{\Rn} D_\rho u(x) \tilde{\phi}(x-z)\,dx\right) dz\\
&=\int_{\Rn} (\phi*D_\rho u)(x)\cdot \Phi(x)\,dx,
\end{align*}
and having in mind that $\Phi$ is an arbitrary test function, this concludes the proof. 
\end{proof}

As a consequence of Lemma \ref{le:convolutiongradient} and the standard method of approximation by convolution (see, e.g., \cite[Th.\ C.16]{Leoni09}), we find the following result.

\begin{proposition}\label{pr:density} Let $\rho$ satisfy \ref{itm:h0}. 
\begin{enumerate}[label=(\roman*)]
\item For $p \in [1,\infty)$, it holds that $C^\infty(\Rn)\cap H^{\rho,p}(\Rn)$ is dense in $H^{\rho,p}(\Rn)$.

\item For all $u\in H^{\rho,\infty}(\Rn)$, there exists a sequence $\{ u_j \}_{j \in \N}$ in $C^\infty(\Rn)\cap H^{\rho,\infty}(\Rn)$ such that
\begin{align*}
 & u_j \to u \text{ a.e.,} \quad \Gcal_{\rho} u_j \to D_{\rho} u \text{ a.e.,} \\
 & \norm{u_j}_{L^{\infty} (\Rn)} \to \norm{u}_{L^{\infty} (\Rn)} \quad \text{and} \quad \norm{\Gcal_{\rho} u_j}_{L^{\infty} (\Rn, \Rn)} \to \norm{D_{\rho }u}_{L^{\infty} (\Rn, \Rn)} \qquad \text{as } j \to \infty .
\end{align*}
\end{enumerate}
\end{proposition}

The next lemma is a Leibniz rule in this nonlocal context, which is of interest in own right, and also needed in the proof of Theorem \ref{th:density}.

\begin{lemma}\label{le:Leibniz}
Let $p \in[1, \infty]$ and let $\rho$ satisfy \ref{itm:h0}.
Let $g\in H^{\rho,p}(\Rn)$ and $f \in C_c^{\infty}(\Rn)$.
Then, $fg \in H^{\rho,p}(\Rn)$ and 
\begin{equation}\label{eq:Leibniz}
 D_\rho (fg) = f D_\rho g + K_f(g) ,
\end{equation}
with 
\[
 K_f(g) (x) =\int_{\Rn} \frac{f(x)-f(y)}{|x-y|}g(y) \frac{x-y}{|x-y|}\rho(x-y)\,dy, \quad \text{a.e.  }x\in \Rn.
\]
Moreover, there exists $C=C(\rho)>0$ such that
\begin{equation}\label{eq:Kfg}
 \| K_f(g)\|_{L^p(\Rn,\Rn)} \leq C \| f\|_{W^{1,\infty}(\Rn)} \|g\|_{L^p(\Rn)} 
\end{equation}
and, if in addition $\rho \in L^1 (\Rn)$,
\begin{equation}\label{eq:Kfg2}
 \| K_f(g)\|_{L^p(\Rn,\Rn)} \leq C \| \nabla f\|_{L^{\infty}(\Rn, \Rn)} \|g\|_{L^p(\Rn)} .
\end{equation}
\end{lemma}
\begin{proof}
Clearly, $fg \in L^p (\Rn)$.
Let $\Phi\in C_c^\infty(\Rn,\Rn)$.
It is immediate to check that for all $x \in \Rn$,
\[
 \diver_{\rho} (f \Phi) (x) = f(x) \diver_{\rho} \Phi (x) + \int_{\Rn} \frac{f(x)-f(y)}{|x-y|} \frac{x-y}{|x-y|} \rho(x-y) \cdot \Phi (y) \,dy .
\]
Thus,
\begin{align*}
 - \int_{\Rn} & f(x) g(x) \diver_{\rho} \Phi (x) \, dx \\
 & = - \int_{\Rn} g(x) \diver_{\rho} (f \Phi) (x) \, dx + \int_{\Rn} \int_{\Rn} g(x) \frac{f(x)-f(y)}{|x-y|} \frac{x-y}{|x-y|} \rho(x-y) \cdot \Phi (y) \, dy \, dx \\
 & = \int_{\Rn} f(x) D_{\rho} g (x) \cdot \Phi (x) \, dx + \int_{\Rn} K_f (g) (x) \cdot \Phi (x) \, dx ,
\end{align*}
which shows \eqref{eq:Leibniz}.
The bound \eqref{eq:Kfg} follows from Young's convolution inequality by using \ref{itm:h0} and the estimate
\[
\abslr{\frac{f(x)-f(y)}{|x-y|} \frac{x-y}{|x-y|} \rho(x-y)} \leq C\norm{f}_{W^{1,\infty}(\Rn)} \min\{1,\abs{x-y}^{-1}\}\rho(x-y),
\]
while \eqref{eq:Kfg2} uses
\[
 \abslr{\frac{f(x)-f(y)}{|x-y|} \frac{x-y}{|x-y|} \rho(x-y)} \leq C\norm{\nabla f}_{L^{\infty}(\Rn, \Rn)} \rho(x-y).
\]
%can be proved as in \cite[Lemma 3.2]{BeCuMC} or \cite[Lemma 3.3]{BellidoCuetoMoraCorral+2023}.
Therefore, $D_\rho (fg) \in L^p (\Rn)$ and $fg \in H^{\rho,p}(\R^n)$.
\end{proof}

The following result explores the density of $C^{\infty}_c$ functions, whose proof utilizes well-known mollification and cut-off arguments, see also~\cite[Th.\ 1]{CuKrSc23} and \cite[Th.\ 3.3]{HaT23}.

\begin{theorem}\label{th:density}
Let $\rho$ satisfy \ref{itm:h0}.
\begin{enumerate}[label=(\roman*)]
\item\label{item:density1} Let $1\le p<\infty$. Then, $C_c^\infty(\Rn)$ is dense in $H^{\rho,p}(\R^n)$.

\item\label{item:density2} For all $u\in H^{\rho,\infty}(\R^n)$, there exists a sequence $\{ u_j \}_{j \in \N}$ in $C^\infty_c (\Rn)$ such that
\begin{align*}
 & u_j \to u \text{ a.e.,} \quad \Gcal_{\rho} u_j \to D_{\rho} u \text{ a.e.,} \\
 & \norm{u_j}_{L^{\infty} (\Rn)} \to \norm{u}_{L^{\infty} (\Rn)} \quad \text{and} \quad \norm{\Gcal_{\rho} u_j}_{L^{\infty} (\Rn, \Rn)} \to \norm{D_{\rho }u}_{L^{\infty} (\Rn, \Rn)} \qquad \text{as } j \to \infty .
\end{align*}

\item\label{item:density3} Let $1\le p < \infty$ and let $\O \subset \Rn$ be an open and bounded set with a Lipschitz boundary. Then, $C_c^\infty(\O)$ is dense in $H^{\rho,p}_0(\O)$.
%\begin{equation}\label{eq:H0rhop}
% H_0^{\rho,p}(\O) = \left\{ u\in H^{\rho,p} \;:\; u=0 \text{  in  }\O^c\right\}.
%\end{equation}

%\item\label{item:density4} Let $p = \infty$ and let $\O \subset \Rn$ be a half-space.
%Then \eqref{eq:H0rhop} holds.
\end{enumerate}
\end{theorem}
\begin{proof}
Due to Proposition~\ref{prop:carryover} and Example~\ref{ex:carryover}\,\ref{itm:carryover} below, we assume without loss of generality that $\rho \in L^1(\Rn)$. Parts \ref{item:density1} and \ref{item:density2} can be proved as in \cite[Th.\ 1]{CuKrSc23}.
Indeed, consider $\chi\in C_c^\infty(\Rn)$ with $0 \leq \chi \leq 1$ everywhere, $\chi = 1$ in $B_1$, and define $\chi_k:=\chi(\frac{\cdot}{k})$ for $k\in \N$.
Thanks to Proposition \ref{pr:density}, it is enough to construct an approximating sequence for $u \in C^{\infty} (\Rn)\cap H^{\rho,p}(\Rn)$. 
Then, the sequence $\{ \chi_k u \}_{k \in \N}$ is in $C_c^\infty(\Rn)$.
Moreover, for $p \in [1, \infty)$ we have $\chi_k u\to u$ in $L^p(\Rn)$ as $k\to \infty$ and, by Lemma \ref{le:Leibniz},
\[
 \|D_\rho u- D_\rho(\chi_k u)\|_{L^p(\Rn,\Rn)} \leq \|(1-\chi_k)D_\rho u\|_{L^p(\Rn,\Rn)}+C\frac{\| \nabla \chi \|_{L^{\infty}(\Rn, \Rn)}}{k}\|u\|_{L^p(\Rn)}\to 0 \text{  as  }k\to\infty ,
\]
which yields \ref{item:density1}.
For $p = \infty$, we have $|\chi_k u| \leq |u|$ and $\chi_k u\to u$ a.e.
In addition, by Lemma \ref{le:Leibniz}, $D_{\rho} (\chi_k u) = \chi_k D_{\rho} u + K_{\chi_k} (u)$. Because $|\chi_k D_{\rho} u| \leq |D_{\rho} u|$ on $\R^n$,  $\chi_k D_{\rho} u\to D_{\rho} u$ a.e., and
\[
 \norm{K_{\chi_k} (u)}_{L^{\infty} (\Rn, \Rn)} \leq \frac{C}{k} \norm{\nabla \chi}_{L^{\infty} (\Rn, \Rn)} \norm{u}_{L^{\infty} (\Rn)} \to 0 \quad \text{as } k \to \infty ,
\]
we find that \ref{item:density2} holds. \smallskip

Part \ref{item:density3}. We first show that for each $u \in H_0^{\rho,p}(\O)$ and $\epsilon >0$, there exists a $\tilde{u} \in H^{\rho,p}(\Rn)$ with
\begin{equation}\label{eq:shrinkgoal}
\supp \tilde{u} \subset \Omega \quad \text{and} \quad \norm{u-\tilde{u}}_{H^{\rho,p}(\R^n)} \leq \frac{\epsilon}{2}.
\end{equation}
To this aim, we may use the fact that the boundary of $\Omega$ is Lipschitz, to find a partition of unity $\chi_0,\dots,\chi_N \in C_c^{\infty}(\R^n)$ subject to $\O$ and vectors $\zeta_1,\dots,\zeta_N \in \R^n$ such that
\begin{equation}\label{eq:partitionunity}
\sum_{i=0}^{N} \chi_i = 1 \ \text{on $\O$}, \quad \chi_0 \in C_c^{\infty}(\O),
\end{equation}
and
\begin{equation}\label{eq:shrinksupp}
(\supp \chi_i \cap \overline{\O}) +\lambda\zeta_i \subset \O \quad \text{for all $i=1,\dots,N$ and $\lambda >0$ small enough.}
\end{equation}
For such $\lambda$, we define the function
\[
\tilde{u}:=\chi_0 u +\sum_{i=1}^N \tau_{\lambda\zeta_i}(\chi_i u),
\]
where $\tau_\zeta(v):=v(\cdot - \zeta)$ for $v:\Rn \to \R$ denotes translation by the vector $\zeta \in \Rn$. In view of the Leibniz rule from Lemma~\ref{le:Leibniz} and the translation invariance of $D_\rho$, we deduce that $\tilde{u} \in H^{\rho,p}(\Rn)$. Moreover, due to \eqref{eq:shrinksupp}, we find that $\supp \tilde{u} \subset \O$, which guarantees the first condition in \eqref{eq:shrinkgoal}.

For the norm estimate, we may use the $L^p$-continuity of the translation operator and the translation invariance of $D_\rho$, to find a $\lambda=\lambda_\epsilon$ with
\[
\norm{u-\tilde{u}}_{H^{\rho,p}(\Rn)}^p = \normB{\sum_{i=1}^N \chi_i u - \tau_{\lambda_\epsilon \zeta_i}(\chi_i u)}_{L^p(\Rn)}^p+ \normB{\sum_{i=1}^N D_\rho(\chi_i u) - \tau_{\lambda_\epsilon \zeta_i}D_\rho(\chi_i u)}_{L^p(\Rn,\Rn)}^p \leq \left(\frac{\epsilon}{2}\right)^p,
\]
where we have used $u=\sum_{i=0}^N \chi_i u$ in light of the first part of \eqref{eq:partitionunity} and the fact that $u \in H^{\rho,p}_0(\O)$. This proves \eqref{eq:shrinkgoal}. By mollifying $\tilde{u}$ suitably, we can find a $\phi \in C_c^{\infty}(\Omega)$ such that $\norm{\tilde{u}-\phi}_{H^{\rho,p}(\Rn)}\leq \epsilon/2$, which yields
\[
\norm{u-\phi}_{H^{\rho,p}(\Rn)} \leq \norm{u-\tilde{u}}_{H^{\rho,p}(\Rn)}+\norm{\tilde{u}-\phi}_{H^{\rho,p}(\Rn)} \leq \epsilon,
\]
and finishes the proof.
\end{proof}

\subsection{Equivalence of spaces with different kernels}\label{se:equivalence}

Here, we provide a sufficient condition so that two kernels give rise to the same space.
This condition describes that the two kernels behave similarly at the origin.
Moreover, one can carry over Poincar\'{e} inequalities from one gradient to the other.

\begin{proposition}\label{prop:carryover}
Let $\O \subset \Rn$ be open.
Let $\rho_1, \rho_2$ satisfy \ref{itm:h0} and assume that $(\rho_1-\rho_2)/\abs{\cdot} \in L^1 (\Rn)$. Then, the following two statements hold:

\begin{enumerate}[label = (\roman*)]
\item\label{item:carry1} Let $p \in [1, \infty]$.
The identity $H^{\rho_1,p}_0(\Omega)= H^{\rho_2,p}_0(\Omega)$ holds with equivalent norms.

\item\label{item:carry2}
Let $p \in (1, \infty)$ and assume that $\O$ is bounded.
If there is a $C_1>0$ such that for all $u \in H^{\rho_1,p}_0(\Omega)$,
\begin{equation}\label{eq:poincareassumption}
\norm{u}_{L^p(\Omega)} \leq C_1\norm{D_{\rho_1}u}_{L^p(\R^n,\R^n)},
\end{equation}
then there is a $C_2>0$ such that for all $u \in H^{\rho_2,p}_0(\Omega)$,
\[
\norm{u}_{L^p(\Omega)} \leq C_2\norm{D_{\rho_2}u}_{L^p(\R^n,\R^n)}.
\]
\end{enumerate}
\end{proposition}
\begin{proof}
Part \ref{item:carry1}. Set 
\[
F(x)=\frac{x}{\abs{x}}\frac{\rho_2(x)-\rho_1(x)}{\abs{x}} .
\]
Then, we find that $F \in L^1 (\R^n, \Rn)$ and, for all $\Phi \in C^{\infty}_c (\Rn,\Rn)$,
\begin{equation*}
 \Div_{\rho_1} \Phi = \Div_{\rho_2} \Phi + F*\Phi.
\end{equation*}
By Young's inequality, the operator $u \mapsto F * u$ is bounded on $L^p (\Rn)$. Hence, for $u \in H^{\rho_2,p}_0(\Omega)$ and $\Phi \in C^{\infty}_c (\Rn,\Rn)$ it holds that
\begin{equation*}
\int_{\R^n} (D_{\rho_2} u + F*u)\cdot \Phi\,dx = - \int_{\R^n} u \,(\Div_{\rho_2} \Phi + F*\Phi)\,dx
= -\int_{\R^n} u \Div_{\rho_1}\Phi\,dx,
\end{equation*}
where the first equality uses Fubini's theorem and the definition of the weak nonlocal gradient $D_{\rho_2}$. Therefore, $u \in H^{\rho_1,p}_0 (\O)$,
\begin{equation}\label{eq:Gr1uGr2u}
 D_{\rho_1} u = D_{\rho_2} u + F*u
\end{equation}
and
\[
 \norm{D_{\rho_1} u}_{L^p (\R^n, \Rn)} \leq \norm{D_{\rho_2} u}_{L^p (\R^n, \Rn)} + \norm{F}_{L^1 (\R^n, \Rn)} \norm{u}_{L^p (\Rn)} .
\]
The reverse inclusion and inequality are proved analogously.

 \smallskip

Part \ref{item:carry2}.
We first prove that if $u \in C^{\infty}_c (\Rn)$ satisfies $\Gcal_{\rho_2} u = 0$, then $u=0$.
By Lemma \ref{le:lamda},
\[
 0 = \widehat{\Gcal_{\rho_2} u}(\xi) = \lambda_{\rho_2}(\xi) \widehat{u}(\xi) \quad \text{for all $\xi \in \R^n$}.
\]
If $u$ were not identically zero, then $\widehat{u}$ would be a non-zero analytic function, and hence, non-zero in a set of full measure. As such, we deduce that $\lambda_{\rho_2}=0$ a.e.
A further application of Lemma \ref{le:lamda} to any $\phi \in C_c^{\infty}(\R^n)$ implies that $\widehat{\Gcal_{\rho_2}\phi} = 0$ a.e., hence $\Gcal_{\rho_2}\phi = 0$ a.e., and, in fact, everywhere thanks to Proposition \ref{pr:Qhat} \ref{item:Qhat1}.
On the other hand, choosing a $\phi$ such that $\phi(y)=-\phi(-y)$ for all $y \in \Rn$, $\phi \geq 0$ in $\{ y_1 > 0 \}$ and $\phi$ is not identically zero on $B_{\e} \cap \{ y_1 > 0 \}$ yields
\[
\Gcal_{\rho_2}\phi(0) =   \int_{\{ y_1 > 0 \}} \frac{2\phi(y)}{\abs{y}}\frac{y}{\abs{y}} \rho_2 (y)\,dy >0,
\]
cf.~\ref{itm:h0}. This contradiction concludes that $u$ must be zero, which proves the claim. 

Now we show that if $u \in H^{\rho_2, p}_0 (\Omega)$ satisfies $D_{\rho_2} u = 0$, then $u=0$.
Let $\phi \in C^{\infty}_c (\Rn)$.
By Lemma \ref{le:convolutiongradient} and the fact that $\O$ is bounded, $\phi * u \in C^{\infty}_c (\Rn)$ and
\[
 \Gcal_{\rho_2} (\varphi * u) = \varphi * D_{\rho_2} u = 0 .
\]
By the claim above, $\phi * u = 0$.
As this is true for all $\phi \in C^{\infty}_c (\Rn)$, we conclude, by taking $\phi$ to be a family of mollifiers, that $u = 0$.

Now, to prove the statement, we argue by contradiction. Suppose $(u_j)_j \subset H^{\rho_2,p}_0 (\O)$ is a sequence satisfying
\[
 1 = \norm{u_j}_{L^p(\Omega)} > j \norm{D_{\rho_2}u_j}_{L^p(\R^n,\R^n)} \quad \text{for all $j \in \N$}.
\]
As $H^{\rho_2,p}_0 (\O)$ is reflexive, there exists $u \in H^{\rho_2,p}_0 (\O)$ such that, up to subsequence, $u_j \weakto u$ in $H^{\rho_2,p}_0 (\O)$.
As $D_{\rho_2}u_j \to 0$ in $L^p(\R^n,\R^n)$, we obtain that in fact $D_{\rho_2}u=0$, which implies $u = 0$ thanks to the result of the previous paragraph.

Since for all large $R>0$, it holds that $\supp u_j \subset B_R$ for all $j \in \N$, we find that
\[
\norm{F*u_j}_{L^p(B_{2R}^c,\Rn)} = \norm{(\mathbbm{1}_{B_R^c}F)*u_j}_{L^p(B_{2R}^c,\Rn)} \leq \norm{F}_{L^1(B_R^c,\Rn)} \to 0,
\]
as $R \to \infty$. Together with the compactness of convolution operators on bounded sets due to the Fr\'{e}chet-Kolmogorov criterion (cf.~\cite[Cor.\ 4.28]{Brezis}), we infer that $F*u_j \to 0$ in $L^p(\Rn)$, and hence, by \eqref{eq:Gr1uGr2u}
\[
 D_{\rho_1}u_j = D_{\rho_2}u_j + F*u_j \to 0 \quad \text{in $L^p(\R^n,\R^n)$}.
\]
In view of \eqref{eq:poincareassumption}, we deduce that $u_j \to 0$ in $L^p(\Omega)$ which contradicts $\norm{u_j}_{L^p(\Omega)}=1$.
\end{proof}

We note that in part~\ref{item:carry1}, we can take $\Omega=\Rn$, which yields the correspondence $H^{\rho_1,p}(\Rn)=H^{\rho_2,p}(\Rn)$ with equivalent norms. Additionally, having in mind \eqref{eq:H0corbarrho}, we can see that the condition $(\rho_1-\rho_2)/\abs{\cdot} \in L^1 (\Rn)$ can be equivalently written as
\[
 \int_0^r \left( \overline{\rho_1} (t) - \overline{\rho_2} (t) \right) t^{n-2} \, dt < \infty , \quad \text{for some } r > 0 .
\]

\begin{example}\label{ex:carryover}
We present two applications of Proposition~\ref{prop:carryover}.
\begin{enumerate}[label=(\alph*)]
\item \label{itm:carryover}
Let $\rho$ satisfy \ref{itm:h0} and let $\chi \in L^{\infty} (\Rn)$ be radial with $\chi \geq 0$ and $\chi|_{B_r} = 1$ for some $r>0$.
Then $\chi \rho$ satisfies \ref{itm:h0} and $H^{\rho,p}_0(\Omega)= H^{\chi \rho,p}_0(\Omega)$ for all $p \in [1, \infty]$ and any open set $\O \subset \Rn$.
The assumption $\chi|_{B_r} = 1$ can be weakened to $\inf_{B_{r}} \chi > 0$ and
\[
 \int_{B_r} \frac{\chi (x) - \chi (0)}{|x|} \rho (x) \, dx < \infty .
\]
Moreover, if $p \in (1, \infty)$ and $\O$ is bounded, any Poincar\'e inequality for $D_{\rho}$ implies an analogous one for $D_{\chi \rho}$.
We conclude that we can associate to every kernel satisfying \ref{itm:h0} a kernel satisfying \ref{itm:h0} with compact support and giving rise to the same space.
In the development of the theory we will often require that the kernel has compact support.
Thanks to this observation, this can be assumed without loss of generality. \smallskip

\item
Let $s \in (0,1)$.
Let $\chi \in L^{\infty} (\Rn)$ be radial with $\chi \geq 0$ and $(\chi(0)-\chi)/\abs{\cdot}^{n+s} \in L^1(\R^n)$; the latter condition happens, for example, if $\chi$ is $\gamma$-H\"{o}lder continuous at $0$ with $\gamma >s$.
Then the kernel
\[
\rho(x) = \frac{\chi(x)}{\abs{x}^{n+s-1}}
\]
can be compared with the kernel \eqref{eq:Rieszkernel} of the Riesz $s$-fractional gradient.
Thus, $H^{s, p}_0 (\O) = H^{\rho, p}_0 (\O)$ for any open $\O \subset \Rn$ and $p \in [1, \infty]$.
If, in addition, $\O$ is bounded and $p \in (1, \infty)$, we have that $\norm{u}_{L^p(\Omega)} \leq C \norm{D_{\rho} u}_{L^p(\R^n,\R^n)}$ for all $u \in H^{\rho, p}_0 (\O)$, as a consequence of the corresponding inequality for the Bessel-potential space $H^{s,p}$ (see \cite[Th.\ 1.8]{ShS2015}).
This constitutes an alternative proof, as well as a generalization of the Poincar\'e inequality for truncated fractional gradients of \cite[Th.\ 6.2]{BeCuMo22}.

\end{enumerate}
\end{example}

\section{Poincar\'{e} inequalities and compact embeddings}\label{sec:poinc}
In this section, we derive conditions on the kernel $\rho$ such that the $\rho$-nonlocal gradient satisfies a Poincar\'{e} inequality, and such that the spaces $H^{\rho,p}_0(\O)$ are compactly embedded into $L^p$. The argument is based on inverting the nonlocal gradient via \eqref{eq:ftocfourier}, and showing with Fourier techniques that this is a bounded or compact operation. The results in the case $p=2$ are tackled first and rely on Parseval's identity. Subsequently, the general case $p \in (1,\infty)$ is considered, which requires an additional assumption in order to apply the Mihlin-H\"{o}rmander multiplier theorem.

\subsection{Positivity of $\widehat{Q}_\rho$}\label{se:positivity}

W show in this section that $\widehat{Q}_\rho > 0$, as a first step to make sense of the expression \eqref{eq:ftocfourier}. We also relate the decay of $\widehat{Q}_\rho$ at infinity with the behavior of $\rho$ around the origin.
For this, we need the following assumption:
\begin{enumerate}[label = (H\arabic*)]
\item \label{itm:h1} The function $f_\rho:(0,\infty)\to \R, \ t \mapsto t^{n-2}\overline{\rho}(t)$ is decreasing, and there is a $0<\mu<1$ such that $\mu f_\rho(t/2) \geq  f_\rho(t)$ for $t \in (0,\epsilon)$.
\end{enumerate}
\begin{remark}\label{rem:h0h1relation}
We note that the second part (the doubling property) of \ref{itm:h1} is satisfied if there is a $\nu >0$ such that $t \mapsto t^{\nu}f_\rho(t)$ is decreasing on $(0,\epsilon)$. The constant $\mu$ is then given by $2^{-\nu}$.

Additionally, if $f_\rho$ is differentiable, then a simple calculation with the product rule shows that $t \mapsto t^{\nu}f_\rho(t)$ is decreasing on $(0,\epsilon)$ if and only if
\[
-\frac{1}{\nu}\frac{d}{dt}f_\rho(t) \geq \frac{f_\rho(t)}{t} \quad \text{for all $t \in (0,\epsilon)$}.
\]
The latter condition will appear again in \ref{itm:h2}, and hence, implies the doubling property of \ref{itm:h1}.
\end{remark}

\begin{example}\label{ex:H0H1}
Classes of kernels $\rho$ satisfying \ref{itm:h0}--\ref{itm:h1} are:
\begin{enumerate}[label=(\alph*)]
\item $\rho$ of Example \ref{ex:H0} \ref{item:Riesz}.

\item\label{item:alphabeta} Given $n-2 < \alpha < n$ and $\beta > n-1$,
\[
 \rho (x) = \frac{\mathbbm{1}_{B_1}(x)}{\abs{x}^{\alpha}} +  \frac{\mathbbm{1}_{B_1^c}(x)}{\abs{x}^{\beta}} .
\]

\item Given $n-2 < \alpha < n$,
\[
\rho(x) = \mathbbm{1}_{B_1}(x) \frac{- \log \abs{x}}{\abs{x}^{\alpha}},
\]

\item Given $n-2 < \alpha < n$ and $r>0$ sufficiently small,
\[
\rho(x) = \frac{\mathbbm{1}_{B_{r}(x)}}{\abs{x}^{\alpha} (- \log \abs{x})},
\]

\item If $\rho_1, \rho_2$ satisfy \ref{itm:h0}--\ref{itm:h1} and $\alpha_1, \alpha_2>0$ then $\alpha_1 \rho_1 + \alpha_2 \rho_2$ satisfies \ref{itm:h0}--\ref{itm:h1}.

\end{enumerate}
\end{example}

We may now state the following result, whose proof takes inspiration from \cite[Lemma~5.3]{BeCuMo22}.

\begin{lemma}\label{le:Qhatpositive}
Suppose that $\rho$ satisfies \ref{itm:h0}, \ref{itm:h1} and $\rho \in L^1 (\Rn)$.
Then, $\widehat{Q}_\rho$ is positive and there is a $C>0$ such that
\[
\widehat{Q}_\rho(\xi) \geq C \frac{\overline{\rho}(1/\abs{\xi})}{\abs{\xi}^{n}} \quad \text{and} \quad \widehat{Q}_\rho(\xi) \geq \frac{C}{\abs{\xi}^2} \qquad \text{for all } \xi \in B_{2/\epsilon}^c .
\]
\end{lemma}
\begin{proof}
Note that since $\widehat{Q}_\rho$ is continuous (by Lemma \ref{le:Q} \ref{item:Q3}) and $\widehat{Q}_\rho(0)=\norm{Q_\rho}_{L^1(\R^n)}>0$, we have that $\widehat{Q}_\rho$ is positive around the origin.
Next, for $\xi \not = 0$ we obtain from \eqref{eq:Qhatrho} and the coarea formula that
\begin{equation}\label{eq:hatQcoarea}
\begin{split}
\widehat{Q}_\rho(\xi) &= %\frac{1}{2\pi\abs{\xi}}\int_{\R^n}\frac{\rho(x)x_1}{\abs{x}^2}\sin(2\pi\abs{\xi}x_1)\,dx\\
 %&= 
 \frac{1}{\pi\abs{\xi}} \int_{\Sn_+} z_1 \int_0^\infty r^{n-2}\overline{\rho}(r)\sin(2 \pi \abs{\xi} r z_1)\,dr\,d \mc{H}^{n-1} (z) \\
&=\frac{1}{\pi\abs{\xi}} \int_{\Sn_+} z_1 \int_0^\infty f_\rho(r) \sin(2 \pi \abs{\xi} r z_1)\,dr\,d \mc{H}^{n-1} (z) .
\end{split}
\end{equation}
Let $\theta>0$, which will play the role of $\abs{\xi} z_1$.
We have
\begin{equation}\label{eq:oscillatory1}
 \int_0^\infty f_\rho(r) \sin(2 \pi \theta r)\,dr = \sum_{k=0}^{\infty} \int_{\frac{k}{\theta}}^{\frac{k + 1}{\theta}} f_\rho(r) \sin(2 \pi \theta r)\,dr
\end{equation}
and for each $k \in \N$,
\begin{equation}\label{eq:oscillatory2}
 \int_{\frac{k}{\theta}}^{\frac{k + 1}{\theta}} f_\rho(r) \sin(2 \pi \theta r)\,dr =  \int_{\frac{k}{\theta}}^{\frac{k + \frac{1}{2}}{\theta}} \left( f_\rho(r) - f_\rho \left( r + \frac{1}{2\theta} \right) \right) \sin(2 \pi \theta r)\,dr \geq 0 ,
\end{equation}
since $f_{\rho}$ is decreasing.
Moreover,
\[
 \int_{0}^{\frac{1}{2\theta}} \left( f_\rho(r) - f_\rho \left( r + \frac{1}{2\theta} \right) \right) \sin(2 \pi \theta r)\,dr > 0,
\]
since otherwise $f_{\rho}$ would be constant near zero, contradicting the doubling property in \ref{itm:h1}.
This shows that $\widehat{Q}_\rho(\xi) > 0$.

Now, for $\theta > 1/\epsilon$ and $0 < r < \frac{1}{2\theta}$, we have
\[
  f_{\rho}(r)-f_{\rho}\left(r+\frac{1}{2\theta} \right) \geq f_{\rho}(r)-f_{\rho}(2r) \geq \left(1-\mu \right) f_\rho(r) \geq \left(1-\mu\right)f_\rho\left(\frac{1}{2\theta}\right) ,
\]
where we have used that $f_\rho$ is decreasing, as well as the doubling condition in \ref{itm:h1}.
Therefore,
\begin{equation*}%\label{eq:longcalc}
%\begin{split}
 \int_0^{\frac{1}{2\theta}} \left(f_{\rho}(r)-f_{\rho}\left(r+\frac{1}{2\theta}\right)\right) \sin(2 \pi \theta r)\,dr \geq \left(1-\mu\right)f_\rho\left(\frac{1}{2\theta}\right) \int_0^{\frac{1}{2\theta}} \sin(2 \pi \theta r)\,dr
= \frac{1 - \mu}{\pi \theta} f_\rho\left(\frac{1}{2\theta}\right).
%\end{split}
\end{equation*}
The above inequality together with \eqref{eq:oscillatory2} and \eqref{eq:hatQcoarea} show that for all $\abs{\xi} > 2/\epsilon$,
\begin{align*}
\begin{split}
\widehat{Q}_{\rho}(\xi) & %=\frac{1}{\pi\abs{\xi}} \int_{\Sn_+} z_1 \int_0^\infty f_\rho(r) \sin(2 \pi \abs{\xi} r z_1)\,dr\,d \mc{H}^{n-1} (z)\\
%&
\geq \frac{1}{2\pi\abs{\xi}} \int_{\Sn \cap \{z_1 > 1/2\}} \int_0^\infty f_\rho(r) \sin(2 \pi \abs{\xi} r z_1)\,dr\,d \mc{H}^{n-1} (z) \\
&\geq \frac{1 - \mu}{2\pi^2\abs{\xi}^2} \int_{\Sn \cap \{z_1 > 1/2\}} \frac{1}{z_1}f_\rho\left(\frac{1}{2\abs{\xi}z_1}\right)\,d \mc{H}^{n-1} (z) %\\ &
\geq \frac{C}{|\xi|^2} f_\rho\left(\frac{1}{\abs{\xi}}\right) ,
\end{split}
\end{align*}
where in the last inequality we have used that $f_{\rho}$ is decreasing.
The constant $C$ depends on $\mu$ and $n$.
The proof is concluded thanks to the definition of $f_{\rho}$, as well as the fact that $f_\rho\left(\frac{1}{\abs{\xi}}\right) \geq f_\rho\left(\frac{\epsilon}{2}\right)$.
\end{proof}

The following proposition constitutes a further step to prove formula \eqref{eq:ftocfourier}.

%\begin{definition}\label{de:Wrho}
%We define the tempered distribution $W_\rho \in \Scal'(\R^n;\C^n)$ given by
%\begin{equation}\label{eq:vrhofourier}
%\inner{W_\rho,\eta} = \lim_{r \downarrow 0} \int_{B_r^c} \frac{-i\xi}{2\pi\abs{\xi}^2\widehat{Q}_\rho(\xi)}\eta(\xi)\,d\xi , \qquad \eta \in \Scal(\R^n) .
%\end{equation}
%%\textcolor{red}{The same symbol denotes the tempered distribution $W_\rho \in \Scal'(\R^n)$ given by
%%\begin{equation*}%\label{eq:vrhofourier}
%% W_\rho \cdot \eta = \lim_{r \downarrow 0} \int_{B_r(0)^c} \frac{-i\xi}{2\pi\abs{\xi}^2\widehat{Q}_\rho(\xi)} \cdot \eta(\xi)\,d\xi , \qquad \eta \in \Scal(\R^n, \C^n) .
%%\end{equation*}}
%\end{definition}

\begin{proposition}\label{pr:Wrho}
Suppose that $\rho$ satisfies \ref{itm:h0}, \ref{itm:h1} and $\rho \in L^1 (\Rn)$.
Then $W_\rho$ given by
\begin{equation}\label{eq:vrhofourier}
\inner{W_\rho,\eta} = \lim_{r \downarrow 0} \int_{B_r^c} \frac{-i\xi}{2\pi\abs{\xi}^2\widehat{Q}_\rho(\xi)}\eta(\xi)\,d\xi , \qquad \text{for} \ \eta \in \Scal(\R^n),
\end{equation}
defines a tempered distribution. Moreover, if $\rho$ has compact support, it holds that
\begin{equation}\label{eq:ftoctempered}
W_\rho \cdot \widehat{\Gcal_\rho \phi} = \widehat{\phi} \quad \text{for all $\phi \in C_c^{\infty}(\Rn)$}.
\end{equation}
\end{proposition}

\begin{proof}
The integrand of \eqref{eq:vrhofourier} is well defined since $\widehat{Q}_\rho$ is positive (Lemma \ref{le:Qhatpositive}). Let us see that the integral is absolutely convergent for each $r > 0$, and for this we can assume that $r < 2/ \epsilon$.
For $|\xi| \geq 2/\epsilon$ we have that
\[
 \left| \frac{-i\xi}{2\pi\abs{\xi}^2\widehat{Q}_\rho(\xi)}\eta(\xi) \right| \leq \frac{\abs{\eta(\xi)}}{2\pi\abs{\xi}\widehat{Q}_\rho(\xi)} \leq \frac{\abs{\eta(\xi)} \abs{\xi}}{2\pi C} ,
\]
thanks to Lemma~\ref{le:Qhatpositive}, and the right-hand side is integrable in $B_{2 / \epsilon}^c$ since $\eta \in \Scal(\R^n)$.
For $r < |\xi| < 2/\epsilon$ we have that
\[
 \left| \frac{-i\xi}{2\pi\abs{\xi}^2\widehat{Q}_\rho(\xi)}\eta(\xi) \right| \leq \frac{\norm{\eta}_{\infty}}{2\pi\abs{\xi}\widehat{Q}_\rho(\xi)} ,
\]
and the right-hand side is integrable in $B_{2 / \epsilon} \setminus B_r $ since $\widehat{Q}_\rho$ is positive (Lemma \ref{le:Qhatpositive}) and continuous (Lemma \ref{le:Q} \ref{item:Q3}).

Now, by symmetry (as $\widehat{Q}_\rho$ is radial; cf.\ \eqref{eq:Qhatrho}),
\[
 \int_{B_{2/\epsilon} \setminus B_r^c} \frac{-i\xi}{2\pi\abs{\xi}^2\widehat{Q}_\rho(\xi)}\eta(\xi)\,d\xi = \int_{B_{2/\epsilon} \setminus B_r^c} \frac{-i\xi}{2\pi\abs{\xi}^2\widehat{Q}_\rho(\xi)}\parent{\eta(\xi) - \eta (0)} d\xi ,
\]
with
\[
 \left| \frac{-i\xi}{2\pi\abs{\xi}^2\widehat{Q}_\rho(\xi)}\parent{\eta(\xi) - \eta (0)} \right| \leq \frac{\norm{\nabla \eta}_{\infty}}{2 \pi \widehat{Q}_\rho(\xi)} ,
\]
and the right-hand side is integrable in $B_{2/\epsilon}$ since $\widehat{Q}_\rho$ is positive and continuous.
Therefore, the limit of \eqref{eq:vrhofourier} exists and defines a tempered distribution.

For the final part, we note that if $\rho$ has compact support, then $\Gcal_\rho \phi \in C_c^{\infty}(\Rn,\Rn)$ for all $\phi \in C_c^{\infty}(\Rn)$ by Proposition~\ref{pr:Qhat} \ref{item:Qhat1} and hence, $\widehat{\Gcal_\rho \phi} \in \Scal(\Rn,\C^n)$. With this and Proposition \ref{pr:Qhat} \ref{item:Qhat2}, we obtain that for all $\eta \in \Scal(\R^n)$,
\begin{equation*}
\dprlr{W_\rho\cdot \widehat{\Gcal_\rho \phi},\eta} = \dprlr{W_\rho,\widehat{\Gcal_\rho \phi} \, \eta}  
= \lim_{r \downarrow 0} \int_{B_r^c} \frac{-i\xi}{2\pi\abs{\xi}^2\widehat{Q}_\rho(\xi)}\cdot2\pi i \xi \widehat{Q}_\rho(\xi)\widehat{\phi}(\xi)\eta(\xi)\,d\xi 
 = \int_{\R^n} \widehat{\phi}(\xi) \eta(\xi)\,d\xi.
\end{equation*}
This proves that $W_\rho \cdot \widehat{\Gcal_\rho \phi} = \widehat{\phi}$, as desired.
\end{proof}

Note that, when $n>1$, the distribution $W_\rho$ actually agrees with the locally integrable function
\[
 W_\rho (\xi) = \frac{-i\xi}{2\pi\abs{\xi}^2\widehat{Q}_\rho(\xi)} .
\]

\subsection{Poincar\'{e} inequality and Compactness in $L^2$}\label{se:PoincareL2}

The bounds in Lemma~\ref{le:Qhatpositive} allow us to swiftly prove a Poincar\'{e} inequality and compactness result in the $L^2$-setting, by prescribing that $\rho$ is of compact support and satisfies certain bounds.

Recall from Section \ref{se:first} that if $\supp \rho = \overline{B_\delta}$ for some $\d>0$, given an open set $\O \subset \Rn$ and $\phi \in C_c^{\infty}(\Omega)$, the function $\Gcal_\rho \phi$ is supported in $\Omega_\delta=\Omega+B_\d$.

\begin{theorem}\label{th:l2bound}
Let $\Omega \subset \R^n$ be open and bounded and suppose that $\rho$ satisfies \ref{itm:h0}, \ref{itm:h1} and has compact support. Then, the following two statements hold:
\begin{enumerate}[label = (\roman*)]
\item\label{item:Poincare1} If $\liminf_{t \downarrow0} t^{n-1}\overline{\rho}(t)>0$, then there is a $C=C(\Omega,n,\rho)>0$ such that
\[
\norm{u}_{L^2(\Omega)} \leq C \norm{D_\rho u}_{L^2(\R^n,\R^n)} \quad \text{for all $u \in H^\rho_0(\Omega)$}.
\]
\item\label{item:Poincare2} If $\lim_{t \downarrow 0}t^{n-1}\overline{\rho}(t)=\infty$, then $H^\rho_0(\Omega)$ is compactly embedded into $L^2(\R^n)$.
\end{enumerate}
\end{theorem}
\begin{proof}
Note first that as $\rho$ has compact support, $\rho \in L^1 (\Rn)$, and, hence, Proposition \ref{pr:Wrho} can be applied.  \smallskip

Part \ref{item:Poincare1}. Let $W_\rho \in \Scal'(\R^n,\C^n)$ be as in \eqref{eq:vrhofourier} and let $\chi \in C_c^{\infty}(\R^n)$ be radial with $\chi =1$ on $B_1$. Then, we set
\[
L = \left(\chi W_\rho\right)^{\vee} \in C^{\infty}(\R^n,\R^n) \quad \text{and} \quad M = (1-\chi)W_\rho .
\]
We specify the above definitions.
Naturally, $\chi W_\rho \in \Scal'(\R^n,\C^n)$ is the distribution defined as
\[
 \inner{\chi W_\rho,\eta} = \lim_{r \downarrow 0} \int_{B_r^c} \chi(\xi) \frac{-i\xi}{2\pi\abs{\xi}^2\widehat{Q}_\rho(\xi)}\eta(\xi)\,d\xi , \qquad \eta \in \Scal(\R^n) ;
\]
cf.\ \eqref{eq:vrhofourier}.
As $\chi$ has compact support, so does $\chi W_\rho$, and, hence, by the Paley--Wiener theorem, $L$ is analytic.
Likewise, $M \in \Scal'(\R^n,\C^n)$ is the distribution defined as
\[
 \inner{M,\eta} = \int_{\Rn} \parent{1 - \chi(\xi)} \frac{-i\xi}{2\pi\abs{\xi}^2\widehat{Q}_\rho(\xi)}\eta(\xi)\,d\xi , \qquad \eta \in \Scal(\R^n) ,
\]
so $M$ can be identified with the function
\begin{equation}\label{eq:M}
 M(\xi) = \parent{1 - \chi(\xi)} \frac{-i\xi}{2\pi\abs{\xi}^2\widehat{Q}_\rho(\xi)} , \qquad \xi \in \Rn.
\end{equation}
Moreover, $M$ is smooth (by Lemma \ref{le:Q} \ref{item:Q3}, as $Q_\rho$ has compact support) and bounded by Lemma~\ref{le:Qhatpositive} and the fact
\begin{equation}\label{eq:limsuptinfty}
 \limsup_{t\uparrow \infty} \frac{t^{n-1}}{\overline{\rho} (1/t)} < \infty ,
\end{equation}
which is a consequence of the assumption in \ref{item:Poincare1}.
Therefore, we may define the bounded operator
\begin{equation}\label{eq:TM}
 T_M:L^2(\Rn,\R^n) \to L^2(\R^n), \quad U \mapsto \left(M \cdot \widehat{U}\right)^{\vee}.
\end{equation}
On the other hand, using Proposition \ref{pr:Wrho} we have that for all $\phi \in C_c^{\infty}(\Rn)$,
\[
 \widehat{\phi} = W_\rho \cdot \widehat{\Gcal_\rho \phi}  =  (\chi W_\rho) \cdot \widehat{\Gcal_\rho \phi} + M \cdot \widehat{\Gcal_\rho \phi},
\]
whence taking the inverse Fourier transform, we obtain
\begin{equation*}%\label{eq:ftocsplitting}
 \phi = L *\Gcal_\rho \phi + T_M(\Gcal_\rho \phi).
\end{equation*}
By Theorem \ref{th:density} \ref{item:density1}, this identity can be extended to
\begin{equation}\label{eq:ftocsplitting2}
 u = L * D_\rho u + T_M (D_\rho u)  \ \ \text{for all $u \in H^{\rho}_0(\Omega)$}.
\end{equation}
Therefore, part \ref{item:Poincare1} follows with
\[
 C=  \norm{L}_{L^1(\O - \O_{\d},\R^n)} +  \norm{M}_{L^{\infty}(\R^n,\R^n)}.
\]
% $C= \norm{L}_{L^1(B_R(0);\R^n)}} + \norm{M}_{L^{\infty}(\R^n;\R^n)$ choosing $R>\diam(\Omega_\d)$.
\smallskip

Part \ref{item:Poincare2}. We note that the operator $U \mapsto L*U$ is compact from $L^2(\Omega_\d,\R^n)$ to $L^2 (\O)$, since $L$ is locally bounded (see, e.g., \cite[Prop.\ 4.7]{Conway90}); the precise definition of $L*U$ is
\[
  L*U (x) = \int_{\O_{\d}} L (x-y) \cdot U(y) \, dy , \qquad x \in \O .
\]
Moreover, as a consequence of the assumption in \ref{item:Poincare2} we have
\[
 \lim_{t\uparrow \infty} \frac{t^{n-1}}{\overline{\rho} (1/t)} = 0 ,
\]
which implies that
\begin{equation*}\label{eq:Minfty}
 \abs{M(\xi)} \to 0 \text{ as } \abs{\xi} \to \infty 
\end{equation*}
thanks to Lemma \ref{le:Qhatpositive}. 
Therefore, for any $U \in L^2(\Rn,\Rn)$ with $\norm{U}^2_{L^2(\Rn,\Rn)} \leq 1$, it holds that
\[
\lim_{R \to \infty} \int_{B_R^c}\abs{\widehat{T_M U}}^2 \,d\xi \leq \lim_{R \to \infty} \sup_{\abs{\xi}\geq R} \abs{M(\xi)}^2 =0.
\]
By a version of the Fr\'{e}chet-Kolmogorov criterion in Fourier space, cf.~\cite[Th.\ 3]{Peg85}, we deduce that $T_M : L^2(\R^n,\R^n) \to L^2(\Omega)$ is compact.

The compact embedding of $H^\rho_0(\Omega)$ into $L^2(\R^n)$ (or, equivalently, $L^2(\O)$) is concluded thanks to identity \eqref{eq:ftocsplitting2} and the fact that both operators $L * \cdot$ and $T_M$ are compact from $L^2(\Omega_\d,\R^n)$ to $L^2 (\O)$, where in the case of $T_M$ it is understood that $L^2(\Omega_\d,\R^n)$ is extended as zero to $L^2(\R^n,\R^n)$.
\end{proof}

\begin{example}\label{ex:l2}
Classes of kernels of compact support satisfying  \ref{itm:h0}--\ref{itm:h1}, and $\lim_{t \downarrow 0}t^{n-1}\overline{\rho}(t) = \infty$ are:
\begin{enumerate}[label=(\alph*)]

\item Given $0 < s < 1$,
\[
 \rho (x) = \frac{\mathbbm{1}_{B_1}(x)}{\abs{x}^{n+s-1}} .
\]

\item Given $0 \leq s < 1$,
\[
\rho(x) = \mathbbm{1}_{B_1}(x) \frac{- \log \abs{x}}{\abs{x}^{n+s-1}}.
\]

\item Given $0 < s < 1$ and $r>0$ sufficiently small,
\[
\rho(x) = \frac{\mathbbm{1}_{B_{r}(x)}}{\abs{x}^{n+s-1} (- \log \abs{x})}.
\]
\end{enumerate}
\end{example}

\begin{example}
A kernel of compact support satisfying  \ref{itm:h0}--\ref{itm:h1}, and $\lim_{t \downarrow 0}t^{n-1}\overline{\rho}(t) > 0$ is
\[
 \rho (x) = \frac{\mathbbm{1}_{B_1}(x)}{\abs{x}^{n-1}}.
\]
\end{example}

We will see in Proposition~\ref{pr:Poincareconverse}, that under some additional assumptions on $\rho$, the growth conditions at the origin of Theorem~\ref{th:l2bound} are sharp in order to obtain the validity of a Poincar\'{e} inequality or a compact embedding, respectively.

\subsection{Poincar\'{e} inequality and Compactness in $L^p$}\label{se:PoincareLp}

We derive here an analogue of Theorem~\ref{th:l2bound} in the $L^p$ setting, by applying the Mihlin-H\"{o}rmander theorem to show that the function $M$ in the proof of Theorem~\ref{th:l2bound} is an $L^p$ multiplier. This requires us to study also the decay of the derivatives of $\widehat{Q}_\rho$, and we impose the following assumption for that:
\begin{enumerate}[label = (H\arabic*)]
\setcounter{enumi}{1}
\item \label{itm:h2} The function $f_\rho$ is smooth in $(0,\infty)$, and for $t \in (0,\epsilon)$,
\[
-C \frac{d}{dt}f_\rho(t) \geq \frac{f_\rho(t)}{t} \quad \text{and} \quad \abslr{\frac{d^k}{dt^k}f_\rho(t)} \leq C_k \frac{f_\rho(t)}{t^k} \quad \text{for $k \in \N$}.
\]
\end{enumerate}
The first condition is equivalent to the property that $t \mapsto t^{\nu}f_\rho(t)$ is decreasing on $(0,\epsilon)$ for some $\nu>0$, cf.~Remark~\ref{rem:h0h1relation}, while the second imposes that $f_\rho$ does not oscillate too much.
%See Example~\ref{ex:one} for functions that satisfy this condition.
A consequence of \ref{itm:h2} that will be repeatedly used is
\begin{equation}\label{eq:H2cons}
 \abslr{\frac{d^l}{dt^l}(t^kf_\rho(t))} \leq C_k t^{k-l} f_\rho(t), \qquad 0 \leq l \leq k , \quad t \in (0, \e) ,
\end{equation}
as well as 
\begin{equation}\label{eq:H2cons2}
 \abslr{\frac{d^{k+1}}{dt^{k+1}}(t^k f_\rho(t))} \leq - C_k \frac{d}{dt}f_\rho(t) , \qquad k \in \N , \quad t \in (0, \e) .
\end{equation}
The following consequence is also useful.

\begin{lemma}\label{le:t2f}
Let $\rho$ satisfy \ref{itm:h0} and \ref{itm:h2}.
Then, $\lim_{t \downarrow 0} t^2 f_\rho(t)=0$.
\end{lemma}
\begin{proof}
As a consequence of \eqref{eq:H0corbarrho}, we have $\liminf_{t \downarrow 0} t^2 f_\rho(t) = 0$. Let $0 < t < \e$. By the fundamental theorem of calculus,
\begin{equation}\label{eq:t2f}
 \e^2 f_\rho(\e) - t^2 f_\rho(t) = \int_t^{\e} \frac{d}{dr}(r^2 f_\rho(r)) \, dr .
\end{equation}
Now, by \eqref{eq:H2cons} and \eqref{eq:H0corbarrho}
\[
 \int_0^{\e} \abslr{\frac{d}{dr}(r^2 f_\rho(r))} dr \leq C_2 \int_0^{\e} r f_\rho(r) \, dr < \infty .
\]
Therefore, the limit when $t \downarrow 0$ of the right-hand side of \eqref{eq:t2f} exists, and, consequently, so does the limit of the left-hand side, which proves the result.
\end{proof}

\begin{example}\label{ex:H2}
Classes of kernels $\rho$ satisfying \ref{itm:h0}--\ref{itm:h2} are:
\begin{enumerate}[label=(\alph*)]
\item $\rho$ of Example \ref{ex:H0} \ref{item:Riesz}.

%\color{red}\item $\rho$ of Example \ref{ex:H0H1} \ref{item:alphabeta}. NOT SMOOTH\color{black}

\item\label{itm:exH2c} Given $0 \leq s < 1$ and a non-negative radial function $\chi \in C_c^{\infty}(\R^n)$ with $\chi(0)>0$ and $\chi(x)/\abs{x}^{1+s}$ radially decreasing,
\[
\rho(x) = \frac{\chi(x)}{\abs{x}^{n+s-1}}.
\]

\item\label{itm:exH2d} Given $0 \leq s < 1$ and a non-negative, radial function $\chi \in C_c^{\infty}(B_1)$ with $\chi(0)>0$ and $\chi(x)(-\log\abs{x})/\abs{x}^{1+s}$ radially decreasing,
\[
 \rho(x) = \frac{\chi(x) (- \log \abs{x})}{\abs{x}^{n+s-1}}.
\]
Indeed, \ref{itm:h0} and \ref{itm:h1} are simple to verify, whereas \ref{itm:h2} follows since the derivatives of $\log$ behave similarly as a power function.

\item\label{itm:exH2e} Given $0 \leq s < 1$ and a non-negative, radial function $\chi \in C_c^{\infty}(B_{1})$ with $\chi(0)>0$ and $\chi(x)/(-\abs{x}^{1+s}\log\abs{x})$ radially decreasing,
\[
 \rho (x) = \frac{\chi(x)}{\abs{x}^{n+s-1} (- \log \abs{x})}  .
\]
The verification of \ref{itm:h0}--\ref{itm:h2} is similar to that of the previous example.

\item\label{itm:exH2f} Given a smooth $s:[0,\infty) \to (0,1)$ and a non-negative, radial function $\chi \in C_c^{\infty}(\R^n)$ with $\chi(0)>0$ and $\chi(x)/\abs{x}^{1+s(\abs{x})}$ radially decreasing,
\[
 \rho(x) = \frac{\chi(x)}{\abs{x}^{n+ s(\abs{x})-1}}.
\]
Again, \ref{itm:h0} and \ref{itm:h1} follow readily, whereas for \ref{itm:h2} we first note that
\[
\frac{d}{dt}f_\rho(t) = \overline{\chi}'(t)\frac{1}{t^{1+ s(t)}}+\overline{\chi}(t)\left(\frac{-(1+ s(t))}{t^{2+ s(t)}}+\frac{-\log(t) s'(t)}{t^{1+ s(t)}}\right),
\]
which satisfies $\abs{\frac{d}{dt}f_\rho(t)} \leq Cf_\rho(t)/t$ since $t/\log t$ is locally bounded. The other derivatives can be bounded in a similar way, so \ref{itm:h2} holds.

\item If $\rho_1, \rho_2$ satisfy \ref{itm:h0}--\ref{itm:h2} and $\alpha_1, \alpha_2>0$ then $\alpha_1 \rho_1 + \alpha_2 \rho_2$ satisfies \ref{itm:h0}--\ref{itm:h2}.

\end{enumerate}
\end{example}

\begin{lemma}\label{le:Qhatderivdecay}
Let $\rho$ have compact support and satisfy \ref{itm:h0}--\ref{itm:h2}. Then, for every $\alpha \in \N^n$,
\[
\abs{\partial^{\alpha}\widehat{Q}_\rho(\xi)} \leq C_{\alpha} \left(\abs{\xi}^{-\abs{\alpha}}\widehat{Q}_\rho(\xi)+\abs{\xi}^{-\abs{\alpha}-1}\right) , \qquad \abs{\xi} \geq 1 .
\]
\end{lemma}
\begin{proof}
\textit{Step 1: Integral bounds.}
We show that for all $\theta > 1/\epsilon$ and $k \in \N$,
\begin{equation}\label{eq:intbounds}
\abslr{\int_0^\infty r^{k}f_\rho(r)a_k(2 \pi \theta r)\,dr} \leq \frac{C_k}{\theta^k}\int_0^\infty f_\rho(r)\sin(2 \pi \theta r)\,dr+\frac{C_k}{\theta^k},
\end{equation}
where $a_k =\cos$ when $k$ is odd and $a_k = \sin$ when $k$ is even.

Using integration by parts $k$ times we obtain
\begin{equation}\label{eq:intbypartsk}
\int_0^\infty r^k f_\rho(r)a_k(2 \pi \theta r)\,dr = \frac{D_k}{\theta^k}\int_{0}^\infty \frac{d^k}{dr^k}(r^kf_\rho(r)) \sin(2\pi\theta r)\,dr,
\end{equation}
%\begin{equation}\label{eq:intbypartsk}
%\begin{split}
%\int_0^\infty r^k f_\rho(r)a_k(2 \pi \abs{\xi} r z_1)\,dr&=\lim_{\lambda \downarrow 0} \int_{\lambda}^\infty r^k f_\rho(r)a_k(2 \pi \abs{\xi} r z_1)\,dr\\
%&= \frac{C_k}{\theta^k}\lim_{\lambda \downarrow 0} \int_{\lambda}^\infty \frac{d^k}{dr^k}(r^kf_\rho(r)) \sin(2\pi\theta r)\,dr + \sum_{l=0}^{k-1} \frac{d^l}{dr^l}\left(r^kf_\rho(r)\right) a_{k-1-l}(2\pi\theta r)|_{r=\lambda}\\
%&= \frac{C_k}{\theta^k}\int_{\R} \frac{d^k}{dr^k}(r^kf_\rho(r)) \sin(2\pi\theta r)\,dr,
%\end{split}
%\end{equation}
where the constants $D_k$ may be negative.
Indeed, equality \eqref{eq:intbypartsk} follows directly by integration by parts except possibly for the boundary terms.
They turn out to be zero since $\rho$ has compact support, $\lim_{t \downarrow 0}\frac{d^l}{dt^l}(t^kf_\rho(t))=0$ when $l \leq k-2$ and $\lim_{t \downarrow 0}\frac{d^l}{dt^l}(t^kf_\rho(t))\sin(2\pi\theta t) =0$ when $l=k-1$ by \eqref{eq:H2cons} and Lemma \ref{le:t2f}; this also shows that the intermediate integrals leading to \eqref{eq:intbypartsk} in the induction process are finite.
The final integral on the right-hand side of \eqref{eq:intbypartsk} is also finite by \eqref{eq:H2cons}, \eqref{eq:H0corbarrho} and the compact support of $\rho$.
Setting $N=\lfloor \theta \epsilon \rfloor \geq \theta \epsilon /2$ (where $\floor{\cdot}$ denotes the integer part) and $b_k(r)=\frac{d^k}{dr^k}(r^kf_\rho(r))$, we estimate the right-hand side of \eqref{eq:intbypartsk} as follows:
\begin{equation}\label{eq:estimatedkdtk}
 \abslr{\int_{0}^\infty \frac{d^k}{dr^k}(r^kf_\rho(r)) \sin(2\pi\theta r)\,dr} \leq  \abslr{\int_0^{\frac{N}{\theta}}b_k(r) \sin(2\pi\theta r)\,dr} + \int_{\epsilon/2}^\infty\abslr{\frac{d^k}{dr^k}(r^kf_\rho(r))} dr ,
\end{equation}
since $N \geq \theta \epsilon /2$.
On the one hand,
\[
 \int_{\epsilon/2}^\infty\abslr{\frac{d^k}{dr^k}(r^kf_\rho(r))}\,dr \leq C_k \int_{\epsilon/2}^\infty f_\rho(r) \,dr \leq C_k ,
\]
by \eqref{eq:H2cons} and \eqref{eq:H0corbarrho}.
On the other hand, as in \eqref{eq:oscillatory1} and \eqref{eq:oscillatory2},
\begin{equation}\label{eq:oscilatorybk}
 \int_0^{\frac{N}{\theta}}b_k(r) \sin(2\pi\theta r)\,dr = \sum_{j=0}^{N-1} \int_{\frac{j}{\theta}}^{\frac{j+ \frac{1}{2}}{\theta}} \left(b_k(r)-b_k\left(r+\frac{1}{2\theta}\right)\right)\sin(2\pi\theta r)\,dr .
\end{equation}
By the fundamental theorem of calculus,
\begin{align*}
\abslr{ b_k(r)-b_k\left(r+\frac{1}{2\theta}\right)} &\leq \int_{r}^{r+\frac{1}{2\theta}} \abslr{\frac{d}{dt}b_k(t)}\,dt \\&\leq C_k \int_{r}^{r+\frac{1}{2\theta}} -\frac{d}{dt}f_\rho(t)\,dt =C_k \left( f_\rho (r) - f_\rho \left(r+\frac{1}{2\theta}\right) \right),
\end{align*}
where we have used \eqref{eq:H2cons2}.
Thus,
\[
  \sum_{j=0}^{N-1} \int_{\frac{j}{\theta}}^{\frac{j+ \frac{1}{2}}{\theta}} \abslr{b_k(r)-b_k\left(r+\frac{1}{2\theta}\right)}\sin(2\pi\theta r)\,dr \leq C_k \int_0^{\frac{N}{\theta}} f_\rho (r) \sin(2\pi\theta r)\,dr ,
\]
reasoning as in \eqref{eq:oscilatorybk}.
Putting together the formulas following \eqref{eq:estimatedkdtk}, we obtain
\[
 \abslr{\int_{0}^\infty \frac{d^k}{dr^k}(r^kf_\rho(r)) \sin(2\pi\theta r)\,dr} \leq C_k \int_{0}^{\frac{N}{\theta}}f_\rho(r)\sin(2\pi\theta r)\,dr + C_k .
\]
Together with \eqref{eq:intbypartsk}, the bounds \eqref{eq:intbounds} follow.

\smallskip

\textit{Step 2: Conclusion.} Recalling formula \eqref{eq:Qhatrho}, we obtain, from Leibniz' rule and interchanging the derivative with the integral, that 
\begin{equation}\label{eq:qhatderiv}
\abs{\partial^{\alpha} \widehat{Q}(\xi)} \leq C_{\alpha} \sum_{k=0}^{\abs{\alpha}} \frac{1}{\abs{\xi}^{\abs{\alpha}+1-k}}\abslr{\int_{\R^n} \frac{\rho(x)x_1^{k+1}}{\abs{x}^2}a_k(2\pi\abs{\xi}x_1)\,dx}, \qquad \xi \in \Rn \setminus \{ 0 \} .
\end{equation}
When $k=0$, the term in the sum is equal to $2\pi\abs{\xi}^{-\abs{\alpha}}\widehat{Q}_\rho(\xi)$, cf.~\eqref{eq:Qhatrho}, which already has the correct form. For $k \geq 1$, we can compute as in \eqref{eq:hatQcoarea} that
\[
 \int_{\R^n} \frac{\rho(x)x_1^{k+1}}{\abs{x}^2}a_k(2\pi\abs{\xi}x_1)\,dx = 2 \int_{\Sn_+} z_1^{k+1} \int_0^\infty r^{k}f_\rho(r)a_k(2\pi\abs{\xi}z_1r)\,dr\,d \mc{H}^{n-1} (z) .
\]
We split the integral in $\Sn_+$ into $\{z_1\abs{\xi} > 1/\epsilon\}$ and $\{z_1\abs{\xi} \leq 1/\epsilon\}$.
In the first subset we have, thanks to \eqref{eq:intbounds} with $\theta=\abs{\xi}z_1$,
\begin{align*}
 &\abslr{\int_{\Sn_+ \cap \{z_1\abs{\xi} > 1/\epsilon\}} z_1^{k+1}\int_0^\infty r^{k}f_\rho(r)a_k(2\pi\abs{\xi}z_1r)\,dr\,d \mc{H}^{n-1} (z)} \\
 & \leq C_k \abslr{\int_{\Sn_+ \cap \{z_1\abs{\xi} > 1/\epsilon\}} \frac{z_1^{k+1}}{(\abs{\xi}z_1)^{k}}\left(\int_0^\infty f_\rho(r)\sin(2\pi\abs{\xi}z_1r)\,dr+1\right)\,d \mc{H}^{n-1} (z)} \\
 & \leq \frac{C_k}{\abs{\xi}^k}\abslr{\int_{\Sn_+} z_1\int_0^\infty f_\rho(r)\sin(2\pi\abs{\xi}z_1r)\,dr\,d \mc{H}^{n-1} (z)} + \frac{C_k}{\abs{\xi}^{k}} = \frac{C_k}{\abs{\xi}^{k-1}}\widehat{Q}_\rho(\xi) + \frac{C_k}{\abs{\xi}^{k}} ,
\end{align*}
where the last equality is due to \eqref{eq:hatQcoarea}.
In the second subset we have, for $\abs{\xi} \geq 1$ and $k \geq 1$,
\begin{align*}
& \abslr{\int_{\Sn_+ \cap \{z_1\abs{\xi} \leq 1/\epsilon\}} z_1^{k+1}\int_0^\infty r^{k}f_\rho(r)a_k(2\pi\abs{\xi}z_1r)\,dr\,d \mc{H}^{n-1} (z)} \\
& \leq \abslr{\int_{\Sn_+ \cap \{z_1\abs{\xi} \leq 1/\epsilon\}} \frac{1}{(\abs{\xi}\epsilon)^{k+1}}\int_0^\infty r^{k}f_\rho(r)\,dr\,d \mc{H}^{n-1} (z)} \leq \frac{C_k}{\abs{\xi}^{k+1}} \leq \frac{C_k}{\abs{\xi}^{k}}, 
\end{align*}
where we have used \eqref{eq:H0corbarrho} and the compact support of $\rho$.
%For $k= 0$ and $\abs{\xi} \geq 1$, the above inequality together with $\sin t \leq t$ (for $t \geq 0$) turns into
%\begin{equation*}
% \abslr{\int_{\Sn_+ \cap \{z_1\abs{\xi} \leq 1/\epsilon\}} z_1 \int_0^\infty f_\rho(r) a_0 (2\pi\abs{\xi}z_1r)\,dr\,d \mc{H}^{n-1} (z)} \leq \frac{\sigma_{n-1} \pi}{\epsilon^2 \abs{\xi}} \int_0^\infty r f_\rho(r)\,dr \leq C_0 ,
%\end{equation*}
%thanks to \eqref{eq:H0corbarrho} and the compact support of $\rho$.
All in all, we have shown
\[
 \abslr{\int_{\R^n} \frac{\rho(x)x_1^{k+1}}{\abs{x}^2}a_k(2\pi\abs{\xi}x_1)\,dx} \leq \frac{C_k}{\abs{\xi}^{k-1}}\widehat{Q}_\rho(\xi) + \frac{C_k}{\abs{\xi}^{k}} , \qquad \abs{\xi} \geq 1 , \quad k \geq 0 ,
\]
which yields the result thanks to \eqref{eq:qhatderiv}.
\end{proof}
%\begin{remark}
%Condition (H2) does not appear to be strictly necessary for Lemma~\ref{le:Qhatderivdecay} to hold, since it rules out the $\rho$ in Example~\ref{ex:l2}. However, condition (H3) probably needs to be expanded for this to be possible since we want to conclude that
%\[
%\abslr{\int_{0}^{\infty} \frac{d^k}{dr^k}(r^kf_\rho(r)) \sin(2\pi\theta r)\,dr} \leq C \int_0^\infty f_\rho(r)\sin(2 \pi \abs{\xi} r z_1)\,dr,
%\]
%and just having $\abs{\frac{d^k}{dr^k}(r^kf_\rho(r))} \leq f_\rho(r)$ may not be enough without (H2). 
%\end{remark}

With these bounds we can apply the Mihlin-H\"{o}rmander theorem to prove the analogue of Theorem~\ref{th:l2bound} in the $L^p$-setting.

\begin{theorem}\label{th:lpbound}
Let $p \in (1,\infty)$.
Let $\Omega \subset \R^n$ be open and bounded and suppose that $\rho$ satisfies \ref{itm:h0}--\ref{itm:h2} and has compact support. Then, the following two statements hold:
\begin{enumerate}[label = (\roman*)]
\item\label{item:Poincarep1} If $\liminf_{t \downarrow0} t^{n-1}\overline{\rho}(t)>0$, then there is a $C=C(\Omega,n,\rho)>0$ such that
\[
\norm{u}_{L^p(\Omega)} \leq C \norm{D_\rho u}_{L^p(\R^n,\R^n)} \quad \text{for all $u \in H^{\rho,p}_0(\Omega)$}.
\]
\item\label{item:Poincarep2} If $\lim_{t \downarrow 0}t^{n-1}\overline{\rho}(t)=\infty$, then $H^{\rho,p}_0(\Omega)$ is compactly embedded into $L^p(\R^n)$.
\end{enumerate}
\end{theorem}
\begin{proof}
Part \ref{item:Poincarep1}. In light of Lemma~\ref{le:Qhatpositive} and \eqref{eq:limsuptinfty}, we find that
\begin{equation}\label{eq:hatQxi-1}
 \widehat{Q}_\rho(\xi) \geq C/\abs{\xi} \qquad \text{for } \abs{\xi} \geq 1 ,
\end{equation}
so it follows from Lemma~\ref{le:Qhatderivdecay} that for every $\alpha \in \N^n$,
\[
\abs{\partial^{\alpha}\widehat{Q}_\rho(\xi)} \leq C_\alpha \abs{\xi}^{-\abs{\alpha}}\widehat{Q}_\rho(\xi) \quad \text{for $\abs{\xi} \geq 1$}.
\]
Applying Fa\`a di Bruno's formula for the derivatives of a composition,
\[
 \abs{\partial^{\alpha}\widehat{Q}_{\rho}^{-1}} \leq C_\alpha \sum_{k=1}^{\abs{\alpha}} \widehat{Q}_\rho^{-1-k} \sum_{j_1, \ldots , j_{\abs{\a} - k + 1}} B_{j_1, \ldots , j_{\abs{\a} - k + 1}} ,
\]
where the indices $j_1, \ldots , j_{\abs{\a} - k + 1} \in \N$ in the sum are taken with the restrictions
\[
 j_1 + \cdots + j_{\abs{\a} - k + 1} = k \quad \text{and} \quad j_1 + \cdots + (\abs{\a} - k + 1) j_{\abs{\a} - k + 1} = \abs{\a} ,
\]
and $B_{j_1, \ldots , j_{\abs{\a} - k + 1}}$ is a symbol for all products of the absolute value of the partial derivatives of $\widehat{Q}_{\rho}$ with $j_i$ derivatives of order $i$, for $i=1, \ldots, \abs{\a} - k + 1$.
Thus, 
\[
 B_{j_1, \ldots , j_{\abs{\a} - k + 1}} \leq C_{\a} \parent{\abs{\xi}^{-1} \widehat{Q}_{\rho}}^{j_1} \cdots \parent{\abs{\xi}^{-(\abs{\a}-k+1)} \widehat{Q}_{\rho}}^{j_{\abs{\a}-k+1}} .
\]
Hence, for $\abs{\xi}\geq 1$, using \eqref{eq:hatQxi-1},
\[
 \abs{\partial^{\alpha}\widehat{Q}_{\rho}^{-1}} \leq C_{\a} \abs{\xi}^{-\abs{\a}} \widehat{Q}_{\rho}^{-1} \leq C_{\a} \abs{\xi}^{1-\abs{\a}} .
\]
Now, defining $R(\xi) = \xi \abs{\xi}^{-2}$, we have that
\[
 \abs{\partial^{\alpha} R(\xi)} \leq C_\alpha \abs{\xi}^{-1-\abs{\alpha}} .
\]

Let $\chi \in C^{\infty}_c (\Rn)$ and $M \in C^{\infty}(\R^n,\R^n)$ be as in the proof of Theorem~\ref{th:l2bound}, so that identity \eqref{eq:M} holds.
With the calculations above, thanks to Leibniz' rule we can estimate for all $\alpha \in \N^n$ that
\begin{equation}\label{eq:pM}
\abs{\partial^{\alpha} M(\xi)} \leq C_{\alpha}\abs{\xi}^{-1-\abs{\alpha}} \widehat{Q}_{\rho}^{-1} \leq C_{\alpha}\abs{\xi}^{-\abs{\alpha}}, \quad \text{for $\abs{\xi}\geq 1$}.
\end{equation}
As such, the Mihlin-H\"{o}rmander theorem (cf.~\cite[Th.\ 6.2.7]{Gra14a}) implies that the operator $T_M$ of \eqref{eq:TM} can be extended to a bounded operator from $L^p(\R^n,\R^n)$ to $L^p(\R^n)$. The Poincar\'{e} inequality can now be argued as in Theorem~\ref{th:l2bound}. \smallskip

Part \ref{item:Poincarep2}. We saw in the proof of Theorem~\ref{th:l2bound} \ref{item:Poincare2} that the operator $T_M$ is compact from $L^2(\Omega_\d,\R^n)$ to $L^2(\Omega)$, since $\lim_{t \downarrow 0}t^{n-1}\overline{\rho}(t)=\infty$.
As a consequence of the proof of part \ref{item:Poincarep1}, $T_M$ is bounded from $L^p(\Omega_\d,\R^n)$ to $L^p(\Omega)$ for all $p \in (1,\infty)$ and, hence, by Krasnoselskii's interpolation theorem \cite[Th.\ IV.2.9]{BeS88}, also compact for all $p \in (1,\infty)$. The compact embedding now follows as in the proof of Theorem~\ref{th:l2bound} \ref{item:Poincare2}.
\end{proof}

\begin{example}\label{ex:one}
Let $p \in (1,\infty)$ and let $\Omega \subset \R^n$ be open and bounded.

\begin{enumerate}[label=(\alph*)]

\item
Let $0 \leq s < 1$, $\chi \in C_c^{\infty}(\R^n)$ and $\rho$ be as in Example \ref{ex:H2} \ref{itm:exH2c}.
By Theorem \ref{th:lpbound}, if $s=0$ we have a Poincar\'{e} inequality, while for $s>0$ we obtain compactness.
This constitutes an alternative proof as well as a generalization of the results \cite[Thms.\ 6.2 and 7.3]{BeCuMo22}, where it was assumed that $\chi$ is constant around the origin and radially decreasing.

\item Let $0 \leq s < 1$ and $\rho$ be as in Example \ref{ex:H2} \ref{itm:exH2d}.
By Theorem \ref{th:lpbound}, we have both a Poincar\'{e} inequality and compactness.

\item
Let $0 \leq s < 1$ and $\rho$ be as in Example \ref{ex:H2} \ref{itm:exH2e}.
By Theorem \ref{th:lpbound}, if $s>0$ we have both a Poincar\'{e} inequality and compactness.

\item
Let $s:[0,\infty) \to (0,1)$ and $\rho$ be as in Example \ref{ex:H2} \ref{itm:exH2f}.
By Theorem \ref{th:lpbound}, we have both a Poincar\'{e} inequality and compactness.

\end{enumerate}
\end{example}

\section{Fundamental theorem of calculus}\label{se:fundamental}

In this section, we will study when the inverse Fourier transform of $W_\rho$ (see~\eqref{eq:vrhofourier}) is a locally integrable function, and identify its behavior at the origin. This function will then give us an analogue of the fundamental theorem of calculus for the nonlocal gradient; precisely, applying the inverse Fourier transform to \eqref{eq:ftoctempered} yields $u = V_{\rho} * \mc{G}_{\rho} u$ with $V_\rho:=W_\rho^{\vee}$.

For $W_\rho^{\vee}$ to be a locally integrable function, we require the following mild assumptions, saying that the kernel $\rho$ lies between two fractional kernels:
\begin{enumerate}[label = (H\arabic*)]
\setcounter{enumi}{2}

\item \label{itm:h3} The function $g_\rho:(0,\infty) \to \R, \ t \mapsto t^{n+\sigma-1}\overline{\rho}(t)$ is almost decreasing on $(0,\epsilon)$ for some $\sigma \in (0,1)$; 

\item \label{itm:h4} the function $h_\rho:(0,\infty) \to \R, \ t \mapsto t^{n+\gamma-1}\overline{\rho}(t)$ is almost increasing on $(0,\epsilon)$ for some $\gamma \in (0,1)$.
\end{enumerate}

Recall from the notation section the definition of \emph{almost decreasing/increasing}.
We will use, as a consequence of \ref{itm:h3} and Lemma~\ref{le:Qhatpositive} that
\begin{equation}\label{eq:tn-1r}
\widehat{Q}_\rho(\xi) \geq C \abs{\xi}^{\sigma-1}, \quad \abs{\xi} \geq 1.
\end{equation}
Note also that \ref{itm:h3}--\ref{itm:h4} require $\sigma \leq \gamma$; this can be seen directly or invoking \eqref{eq:doubling} below.

\begin{example}\label{ex:H4}
Classes of kernels $\rho$ satisfying \ref{itm:h0}--\ref{itm:h4} are:
\begin{enumerate}[label=(\alph*)]
\item\label{itm:exH4a} Given $s \in (0,1)$,
\[
 \rho (x) = \frac{1}{\abs{x}^{n+s-1}} .
\]
In this case, one can take $\sigma = \gamma = s$.
The same conclusion holds for 
\begin{equation}\label{eq:truncatedRiesz}
 \rho(x) = \frac{\chi(x)}{\abs{x}^{n+s-1}}
\end{equation}
with $\chi \in C_c^{\infty}(\R^n)$ a non-negative radial function such that $\chi(0)>0$ and $\chi(x)/\abs{x}^{1+s}$ is radially decreasing.

\item\label{itm:exH4b} Given $s \in (0,1)$,
\begin{equation}\label{eq:truncatedlog}
 \rho(x) = \frac{\chi(x) (- \log \abs{x})}{\abs{x}^{n+s-1}}
\end{equation}
with $\chi \in C_c^{\infty}(B_1)$ a non-negative radial function such that $\chi(0)>0$ and $\chi(x)(-\log\abs{x})/\abs{x}^{1+s}$ is radially decreasing. In this case one can take $\sigma = s$ and any $\gamma \in (s, 1)$.

\item\label{itm:exH4c} Given $s \in (0,1)$,
\begin{equation}\label{eq:truncatedlogdown}
 \rho(x) = \frac{\chi(x)}{\abs{x}^{n+s-1}(- \log \abs{x})}
\end{equation}
with $\chi \in C_c^{\infty}(B_1)$ a non-negative radial function such that $\chi(0)>0$ and $\chi(x)/(-\abs{x}^{1+s}\log\abs{x})$ is radially decreasing. In this case one can take any $\sigma  \in(0,s)$ and $\gamma=s$.

\item Given a smooth $s:[0,\infty) \to (0,1)$ and a non-negative radial function $\chi \in C_c^{\infty}(\R^n)$ such that $\chi(0)>0$ and $\chi(x)/\abs{x}^{1+s(\abs{x})}$ is radially decreasing,
\[
 \rho(x) = \frac{\chi(x)}{\abs{x}^{n+ s(\abs{x})-1}}.
\]
In this case, one can take $\sigma = \min_{[0, \e]} s$ and $\gamma = \max_{[0, \e]} s$ for any $\epsilon>0$.

\item If $\rho_1, \rho_2$ satisfy \ref{itm:h0}--\ref{itm:h4} and $\alpha_1, \alpha_2 >0$ then $\alpha_1 \rho_1 + \alpha_2 \rho_2$ satisfies \ref{itm:h0}--\ref{itm:h4}.
In fact, let $\sigma_1, \sigma_2$ be the exponents of \ref{itm:h3} for $\rho_1, \rho_2$, respectively; let $\gamma_1, \gamma_2$ be the exponents of \ref{itm:h4} for $\rho_1, \rho_2$, respectively.
Then, \ref{itm:h3} holds for $\alpha_1 \rho_1 + \alpha_2 \rho_2$ with the exponent $\min \{ \sigma_1, \sigma_2 \}$, while \ref{itm:h4} holds with the exponent $\max \{ \gamma_1, \gamma_2 \}$.

\end{enumerate}
\end{example}

In the following result, the proof that $V_\rho$ is a function is adapted from that of \cite[Th.\ 5.9]{BeCuMo22}, while the bounds around the origin on $V_\rho$ require different arguments since we cannot compare it with the Riesz kernel.

\begin{theorem}\label{th:NTFC}
Let $\rho$ have compact support and satisfy \ref{itm:h0}--\ref{itm:h4}. Then, there exists a vector radial function $V_\rho \in C^{\infty}(\R^n\setminus \{0\},\R^n) \cap L^1_{\loc} (\R^n,\R^n)$ such that for all $\phi \in C_c^{\infty}(\R^n)$,
\begin{equation}\label{eq:ftoc}
\phi(x) = \int_{\R^n} V_\rho(x-y) \cdot\Gcal_\rho \phi(y)\,dy \quad \text{for all $x \in \R^n$}.
\end{equation}
Moreover, there is a constant $C=C(n,\rho)>0$ such that for $x \in B_{\epsilon} \setminus\{0\}$,
\begin{equation}\label{eq:Vrhobound}
\abs{V_\rho(x)}\leq \frac{C}{\abs{x}^{2n-1}\rho(x)} \quad \text{and} \quad \abs{\nabla V_\rho(x)} \leq \frac{C}{\abs{x}^{2n}\rho(x)}.
\end{equation}
\end{theorem}
\begin{proof}
In light of Proposition~\ref{pr:Wrho} and the well-known interaction between Fourier transforms and multiplication and convolution, it suffices to show that the inverse Fourier transform $V_\rho:=W_\rho^{\vee}$ agrees with a locally integrable function with the stated properties. \smallskip

\textit{Step 1: $V_\rho$ is a function.} Let $\chi \in C_c^{\infty}(\R^n)$ be a radial cut-off function with $\chi \equiv 1$ on $B_{2/\epsilon}$.
As in the proof of Theorem \ref{th:l2bound} \ref{item:Poincare1}, we can write
\[
 W_\rho = \chi W_\rho + (1-\chi)W_\rho=:W_\rho^1 + W_\rho^2.
\]
Since $W_\rho^1$ has compact support, it follows from the Paley-Wiener theorem that $(W_\rho^1)^{\vee}$ is analytic. We also observe that $W_\rho^2$ is actually a smooth locally integrable function, namely,
\[
W^2_\rho(\xi) = (1-\chi(\xi))\frac{-i\xi}{2\pi\abs{\xi}^2\widehat{Q}_\rho(\xi)} ,
\]
as in \eqref{eq:M}.
From Lemma~\ref{le:Qhatderivdecay} and its consequence \eqref{eq:pM}, and \eqref{eq:tn-1r} we obtain that for any $\alpha \in \N^n$,
\begin{equation}\label{eq:w2bound}
\abs{\partial^{\alpha} W^2_\rho(\xi)} \leq \frac{C_{\alpha}}{(1+\abs{\xi}^{1+\abs{\alpha}}) \widehat{Q}_{\rho} (\xi)} \leq \frac{C_{\alpha}}{1+\abs{\xi}^{\abs{\alpha}+\sigma}}, \quad \xi \in \R^n.
\end{equation}
If $\abs{\alpha} \geq n + m$ for some $m \in \N$, we deduce from \cite[Exercise~2.4.1]{Gra14a} that 
\begin{equation}\label{eq:Kform}
(\partial^{\alpha} W^2_\rho)^{\vee} = (-2\pi i \cdot)^{\alpha} (W^2_\rho)^{\vee},
\end{equation}
lies in $C^m(\R^n)$. In particular, we find that $(W^2_\rho)^{\vee}$ coincides with a smooth function $K:\R^n \setminus\{0\} \to \R^n$ outside the origin. We show in step 2 below that $K$ is integrable, which implies that
\[
(W^2_\rho)^{\vee} = K + V_0,
\]
where $V_0$ is a tempered distribution supported at the origin. Therefore, by \cite[Prop.\ 2.4.1]{Gra14a}, $V_0$ can be written as a linear combination of derivatives of Dirac deltas, i.e.,
\[
V_0 = \sum_{\abs{\alpha} \leq k} c_\alpha \partial^{\alpha} \delta_0 \quad \text{and} \quad \widehat{V_0} = \sum_{\abs{\alpha} \leq k} c_\alpha (2\pi i \cdot)^{\alpha}
\]
for some $k \in \N$ and $c_{\a} \in \C$.
We then obtain that
\[
W^2_\rho = \widehat{K} + \sum_{\abs{\alpha} \leq k} c_\alpha (2\pi i \cdot)^{\alpha}.
\]
However, $W^2_\rho$ vanishes at infinity (thanks to \eqref{eq:w2bound}) and so does $\widehat{K}$ as the Fourier transform of an integrable function, so we must have $c_\alpha = 0$ for all $\alpha \in \N^n$ with $\abs{\alpha} \leq k$.
This shows that $(W^2_\rho)^{\vee} = K$, and hence
\[
V_\rho = W_\rho^{\vee} = (W^1_\rho)^{\vee} + (W^2_\rho)^{\vee} \in C^{\infty}(\R^n\setminus\{0\};\R^n),
\]
is a locally integrable function.
Note that $V_\rho$ is real-valued and vector radial since $W_\rho$ is imaginary-valued and vector radial. \smallskip

\textit{Step 2: $K$ is integrable.} By choosing $\alpha = (n,0,\ldots,0)$ and $\alpha = (n+1,0,\cdots,0)$ in~\eqref{eq:Kform} and using symmetry considerations, we have for $x \not = 0$ that
\[
\abs{K(x)} = \abs{K(\abs{x}e_1)} = \frac{\abs{(\partial^{n}_1 W^2_\rho)^{\vee}(\abs{x}e_1)}}{(2\pi  \abs{x})^{n}} \quad \text{and} \quad \abs{K(x)} = \abs{K(\abs{x}e_1)} = \frac{\abs{(\partial^{n+1}_1 W^2_\rho)^{\vee}(\abs{x}e_1)}}{(2\pi  \abs{x})^{n+1}}.
\]
The function $(\partial^{n+1}_1 W^2_\rho)^{\vee}$ is bounded since $\partial^{n+1}_1 W^2_\rho$ is integrable by \eqref{eq:w2bound}; consequently, we deduce from the second identity that
\[
\abs{K(x)} \leq \frac{C}{\abs{x}^{n+1}} \quad x \in \R^n\setminus\{0\},
\]
which shows that $K$ is integrable away from the origin.
Near the origin, we use that $\partial^{n}_1 W^2_\rho$ is also integrable (cf.~\eqref{eq:w2bound}). Hence, we may utilize the standard formula of the Fourier transform and partial integration for $0<\abs{x} < 2/\epsilon$, to find
\begin{align*}
\abslr{(\partial^{n}_1 W^2_\rho)^{\vee}(\abs{x}e_1)} &= \abslr{\int_{\R^n} \partial^n_1 W^2_\rho(\xi)e^{2\pi i \abs{x}\xi_1}\,d\xi}\\
&\leq 2\pi\abs{x}\abslr{\int_{B_{1/\abs{x}}\setminus B_{2/\epsilon}} \partial^{n-1}_1 W^2_\rho(\xi) e^{2\pi i \abs{x}\xi_1}\,d\xi} \\
&\quad + \abslr{\int_{\partial B_{1/\abs{x}}} \partial^{n-1}_1 W^2_\rho(\xi) e^{2\pi i \abs{x}\xi_1}\,d\Hcal^{n-1}(\xi)}
+ \abslr{\int_{B_{1/\abs{x}}^c} \partial^n_1 W^2_\rho(\xi)e^{2\pi i \abs{x}\xi_1}\,d\xi} \\
&\leq C\abs{x} \int_{B_{1/\abs{x}}\setminus B_{2/\epsilon}} \frac{d\xi}{\abs{\xi}^{n}\widehat{Q}_\rho(\xi)} + C\int_{\partial B_{1/\abs{x}}} \frac{d\Hcal^{n-1}(\xi)}{\abs{\xi}^{n}\widehat{Q}_\rho(\xi)}
 + C\int_{B_{1/\abs{x}}^c} \frac{d\xi}{\abs{\xi}^{n+1}\widehat{Q}_\rho(\xi)} .
\end{align*}
The first inequality is integration by parts in $B_{1/\abs{x}}$ and using that $(1-\chi)$ and its derivatives are zero on $B_{2/\epsilon}$, whereas the second inequality uses \eqref{eq:w2bound}.
We now estimate each of the three terms of the right-hand side of the last inequality.
We may use Lemma~\ref{le:Qhatpositive} and \ref{itm:h4} to find
\begin{align*}
\abs{x} \int_{B_{1/\abs{x}}\setminus B_{2/\epsilon}} \frac{d\xi}{\abs{\xi}^{n}\widehat{Q}_\rho(\xi)} &\leq C\abs{x}\int_{B_{1/\abs{x}}\setminus B_{2/\epsilon}} \frac{1}{\overline{\rho}(1/\abs{\xi})}\,d\xi
= C\abs{x} \int_{B_{1/\abs{x}}\setminus B_{2/\epsilon}} \frac{1}{\abs{\xi}^{n+\gamma-1}h_\rho(1/\abs{\xi})}\,d\xi\\
&\leq \frac{C\abs{x}}{h_\rho(\abs{x})}\int_{B_{1/\abs{x}}} \frac{1}{\abs{\xi}^{n+\gamma-1}}\,d\xi = \frac{C\abs{x}}{\abs{x}^{1-\gamma} h_\rho(\abs{x})} = \frac{C}{\abs{x}^{n-1}\rho(x)}.
\end{align*}
For the second term we use only Lemma~\ref{le:Qhatpositive} and find
\begin{align*}
\int_{\partial B_{1/\abs{x}}} \frac{d\Hcal^{n-1}(\xi)}{\abs{\xi}^{n}\widehat{Q}_\rho(\xi)} \leq C\int_{\partial B_{1/\abs{x}}} \frac{1}{\overline{\rho}(1/\abs{\xi})} \,d\Hcal^{n-1}(\xi) = \frac{C}{\abs{x}^{n-1}\rho(x)}.
\end{align*}
Finally, for the last term we compute with Lemma~\ref{le:Qhatpositive} and \ref{itm:h3}
\begin{align*}
\int_{B_{1/\abs{x}}^c} \frac{d\xi}{\abs{\xi}^{n+1}\widehat{Q}_\rho(\xi)} &\leq  C \int_{B_{1/\abs{x}}^c} \frac{1}{\abs{\xi}\overline{\rho}(1/\abs{\xi})}\,d\xi
 = C \int_{B_{1/\abs{x}}^c} \frac{1}{\abs{\xi}^{n+\sigma}g_\rho(1/\abs{\xi})}\,d\xi \\
&\leq  \frac{C}{g_\rho(\abs{x})} \int_{B_{1/\abs{x}}^c} \frac{1}{\abs{\xi}^{n+\sigma}}\,d\xi = \frac{C\abs{x}^\sigma}{g_\rho(\abs{x})} = \frac{C}{\abs{x}^{n-1}\rho(x)}.
\end{align*}
All in all, and using also \ref{itm:h3}, this shows that for $0<\abs{x} < \epsilon/2$
\begin{equation}\label{eq:kbound}
\abs{K(x)} = \frac{\abs{(\partial^{n}_1 W^2_\rho)^{\vee}(\abs{x}e_1)}}{(2\pi \abs{x})^{n}}\leq \frac{C}{\abs{x}^{2n-1}\rho(x)} \leq \frac{C}{\abs{x}^{n- \sigma}} ,
\end{equation}
which proves that $K$ is also integrable around the origin. \smallskip

\textit{Step 3: Bounds on $V_\rho$.} Since $V_\rho$ coincides with $K$ up to a smooth function we deduce from \eqref{eq:kbound} that the first inequality in \eqref{eq:Vrhobound} holds on any bounded set on which $\rho$ is positive, in particular, on $B_{\epsilon}\setminus\{0\}$. The bound on the gradient of $V_\rho$ follows from analogous calculations to that of step 2, since
\[
 \nabla K = ((2\pi i\cdot) \otimes W^2_\rho)^{\vee}.
\]
\end{proof}

When we apply Theorem \ref{th:NTFC} to Example \ref{ex:H4} \ref{itm:exH4a}, we recover and generalize the nonlocal fundamental theorem of calculus of \cite[Th.\ 4.5]{BeCuMo22}. We can also extend \eqref{eq:ftoc} to the setting of Sobolev spaces.
\begin{corollary}\label{cor:ftoc}
Let $\rho$ have compact support and satisfy \ref{itm:h0}--\ref{itm:h4}. Let $p \in [1,\infty]$ and $\Omega \subset \Rn$ be open and bounded. Then, for all $u \in H^{\rho,p}_0(\Omega)$ it holds that
\[
u(x) = \int_{\R^n} V_\rho(x-y)\cdot D_\rho u(y)\,dy \quad \text{for a.e.~$x \in \Rn$}.
\]
\end{corollary}
\begin{proof}
Since $\Omega$ is bounded, we find $H^{\rho,p}_0(\Omega) \subset H^{\rho,1}_0(\O)$, so it suffices to prove the statement for $u \in H^{\rho,1}_0(\O)$. This can be done by a simple mollification argument. Indeed, we can find a sequence $(\phi_j)_j \subset C_c^{\infty}(\O')$ for some open and bounded $\O' \subset \Rn$, such that $\phi_j \to u$ in $L^1(\R^n)$ and $\Gcal_\rho \phi_j \to D_\rho u$ in $L^1(\Rn,\Rn)$. Then, the fact that $(\Gcal_\rho \phi_j)_j$ is supported in a fixed compact set and $V_\rho$ is locally integrable yields by Young's convolution inequality
\[
\phi_j = \int_{\R^n} V_\rho(\cdot-y)\cdot \Gcal_\rho \phi_j(y)\,dy \to \int_{\R^n} V_\rho(\cdot-y)\cdot D_\rho u(y)\,dy \ \text{in $L^1_{\mathrm{loc}}(\Rn)$ as $j \to \infty$}.
\]
Hence, $u$ must coincide with the right-hand side, which proves the statement.
\end{proof}

\section{Embeddings}\label{se:embedding}

The aim of this section is to apply the nonlocal fundamental theorem of calculus (Theorem \ref{th:NTFC}) to prove embeddings, Poincar\'{e} inequalities and compactness results.
One of the advantages of this approach is that some of them can also be proven for $p=1$ and $p=\infty$, which is not possible with purely Fourier arguments, as in Sections \ref{se:PoincareL2} and \ref{se:PoincareLp}.
Moreover, the embeddings shown are not restricted to Lebesgue or H\"older spaces, but to the more general Orlicz spaces and spaces with a prescribed modulus of continuity.
Thus, the proof of those embeddings cannot be obtained by reducing to the fractional setting.

Throughout this section, we assume $\rho$ satisfies \ref{itm:h0}--\ref{itm:h4}. To start the analysis of the embeddings, we introduce the modulus of continuity $\omega: [0,\epsilon) \to [0, \infty)$
\[
 \omega(t) = \begin{cases}
  \frac{1}{t^{n-1}\overline{\rho}(t)} & \text{if } t \in (0, \epsilon) , \\
  0 & \text{if } t = 0 .
  \end{cases}
\]
It is a modulus of continuity in the sense that it is continuous by \ref{itm:h2} and \ref{itm:h3} and almost increasing by \ref{itm:h3}.
In view of \ref{itm:h3} and \ref{itm:h4}, we have that there exists $C>0$ such that for all $t \in (0,\epsilon)$ and $\lambda \in [1,\epsilon/t)$,
\begin{equation}\label{eq:doubling}
\frac{t^{\gamma}}{C}\leq \omega(t) \leq Ct^\sigma \quad \text{and} \quad \frac{\lambda^\sigma \omega(t)}{C} \leq \omega(\lambda t)\leq C\lambda^{\gamma} \omega(t) .
\end{equation}
The second inequalities show that we may take scaling factors outside $\omega$, which we will often use without mention.

\subsection{Embeddings into Orlicz spaces}

Our first result (Theorem \ref{th:HLS}) is an embedding into Orlicz spaces, together with its associated Poincar\'e inequality.
Its proof uses an inequality of the style of Hardy-Littlewood-Sobolev for generalized Riesz potentials proved in  \cite{KNS21}.

As usual in the theory of Orlicz spaces, we say that $A:[0,\infty) \to [0,\infty)$ is a Young function if it is continuous, strictly increasing, convex, with $A(0)=0$ and $\lim_{t \to \infty}A(t)=\infty$;
% (one could consider more general Young functions as well)
note that any Young function is invertible. Then, for $\Omega \subset \R^n$ open one can define the Orlicz space 
\[
 L^A(\Omega) = \{u:\Omega \to \R \ \text{measurable} \,:\, \norm{u}_{L^{A}(\Omega)} < \infty\},
\]
with the Luxemburg norm
\[
 \norm{u}_{L^{A}(\Omega)} = \inf\left\{\lambda > 0\,:\, \int_{\Omega} A\left(\frac{\abs{u(x)}}{\lambda}\right)\,dx\leq 1\right\}.
\]
Of course, if $A(t)=t^p$ with $p \in (1, \infty)$ then $L^{A}(\Omega)=L^p(\Omega)$ with the same norm.

In the proof of Theorem \ref{th:HLS} below, given a measurable function $\tilde{\omega}:(0,\infty) \to (0,\infty)$ we will use the operator $I_{\tilde{\omega}}$, sending measurable functions $u : \Rn \to \R$ to measurable functions $I_{\tilde{\omega}} (u) : \Rn \to \R$, defined as
\[
 I_{\tilde{\omega}} (u)(x) := \int_{\R^n} \frac{\tilde{\omega}(\abs{x-y})}{\abs{x-y}^n} u (y) \,dy , \qquad x \in \Rn .
\]
In particular, we will use the boundedness properties of this operator between Lebesgue and Orlicz spaces proved in \cite[Cor.\ 3.8\,(i)]{KNS21}, which we reproduce for ease of reference.

\begin{proposition}\label{pr:KNS}
Let $p \in (1, \infty)$.
Assume that:
\begin{enumerate}[label=(\alph*)]
\item\label{itm:NKSa} $\displaystyle \int_0^1 \frac{\tilde{\omega}(t)}{t} \, dt < \infty$.

\item\label{itm:NKSb} There exist $C, K_1, K_2 >0$ with $K_1 < K_2$ such that for all $r > 0$,
\[
 \sup_{r \leq t \leq 2r} \tilde{\omega}(t) \leq C \int_{K_1 r}^{K_2 r} \frac{\tilde{\omega}(t)}{t} \, dt .
\]

\item\label{itm:NKSc} There exists $C>0$ such that for all $r>0$,
\begin{equation}\label{eq:NKSc}
 \frac{1}{r^{n/p}} \int_{0}^r \frac{\tilde{\omega}(t)}{t}\,dt +\int_r^{\infty} \frac{\tilde{\omega}(t)}{t^{1+n/p}}\,dt \leq C A^{-1}(1/r^n) .
\end{equation}
\end{enumerate}
Then, $I_{\tilde{\omega}}$ is bounded from $L^p(\R^n)$ to $L^{A}(\R^n)$.
\end{proposition}

With the aid of Theorem \ref{th:NTFC}, we have the following embedding of $H^{\rho,p}_0(\Omega)$ into Orlicz spaces.

\begin{theorem}\label{th:HLS}
Let $\rho$ satisfy \ref{itm:h0}--\ref{itm:h4} and have compact support.
Let $\O$ be bounded.
Assume $p \in (1,\infty)$ satisfies $\gamma p < n$, and $A$ is a Young function such that
\begin{equation}\label{eq:youngident}
\liminf_{t \to \infty} \frac{A^{-1}(t)}{\omega\left(t^{-1/n}\right)t^{1/p}} >0.
\end{equation}
Then, $H^{\rho,p}_0(\Omega)$ is embedded into $L^{A}(\Omega)$ and there is a constant $C=C(\Omega, A, n, \rho)>0$ such that
\begin{equation}\label{eq:Orlicz}
 \norm{u}_{L^{A}(\Omega)} \leq C \norm{D_\rho u}_{L^p(\R^n,\R^n)} \quad \text{for all $u \in H^{\rho,p}_0(\Omega)$}.
\end{equation}
\end{theorem}

\begin{proof}
Define $\tilde{\omega}:(0,\infty) \to (0,\infty)$ as
\[
\tilde{\omega}(t)=\begin{cases}
\omega(t) \quad \text{for $t \in (0,\epsilon)$,} \\
e^{-t} \quad \text{for $t \in [\epsilon,\infty)$.} 
\end{cases}
\]
Bound \eqref{eq:Vrhobound} and the fact that $V_{\rho}$ is locally bounded away from the origin allows us to establish the estimate
\begin{equation}\label{eq:Vomega}
 V_{\rho} (x) \leq C \frac{\tilde{\omega} (\abs{x})}{\abs{x}^n}, \qquad x \in B_{\diam \O + \d} \setminus \{ 0 \},
\end{equation}
for a suitable constant $C>0$ and $\d>0$ given by $\supp \rho = \overline{B}_\d$.

Let $u \in H^{\rho,p}_0(\O)$.
By Corollary~\ref{cor:ftoc} and \eqref{eq:Vomega} we can estimate for a.e.~$x \in \Omega$
\begin{align*}
\abs{u(x)} \leq \int_{\Omega_\d} \abs{V_\d(x-y)} \abs{D_\rho u(y)}\,dy\leq C \int_{\R^n} \frac{\tilde{\omega}(\abs{x-y})}{\abs{x-y}^n}\abs{D_\rho u(y)}\,dy = C I_{\tilde{\omega}} (\abs{D_\rho u})(x).
\end{align*}
On the other hand, the bound $\abs{u} \leq C I_{\tilde{\omega}} (\abs{D_\rho u})$ is obvious in $\O^c$. Therefore, it suffices to show that the operator $I_{\tilde{\omega}}$ is bounded from $L^p(\R^n)$ into $L^A(\Omega)$, since we then find
\[
 \norm{u}_{L^{A}(\Omega)} \leq C \norm{I_{\tilde{\omega}} (\abs{D_\rho u})}_{L^{A}(\Omega)} \leq C \norm{\abs{D_\rho u}}_{L^p(\R^n)} = C \norm{D_\rho u}_{L^p(\R^n, \R^n)}.
\]

For the boundedness of $I_{\tilde{\omega}}$ from $L^p(\R^n)$ into $L^A(\Omega)$, we shall check \ref{itm:NKSa}--\ref{itm:NKSc} of Proposition \ref{pr:KNS}.
Condition \ref{itm:NKSa} holds thanks to \eqref{eq:doubling}, since
\[
 \int_0^{\e} \frac{\tilde{\omega} (t)}{t} \, dt \leq C \int_0^{\e} \frac{1}{t^{1-\sigma}} \, dt < \infty .
\]
As for \ref{itm:NKSb}, we calculate, for $r<\epsilon/2$,
\[
\int_{r/2}^{2r} \frac{\tilde{\omega}(t)}{t}\,dt \geq C\sup_{r \leq t \leq 2r}\tilde{\omega}(t) \int_{r/2}^{2r} \frac{1}{t}\,dt=C\log(4) \sup_{r \leq t \leq 2r}\tilde{\omega}(t),
\]
where we have used the second property in \eqref{eq:doubling}, while for $r>2\epsilon$
\[
\int_{r/2}^{2r} \frac{\tilde{\omega}(t)}{t}\,dt \geq \frac{1}{2r}\int_{r/2}^{2r} e^{-t}\,dt= \frac{e^{r/2}-e^{-r}}{2r} e^{-r} \geq C\sup_{r \leq t \leq 2r}\tilde{\omega}(t).
\]
In closing, the bound
\[
 \int_{r/2}^{2r} \frac{\tilde{\omega}(t)}{t}\,dt \geq C\sup_{r \leq t \leq 2r}\tilde{\omega}(t) , \qquad \frac{\e}{2} \leq r \leq 2 \e 
\]
holds trivially as $\tilde{\omega}$ is a piecewise smooth positive function.
Thus, condition \ref{itm:NKSb} is proved.

Finally, we check \ref{itm:NKSc}. 
In fact, since we are only interested in the embedding into $L^{A}(\Omega)$ for $\Omega$ bounded, it suffices to verify  inequality \eqref{eq:NKSc} for $r< r_0$ for some $r_0 < \epsilon$.
Indeed, for $r\geq r_0$ the left-hand side of \eqref{eq:NKSc} is bounded by a constant times $r^{-n/p}$, so that we can change $A$ around zero such that that the inequality is satisfied everywhere (cf.~\cite[Lemma~4.5]{Ada77}), which leads to an equivalent Orlicz space \cite[Th.\ V.1.3]{RaR91}.
For $r< r_0$ we compute
\begin{equation*}\label{eq:ineqOrlicz}
 \frac{1}{r^{n/p}} \int_{0}^r \frac{\tilde{\omega}(t)}{t}\,dt+\int_r^{\infty} \frac{\tilde{\omega}(t)}{t^{1+n/p}}\,dt
 \leq C\frac{\omega(r)}{r^{n/p+\sigma}}\int_0^r\frac{1}{t^{1-\sigma}}\,dt + C\frac{\omega(r)}{r^{\gamma}}\int_r^{\infty} \frac{1}{t^{1-\gamma+n/p}}\,dt = \frac{C \omega(r)}{r^{n/p}},
\end{equation*}
where we have used that $\tilde{\omega}(t)/t^\sigma$ is almost increasing on $(0, r_0)$, whereas $\tilde{\omega}(t)/t^{\gamma}$ is almost decreasing on $(r,\infty)$; note that the last inequality also uses that $\gamma < n/p$. All in all, this shows that it is sufficient to have
\[
\frac{\omega(r)}{r^{n/p}} \leq C A^{-1}(1/r^n),
\]
for all $r< r_0$.
The validity of such an inequality for some $r_0>0$ is a consequence of \eqref{eq:youngident}.
\end{proof}

%\begin{remark}
%From the proof of Theorem \ref{th:HLS}, we note that the condition $\gamma p <n$ is not strictly needed to find an embedding into an Orlicz space.
%Indeed, in view of \eqref{eq:ineqOrlicz}, it can be replaced with the inequality
%\[
% \frac{1}{r^{n/p}} \int_{0}^r \frac{\tilde{\omega}(t)}{t}\,dt + \int_r^{\infty} \frac{\tilde{\omega}(t)}{t^{1+n/p}}\,dt \leq C A^{-1}(1/r^n) ,
%\]
%for small $r$, which, naturally, is less explicit.
%\end{remark}

\begin{example}
We consider the following applications of Theorem~\ref{th:HLS}:
\begin{enumerate}[label=(\alph*)]
\item\label{item:Orlicza}
Let $\rho$ satisfy \ref{itm:h0}--\ref{itm:h4} and have compact support.
Assume $p, q \in (1,\infty)$ satisfy $\gamma p < n$, and
\begin{equation}\label{eq:qpn}
 \liminf_{t \to 0} t^{1 - \frac{1}{q} + \frac{1}{p}- \frac{1}{n}} \overline{\rho} (t)^{\frac{1}{n}} >0 .
\end{equation}
By Theorem \ref{th:HLS}, $H^{\rho,p}_0(\Omega)$ is embedded into $L^q (\Omega)$.
The assumptions are satisfied for the kernel \eqref{eq:truncatedRiesz} of Example \ref{ex:H4} \ref{itm:exH4a}.
In this case, $q = \frac{n p}{n - s p}$ and $\sigma = \gamma = s$.
Thus, the embedding \cite[Th.\ 6.1]{BeCuMo22} is recovered.

\item
Let $\rho$ satisfy \ref{itm:h0}--\ref{itm:h4} and have compact support.
Assume $p \in (1,\infty)$ satisfies $\gamma p < n$.
Using \ref{itm:h3}, we find that \eqref{eq:qpn} holds with $q = \frac{n p}{n - \sigma p}$.
Therefore, by \ref{item:Orlicza}, $H^{\rho,p}_0(\Omega)$ is embedded into $L^q (\Omega)$.

\item
Consider the kernel \eqref{eq:truncatedlog} of Example \ref{ex:H4} \ref{itm:exH4b}.
Assume, in addition, that $p > 1$ with $sp<n$.
Then, we may pick any $\gamma>s$ such that $\gamma p <n$.
Theorem~\ref{th:HLS} shows that $H^{\rho,p}_0(\Omega)$ embeds into the Orlicz space $L^{A}(\Omega)$, where
\[
A(t) = (t\textrm{lm}(t))^{p_*},
\]
$p_*=\frac{np}{n-sp}$ and $\textrm{lm}(t)$ is the modified logarithm function
\[
\textrm{lm}(t) = \begin{cases}
1/(1 - \log t) \quad &\text{for $t \leq 1$},\\
1+\log t \quad \quad&\text{for $t >1$}.
\end{cases}
\]
Indeed, checking that $A$ is a Young function is a routine calculation.
The inverse of $A$ for $t \geq 1$ is given by
\[
 A^{-1}(t) = \frac{t^{1/p_*}}{W(e t^{1/p_*})} ,
\]
with $W$ the Lambert $W$ function (the inverse function of $r\mapsto r e^r$ for $r\geq 0$).
Since $W$ behaves like $\log$ at infinity (i.e., $W(t)/\log t \to 1$ as $t \to \infty$), we can see that  
\[
A^{-1}(t) \geq C \frac{t^{1/p_*}}{\log t} \qquad \text{for $t$ large}.
\]
On the other hand, the corresponding $\omega$ satisfies
\[
 \omega (t) = \frac{t^s}{\bar{\chi} (t) (- \log t)} \leq C \frac{t^s}{-\log t} , \qquad t \in (0, \min\{ \e, 1 \} ) .
\]
The two inequalities above imply \eqref{eq:youngident} at once, so Theorem \ref{th:HLS} concludes the embedding of $H^{\rho,p}_0(\Omega)$ into $L^{A}(\Omega)$ and the validity of \eqref{eq:Orlicz}.
\end{enumerate}
\end{example}

\subsection{Embeddings into spaces of continuous functions}

The opposite case to that of Theorem \ref{th:HLS} that we present is not $\gamma p \geq n$, as one would desire, but $\sigma p>n$.
In this case, we actually have embeddings into spaces of continuous functions.
To show this, we first prove the following estimates for the integrals of $V_\rho$.
%Recall from \eqref{eq:doubling} the bounds $r^{\gamma -n/p} \leq C \omega(r)r^{-n/p} \leq C r^{\sigma -n/p}$ for $r \in (0, \epsilon)$.
 
\begin{lemma}\label{le:intVbound}
Let $\rho$ have compact support and satisfy \ref{itm:h0}--\ref{itm:h4}, $p \in (1,\infty]$ with $\sigma p>n$ and $R>0$. Then, there is a constant $C=C(n,p,\rho,R)>0$ such that:
\begin{enumerate}[label= (\roman*)]
\item\label{itm:intVbound1}
For all $r\in (0,\epsilon)$ and $\abs{\zeta} \leq r/2$,
\[
\norm{V_\rho}_{L^{p'}(B_r)} \leq C \omega(r)r^{-n/p} \quad \text{and} \quad \norm{V_\rho-V_\rho(\cdot + \zeta)}_{L^{p'}(B_R\setminus B_r)} \leq C\omega(r)r^{-n/p}.
\]

\item\label{itm:intVbound2}
For $\abs{\zeta} \leq \epsilon/3$,
\[
 \norm{V_\rho-V_\rho(\cdot + \zeta)}_{L^{p'}(B_R)} \leq  C\omega(\abs{\zeta}) \abs{\zeta}^{-n/p}.
\]
\end{enumerate}
\end{lemma}
\begin{proof}
Part \ref{itm:intVbound1}.
For the first bound we compute using \eqref{eq:Vrhobound}
\begin{equation*}
\norm{V_\rho}_{L^{p'}(B_r)} \leq C \left( \int_{B_r} \left(\frac{\omega(\abs{x})}{\abs{x}^n}\right)^{p'} dx \right)^{1/p'}
\leq \frac{C\omega(r)}{r^{\sigma}} \left( \int_{B_r} \frac{1}{\abs{x}^{(n-\sigma )p'}}\,dx \right)^{1/p'} = C\omega(r) r^{-n/p},
\end{equation*}
where the second inequality uses that $\omega(r)/r^\sigma$ is almost increasing by \ref{itm:h3}.

For the second inequality of the statement, we first note that we may restrict to integration over $B_{\epsilon/2} \setminus B_r$, since $V_\rho$ is Lipschitz continuous on $B_R \setminus B_{\epsilon/2}$; indeed, we have
\[
 \norm{V_\rho-V_\rho(\cdot + \zeta)}_{L^{p'}(B_R\setminus B_{\epsilon/2})} \leq C\abs{\zeta}\norm{\mathbbm{1}_{B_R}}_{L^{p'} (\Rn)} \leq C R^{\frac{n}{p'}} r \leq C r^{\gamma - n/p} \leq C \omega(r) r^{-n/p} ,
\]
where the third inequality holds because $\gamma-\frac{n}{p} < 1$, and the last one is due to \eqref{eq:doubling}.
As for the integration in $B_{\epsilon/2}\setminus B_{r}$, we first compute for $x \in B_{\epsilon/2}\setminus B_{r}$,
\[
 \abs{V_\rho(x)-V_\rho(x+\zeta)} \leq \abs{\zeta} \int_0^1 \abs{\nabla V_\rho(x+t\zeta)} \,dt 
\leq r \int_0^1 \frac{\omega(\abs{x+t\zeta})}{\abs{x+t\zeta}^{n+1}} \, dt ,
\]
by the fundamental theorem of calculus and \eqref{eq:Vrhobound}.
Therefore,
\[
 \norm{V_\rho-V_\rho(\cdot + \zeta)}_{L^{p'} (B_{\epsilon/2} \setminus B_r)} \leq r \left( \int_{B_{\epsilon/2} \setminus B_r} \left( \int_0^1 \frac{\omega(\abs{x+t\zeta})}{\abs{x+t\zeta}^{n+1}} \, dt \right)^{p'} dx \right)^{\frac{1}{p'}}
\]
and, by Jensen's inequality, Fubini's theorem, and the fact that $x+t\zeta \in B_\epsilon  \setminus B_{r/2}$ for each $ x \in B_{\epsilon/2} \setminus B_r$ and $\zeta \in B_{r/2}$,
\[
 \int_{B_{\epsilon/2} \setminus B_r} \left( \int_0^1 \frac{\omega(\abs{x+t\zeta})}{\abs{x+t\zeta}^{n+1}} \, dt \right)^{p'} dx \leq \int_0^1 \int_{B_{\epsilon/2} \setminus B_r} \left( \frac{\omega(\abs{x+t\zeta})}{\abs{x+t\zeta}^{n+1}} \right)^{p'} dx\, dt \leq \int_{B_\epsilon  \setminus B_{r/2}}  \left(\frac{\omega(\abs{x})}{\abs{x}^{n+1}}\right)^{p'} dx .
\]
Now, by the last inequality of \eqref{eq:doubling}, for all $x \in B_\epsilon \setminus B_{r/2}$,
\[
 \omega(\abs{x}) \leq C \left( \frac{\abs{x}}{r}\right)^{\gamma} \omega (r) ,
\]
so
\[
 \left( \int_{B_\epsilon \setminus B_{r/2}} \left(\frac{\omega(\abs{x})}{\abs{x}^{n+1}}\right)^{p'} dx \right)^{\frac{1}{p'}} \leq C \frac{\omega (r)}{r^{\gamma}} \left( \int_{B_{r/2}^c} \frac{1}{\abs{x}^{(n+1- \gamma) p'}} dx \right)^{\frac{1}{p'}} = C \omega (r) r^{-\frac{n}{p} - 1},
\]
since $(n+1- \gamma) p' > n p' \geq n$.
Altogether,
\[
 \norm{V_\rho-V_\rho(\cdot + \zeta)}_{L^{p'} (B_{\epsilon/2} \setminus B_r)} \leq C \omega (r) r^{-\frac{n}{p}} 
\]
and \ref{itm:intVbound1} is proved. \smallskip

Part \ref{itm:intVbound2}. We have
\[
 \norm{V_\rho-V_\rho(\cdot + \zeta)}_{L^{p'}(B_R)} \leq \norm{V_\rho}_{L^{p'}(B_{2 |\zeta|})} + \norm{V_\rho(\cdot + \zeta)}_{L^{p'}(B_{2 |\zeta|})}+ \norm{V_\rho-V_\rho(\cdot + \zeta)}_{L^{p'}(B_R \setminus B_{2 |\zeta|})} ,
\]
and each of the terms of the right-hand side can be estimated by part \ref{itm:intVbound1} and \eqref{eq:doubling} as follows:
\[
 \norm{V_\rho}_{L^{p'}(B_{2 |\zeta|})} + \norm{V_\rho-V_\rho(\cdot + \zeta)}_{L^{p'}(B_R \setminus B_{2 |\zeta|})} \leq 2 C \omega(2\abs{\zeta}) \abs{2\zeta}^{-n/p} \leq C \omega(\abs{\zeta}) \abs{\zeta}^{-n/p} 
\]
and
\[
 \norm{V_\rho(\cdot + \zeta)}_{L^{p'}(B_{2 |\zeta|})} \leq \norm{V_\rho}_{L^{p'}(B_{3 |\zeta|})} \leq C \omega(3\abs{\zeta}) \abs{3\zeta}^{-n/p} \leq C \omega(\abs{\zeta}) \abs{\zeta}^{-n/p} ,
\]
which completes the proof.
\end{proof}

The next step is to show an embedding into spaces of continuous functions.
We first make some observations and introduce the notation of the spaces of functions with a prescribed modulus of continuity.
For $\alpha \in [0, \sigma)$ we define the function $\omega_{\alpha} : [0, \infty) \to [0, \infty)$ as
\[
 \omega_{\alpha} (t) = \begin{cases}
 0 & \text{if } t = 0 , \\
 \omega (t) t^{-\alpha} & \text{if } t \in (0, \e) , \\
 \omega (\e) \e^{-\alpha} & \text{if } t \in [\e, \infty) .
 \end{cases}
 \]
Thanks to \eqref{eq:doubling}, $\omega_{\alpha}$ is continuous at $0$.
In fact, by \ref{itm:h2}, it is continuous, and by \ref{itm:h0}, it only vanishes at $0$.
Given $U \subset \Rn$, we define the space $C^{\omega_{\alpha}} (U)$ as the set of bounded functions $u : U \to \R$ such that there exists $C>0$ for which
\[
 \abs{u(x) - u(y)} \leq C \omega_{\alpha} (|x-y|) , \qquad x, y \in U,
\]
equipped with the seminorm
\[
 [u]_{C^{\omega_{\alpha}} (\O)} = \sup_{\substack{x, y \in U \\ x \neq y}} \frac{\abs{u(x) - u(y)}}{\omega_{\alpha} (|x-y|)}
\]
and the norm
\[
 \norm{u}_{C^{\omega_{\alpha}} (U)} = \norm{u}_{L^{\infty} (U)} + [u]_{C^{\omega_{\alpha}} (U)} .
\]
A standard argument shows that $C^{\omega_{\alpha}} (U)$ is a Banach space.
Moreover, as a consequence of \eqref{eq:doubling}, $C^{\omega_{\alpha}} (U)$ is embedded in the space $C^{0, \sigma - \alpha} (U)$ of bounded, H\"older continuous functions of exponent $\sigma - \alpha$.

With this language we prove, as a consequence of Lemma \ref{le:intVbound}, the following type of Morrey inequality.

\begin{theorem}\label{th:Morrey}
Let $\rho$ have compact support and satisfy \ref{itm:h0}--\ref{itm:h4}, $p \in (1,\infty]$ with $\sigma p>n$ and $\Omega \subset \R^n$ be open and bounded.
Let $\alpha = n/p$.
Then, any function in $H^{\rho, p}_0 (\O)$ admits a representative that is in $C^{\omega_{\alpha}} (\Rn)$.
Moreover, there exists a constant $C=C(n,p,\rho,\Omega)>0$ such that for all continuous $u \in H^{\rho, p}_0 (\O)$,
\begin{equation}\label{eq:Morrey}
 \norm{u}_{C^{\omega_{\alpha}} (\Rn)} \leq C \norm{D_\rho u}_{L^p(\R^n,\R^n)}.
\end{equation}
\end{theorem}
\begin{proof}
Let $u \in H^{\rho,p}_0(\O)$ and $R>0$ be such that $\Omega_{\d} \subset B_R$, where $\d>0$ is such that $\supp \rho =\overline{B}_\d$. Then, by Corollary~\ref{cor:ftoc} and Lemma \ref{le:intVbound} \ref{itm:intVbound2}, we find for a.e.~$x, z \in \Rn$ with $r:=\abs{x-z}<\eps/3$ that
\begin{equation}\label{eq:Morreyproof}
\begin{split}
\abs{u(x)-u(z)} &\leq \int_{B_R} \abs{V_\rho(x-y)-V_\rho(z-y)}\abs{D_\rho u(y)}\,dy \\
& \leq \norm{V_\rho - V_\rho(\cdot + x-z)}_{L^{p'} (B_{2R}; \Rn)} \norm{D_\rho u}_{L^p (\Rn,\Rn)} % \\ &
 \leq C \omega_{\alpha} (r) \norm{D_\rho u}_{L^p (\Rn, \Rn)}.
\end{split}
\end{equation}
In particular, there is a continuous representative $\bar{u} \in C(\Rn)$ of $u$, which also satisfies \eqref{eq:Morreyproof}. Since $\bar{u}=0$ in $\O^c$, we even find that
\[
\abs{\bar{u}(x)-\bar{u}(z)} \leq C \omega_{\alpha}(|x-z|) \norm{D_\rho \bar{u}}_{L^p (\R^n, \Rn)} , \qquad \text{for all $x, z \in \Rn$}.
\]
Taking $z \in \O^c$, also yields
\[
 \abs{\bar{u}(x)} \leq C \norm{D_\rho \bar{u}}_{L^p (\R^n, \Rn)} , \qquad \text{for all $x\in \Rn$},
\]
which finishes the proof.
\end{proof}

\begin{example}
Let $p \in (1,\infty]$ and $\Omega \subset \R^n$ be open and bounded.
Then, we have the inequality \eqref{eq:Morrey} for the following modulus of continuity $\omega_{\alpha}$ defined for $t \in (0, \e)$:

\begin{enumerate}[label=(\alph*)]
\item For $\rho$ given by \eqref{eq:truncatedRiesz}, as in Example \ref{ex:H4} \ref{itm:exH4a}, we obtain for $s p > n$,
\[
 \omega_{\alpha} (t) = t^{ s - \frac{n}{p}} .
\]
This is a generalization of \cite[Th.\ 6.3]{BeCuMo22}.

\item For $\rho$ given by \eqref{eq:truncatedlog}, as in Example \ref{ex:H4} \ref{itm:exH4b}, we obtain for $s p > n$,
\[
 \omega_{\alpha} (t) = \frac{t^{ s - \frac{n}{p}}}{-\log t} .
\]

\item For $\rho$ given by \eqref{eq:truncatedlogdown}, as in Example \ref{ex:H4} \ref{itm:exH4c}, we obtain for $s p > n$,
\[
 \omega_{\alpha} (t) = t^{ s - \frac{n}{p}} (-\log t) .
\]

\end{enumerate}
\end{example}

\subsection{Compact embeddings}

An immediate consequence of Theorem \ref{th:Morrey} is the compact inclusion of $H^{\rho, p}_0 (\O)$ into $L^p (\Rn)$ in the regime $\sigma p > n$, since the space of $C^{\omega_{\alpha}} (\Rn)$ functions vanishing in $\O^c$ is compactly embedded into $L^p (\O)$.
Thus, we recover part of the conclusion of Theorem \ref{th:lpbound}, but with stronger assumptions.
Here, we present an approach to prove compactness from $H^{\rho, p}_0 (\O)$ into $L^p (\Rn)$ for any $p \in [1, \infty]$.

A similar proof to the Morrey inequality of Theorem \ref{th:Morrey}, using Lemma~\ref{le:intVbound} as well, yields the following bounds on translations.

\begin{proposition}\label{pr:translations}
Let $\rho$ have compact support and satisfy \ref{itm:h0}--\ref{itm:h4}.
Let $p \in [1,\infty]$ and let $\Omega \subset \R^n$ be open and bounded.
Then, there exists a constant $C=C(n, p, \rho, \Omega)>0$ such that for all $\abs{\zeta} < \epsilon/3$,
\[
\norm{u-u(\cdot +\zeta)}_{L^p(\R^n)} \leq C \omega(\abs{\zeta})\norm{D_\rho u}_{L^p(\R^n,\R^n)} \quad \text{for all }  H^{\rho, p}_0 (\O) .
\]
\end{proposition}
\begin{proof}
Let $u \in H^{\rho,p}_0(\O)$ and $R>0$ be such that $\Omega_{\d} \subset B_R$ with $\d>0$ such that $\supp \rho = \overline{B}_\d$.
As in the first inequality of \eqref{eq:Morreyproof}, we find for a.e.~$x \in \O$
\[
\abs{u(x)-u(x + \zeta)} \leq \int_{B_{2R}} \abs{V_\rho(y) - V_\rho(y +  \zeta)}\abs{\Gcal_\rho u(x-y)}\,dy.
\]
Consequently, by Minkowski's integral inequality,
\begin{align*}
 \norm{u-u(\cdot +\zeta)}_{L^p(\R^n)} & \leq \int_{B_{2R}} \abs{V_\rho(y) - V_\rho(y +  \zeta)}\norm{\Gcal_\rho u(\cdot -y)}_{L^p (\Rn)} \,dy \\
 & = \int_{B_{2R}} \abs{V_\rho(y) - V_\rho(y +  \zeta)} \,dy \norm{\Gcal_\rho u}_{L^p (\R^n, \Rn)} \leq C \omega (\abs{\zeta}) \norm{\Gcal_\rho u}_{L^p (\R^n, \Rn)} ,
\end{align*}
where the last inequality is due to Lemma \ref{le:intVbound} \ref{itm:intVbound2}. 
\end{proof}

This result allows for an alternative proof of the compact inclusion from $H^{\rho, p}_0 (\O)$ into $L^p (\Rn)$, via the Fr\'{e}chet-Kolmogorov criterion (when $p < \infty$) or the Ascoli-Arzel\`a theorem (when $p = \infty$), as in \cite[Th.\ 7.3]{BeCuMo22}.
We do not provide a proof since it is standard.
Except for the cases $p = 1, \infty$, the following result requires stronger assumptions than those of Theorem \ref{th:lpbound}.

\begin{corollary}
Let $\rho$ have compact support and satisfy \ref{itm:h0}--\ref{itm:h4}.
Let $p \in [1,\infty]$ and let $\Omega \subset \R^n$ be open and bounded.
Then the inclusion from $H^{\rho, p}_0 (\O)$ into $L^p (\Rn)$ is compact.
\end{corollary}

%On the other hand, we note that the compact embeddings of Theorems \ref{th:l2bound} and \ref{th:lpbound} together with the continuous embeddings of this section (Theorems \ref{th:HLS} and \ref{th:Morrey}) give rise, via a standard interpolation argument, to compact embeddings into Orlicz spaces and spaces of continuous functions with a prescribed modulus of continuity.

\section{Inclusion between spaces for different kernels}\label{se:inclusion}

In this section, we compare the nonlocal Sobolev spaces induced by different kernels.
We first prove the following upper bound on $\widehat{Q}_\rho$, complementing Lemma~\ref{le:Qhatpositive}.

\begin{lemma}\label{le:qhatupper}
Let $\rho$ have compact support, be differentiable outside the origin and satisfy \ref{itm:h0}, \ref{itm:h1} and \ref{itm:h4}. Then, there is a $C=C(n, \rho)>0$ such that
\[
\widehat{Q}_\rho(\xi) \leq C \frac{\overline{\rho}(1/\abs{\xi})}{\abs{\xi}^{n}} \quad \text{for all $\xi \in B_{1/\epsilon}^c$}.
\]
\end{lemma}
\begin{proof}
For $\theta > 1/\epsilon$ and $0 < r < \frac{1}{2 \theta}$ we compute, thanks to \ref{itm:h4},
\begin{align*}
 f_\rho(r)-f_\rho\left(r+\frac{1}{2\theta}\right) 
 & = h_\rho(r)\left(\frac{1}{r^{1+\gamma}}-\frac{1}{\left(r+\frac{1}{2\theta}\right)^{1+\gamma}}\right)
 + \left(h_\rho(r)-h_\rho\left(r+\frac{1}{2\theta}\right)\right)\frac{1}{\left(r+\frac{1}{2\theta}\right)^{1+\gamma}} \\
 & \leq C h_\rho\left(\frac{1}{2\theta}\right)  \left(\frac{1}{r^{1+\gamma}}-\frac{1}{\left(r+\frac{1}{2\theta}\right)^{1+\gamma}}\right) 
 + C \frac{h_\rho\left(r+\frac{1}{2\theta}\right)}{\left(r+\frac{1}{2\theta}\right)^{1+\gamma}} \\
 & = C \frac{1}{(2 \theta)^{1 + \gamma}} f_\rho\left(\frac{1}{2\theta}\right)  \left(\frac{1}{r^{1+\gamma}}-\frac{1}{\left(r+\frac{1}{2\theta}\right)^{1+\gamma}}\right) 
 + C f_\rho\left(r+\frac{1}{2\theta}\right) .
\end{align*}
Consequently,
\begin{align*}
 \left( f_\rho(r)-f_\rho\left(r+\frac{1}{2\theta}\right) \right) & \sin(2\pi\theta r) \\ 
 & \leq \frac{C}{\theta^{1 + \gamma}} f_\rho\left(\frac{1}{2\theta}\right)  \left(\frac{1}{r^{1+\gamma}}-\frac{1}{\left(r+\frac{1}{2\theta}\right)^{1+\gamma}}\right) \sin(2\pi\theta r) 
 + C f_\rho\left(r+\frac{1}{2\theta}\right) .
\end{align*}
Since, as in \eqref{eq:oscillatory2},
\[
 \int_{0}^{\frac{1}{2\theta}}  \left(\frac{1}{r^{1+\gamma}}-\frac{1}{\left(r+\frac{1}{2\theta}\right)^{1+\gamma}}\right) \sin(2\pi\theta r) \, dr = 
 \int_{0}^{\frac{1}{\theta}} \frac{1}{r^{1+\gamma}} \sin(2\pi\theta r) \, dr \leq C \theta \int_{0}^{\frac{1}{\theta}} \frac{1}{r^{\gamma}} dr \leq C \theta^{\gamma}
\]
and, by \ref{itm:h1},
\[
 \int_{0}^{\frac{1}{2\theta}}   f_\rho\left(r+\frac{1}{2\theta}\right) \, dr \leq \frac{1}{2\theta} f_\rho\left(\frac{1}{2\theta}\right) ,
\]
we conclude that
\begin{equation*}
 \int_{0}^{\frac{1}{\theta}} f_\rho(r) \sin(2\pi\theta r)\,dr = \int_{0}^{\frac{1}{2\theta}} \left(f_\rho(r)-f_\rho\left(r+\frac{1}{2\theta}\right)\right)\sin(2\pi\theta r)\,dr \leq \frac{C}{\theta} f_\rho\left(\frac{1}{2\theta}\right).
\end{equation*}
Furthermore, with integration by parts we find
\begin{align*}
\int_{\frac{1}{\theta}}^{\infty} f_\rho(r)\sin(2\pi\theta r)\,dr &= \frac{1}{2\pi\theta}f_\rho\left(\frac{1}{\theta}\right)+\frac{1}{2\pi\theta}\int_{\frac{1}{\theta}}^\infty \frac{d}{dr}f_\rho(r)\cos(2 \pi \abs{\xi} r z_1)\,dr\\
& \leq \frac{C}{\theta} \left(f_\rho\left(\frac{1}{\theta}\right)+\int_{\frac{1}{\theta}}^{\infty} -\frac{d}{dr}f_\rho(r)\,dr\right) = \frac{C}{\theta} f_\rho\left(\frac{1}{\theta}\right)\leq \frac{C}{\theta} f_\rho\left(\frac{1}{2\theta}\right).
\end{align*}
Summing the previous two inequalities, we find that for all $\theta >1/\epsilon$,
\begin{equation}\label{eq:theta1epsilon}
\int_0^\infty f_\rho(r)\sin(2 \pi \theta r)\,dr \leq \frac{C}{\theta} f_\rho\left(\frac{1}{2\theta}\right).
\end{equation}
On the other hand, recalling  \eqref{eq:hatQcoarea}, we obtain that for $\xi \not = 0$,
\begin{align*}
\pi\abs{\xi}\widehat{Q}_\rho(\xi)&= \int_{\Sn_+} z_1 \int_0^\infty f_\rho (r)\sin(2 \pi \abs{\xi} r z_1)\,dr\,d \mc{H}^{n-1} (z)\\
&\leq \int_{\Sn_+ \cap \{z_1\abs{\xi} \leq 1/\epsilon\}} \frac{1}{\abs{\xi}\epsilon}\int_0^\infty f_\rho(r)\sin(2 \pi \abs{\xi} r z_1)\,dr\,d \mc{H}^{n-1} (z) \\
 &\qquad + \int_{\Sn_+ \cap \{z_1\abs{\xi} > 1/\epsilon\}} z_1\int_0^\infty f_\rho(r)\sin(2 \pi \abs{\xi} r z_1)\,dr\,d \mc{H}^{n-1} (z) .
\end{align*}
For the first term of the right-hand side, we use the inequality $\sin t \leq t$ (for $t>0$) as well as \eqref{eq:H0corbarrho} and the compact support of $f_\rho$, to obtain that when $0 < z_1\abs{\xi} \leq 1/\epsilon$,
\[
 \frac{1}{\abs{\xi}\epsilon}\int_0^\infty f_\rho(r)\sin(2 \pi \abs{\xi} r z_1)\,dr \leq  C \frac{1}{\abs{\xi}\epsilon}\abs{\xi}z_1 \leq  C \frac{1}{\abs{\xi}}.
\]
Hence, using that $ \mc{H}^{n-1}  (\Sn_+ \cap \{z_1\abs{\xi} \leq 1/\epsilon\}) \leq C / \abs{\xi}^{n-1}$, we have, for $\xi \in B_{1/\epsilon}^c$,
\[
  \int_{\Sn_+ \cap \{z_1\abs{\xi} \leq 1/\epsilon\}} \frac{1}{\abs{\xi}\epsilon}\int_0^\infty f_\rho(r)\sin(2 \pi \abs{\xi} r z_1)\,dr\,d \mc{H}^{n-1} (z) \leq \frac{C}{\abs{\xi}^n} \leq C \frac{\overline{\rho}(1/\abs{\xi})}{\abs{\xi}^{n-1}} , 
\]
where in the last equality we have used the assumption $\inf_{B_{\e}} \rho > 0$ of \ref{itm:h0}.
For the second term, we apply \eqref{eq:theta1epsilon} to obtain that, when $z_1\abs{\xi} > 1/\epsilon$,
\begin{equation*}
\begin{split}
 z_1\int_0^\infty f_\rho(r)\sin(2 \pi \abs{\xi} r z_1)\,dr &\leq \frac{C}{\abs{\xi}}f_\rho\left(\frac{1}{2\abs{\xi}z_1}\right) \leq \frac{C}{\abs{\xi}}f_\rho\left(\frac{1}{2 \abs{\xi}}\right) \\
 & =  \frac{C}{\abs{\xi}^{n-1}} \overline{\rho} \left(\frac{1}{2 \abs{\xi}}\right)\leq\frac{C}{\abs{\xi}^{n-1}} \overline{\rho} \left(\frac{1}{\abs{\xi}}\right),
 \end{split}
\end{equation*}
where we have used that $f_{\rho}$ is decreasing and, for the last inequality, \ref{itm:h4}.
The proof is concluded.
\end{proof}

As a consequence of this lemma, we may prove embeddings of our nonlocal spaces for different kernels.

\begin{theorem}\label{th:comparison}
Let $\O \subset \Rn$ be open.
Let $\rho_1,\rho_2$ have compact support and satisfy \ref{itm:h0}--\ref{itm:h1}; let $\rho_1$ satisfy \ref{itm:h4}. Assume $\liminf_{t \downarrow 0} \overline{\rho_2}(t)/\overline{\rho_1}(t)>0$ and some of the following:
\begin{enumerate}[label=(\roman*)]
\item\label{itm:comparison1} $p \in (1,\infty)$, $\rho_1,\rho_2$ satisfy \ref{itm:h2} and $\liminf_{t \downarrow0} t^{n-1}\overline{\rho_1}(t)>0$.

\item\label{itm:comparison2} $p=2$ and $\rho_1$ is differentiable outside the origin.
\end{enumerate}
Then, the continuous inclusion $H^{\rho_2,p}_0 (\O) \subset H^{\rho_1,p}_0 (\O)$ holds and there is a $C=C(n,\rho_1,\rho_2)>0$ such that
\[
\norm{D_{\rho_1} u}_{L^p(\R^n,\R^n)} \leq C \norm{D_{\rho_2} u}_{L^p(\R^n,\R^n)} \quad \text{for all } u \in H^{\rho_2,p}_0 (\O) .
\]
\end{theorem}
\begin{proof}
Let $\phi \in C^{\infty}_c (\O)$.
By Proposition \ref{pr:Qhat} \ref{item:Qhat2} and Lemma \ref{le:Qhatpositive}, we can write 
\[
\widehat{\Gcal_{\rho_1} \phi} = \frac{\widehat{Q}_{\rho_1}}{\widehat{Q}_{\rho_2}}\widehat{\Gcal_{\rho_2} \phi}.
\]
Therefore, if 
\[
m:= \frac{\widehat{Q}_{\rho_1}}{\widehat{Q}_{\rho_2}}
\]
is an $L^p$ Fourier multiplier, the result follows readily from a density argument based on Theorem \ref{th:density} \ref{item:density1}.
In the case $p=2$ (part \ref{itm:comparison2}), we only have to check that $m$ is bounded, while for general $p$ (part \ref{itm:comparison1}), we shall see that $m$ satisfies the hypotheses of the Mihlin-H\"{o}rmander theorem (cf.~\cite[Th.\ 6.2.7]{Gra14a}).

To show that $m$ is bounded, we invoke Lemma~\ref{le:qhatupper} for $\rho_1$ and Lemma~\ref{le:Qhatpositive} for $\rho_2$ to obtain for all $\xi \in B_{2/\epsilon}^c$,
\[
 m (\xi) \leq  C \frac{\overline{\rho_1}(1/\abs{\xi})}{\overline{\rho_2}(1/\abs{\xi})} \leq C ,
\]
where in the last inequality we have used $\liminf_{t \downarrow 0} \overline{\rho_2}(t)/\overline{\rho_1}(t)>0$.
The fact that $m$ is bounded in $B_{2/\epsilon}$ is a consequence of Lemma~\ref{le:Qhatpositive}.
This completes the proof under assumption \ref{itm:comparison2}.

For $p \in (1, \infty)$, to check the hypotheses of the Mihlin-H\"{o}rmander theorem we need to verify that
\[
\abs{\partial^{\alpha} m(\xi)} \leq C_{\alpha} \abs{\xi}^{-\abs{\alpha}},
\]
for all multiindices $\alpha \in \N^n$ with $\abs{\alpha} \leq n/2+1$. We note that this condition holds for $\xi \in B_1$, since both $\widehat{Q}_{\rho_1}$ and $\widehat{Q}_{\rho_2}$ are smooth and $\widehat{Q}_{\rho_2}$ is positive. Furthermore, the assumptions imply
\[
 \liminf_{t \downarrow0} t^{n-1}\overline{\rho_1}(t)>0, \qquad \liminf_{t \downarrow0} t^{n-1}\overline{\rho_2}(t)>0 ,
\]
so we can use \eqref{eq:hatQxi-1} as well as Lemma~\ref{le:Qhatderivdecay} to obtain that, for any $\beta \in \N^n$ and $\xi \in B_1^c$,
\[
\abs{\partial^{\beta} \widehat{Q}_{\rho_1}(\xi)} \leq C_\beta \abs{\xi}^{-\abs{\beta}} \abs{\widehat{Q}_{\rho_1}(\xi)} \quad \text{and} \quad \abs{\partial^{\beta} \widehat{Q}_{\rho_2}(\xi)} \leq C_\beta \abs{\xi}^{-\abs{\beta}} \abs{\widehat{Q}_{\rho_2}(\xi)} .
\]
A straightforward yet tedious calculation now shows that for $\abs{\xi} \geq 1$
\begin{equation*}\label{eq:redcution}
\abslr{\partial^{\alpha} \left(\frac{\widehat{Q}_{\rho_1}}{\widehat{Q}_{\rho_2}}\right)(\xi)} \leq C_\alpha\abs{\xi}^{-\abs{\alpha}}\abslr{\frac{\widehat{Q}_{\rho_1}(\xi)}{\widehat{Q}_{\rho_2}(\xi)}}=C_\alpha\abs{\xi}^{-\abs{\alpha}}m(\xi) \leq C_\alpha\abs{\xi}^{-\abs{\alpha}},
\end{equation*}
where in the last inequality we use that $m$ is bounded. This completes the proof under assumption \ref{itm:comparison1}.
\end{proof}

%
%\begin{corollary}
%Let $p \in (1,\infty)$ and $\rho$ have compact support and satisfy (H0)-(H3). For every $s \in (0,1)$ with $\liminf_{t \downarrow 0}t^{n+s-1}\overline{\rho}(t)>0$ there is a $C=C(\Omega,n,s,\rho)>0$ such that
%\[
%\norm{u}_{H^{s,p}(\R^n)} \leq C \norm{\Gcal_\rho u}_{L^p(\R^n;\R^n)} \quad \text{for all $u \in H^{\rho,p}_0(\Omega)$}.
%\]
%\end{corollary}
%\begin{proof}
%By setting $\rho_2=\rho$ and $\rho_1$ a truncated version of the fractional kernel $1/\abs{\cdot}^{n+s-1}$, the result follows from Theorem~\ref{th:lpbound}, the Poincar\'{e} inequality for truncated fractional gradients \cite[Th.\ 6.2]{BeCuMo22} and the equivalence of the spaces defined through truncated fractional gradients and the Bessel potential spaces \cite[Lemma~2.16]{CKS23}.
%\end{proof}

\begin{example}
Let $\O \subset \Rn$ be open, $s \in (0,1)$ and $p \in (1, \infty)$.
As in Examples \ref{ex:H4} \ref{itm:exH4a}--\ref{itm:exH4c}, consider $\rho_1$ as in \eqref{eq:truncatedRiesz}, $\rho_2$ as in \eqref{eq:truncatedlog} and $\rho_3$ as in \eqref{eq:truncatedlogdown}.
Then $H^{\rho_2,p}_0(\Omega) \subset H^{\rho_1,p}_0(\Omega) \subset H^{\rho_3,p}_0(\Omega)$.
Moreover, if $\rho'_2$ is as in \eqref{eq:truncatedlog} but for a exponent $s' \in (0, s)$, then $H^{\rho_3,p}_0(\Omega) \subset H^{\rho'_2,p}_0(\Omega)$.
\end{example}

Of course, changing the roles of $\rho_1$ and $\rho_2$ in Theorem \ref{th:comparison} gives rise to a criterion of equality of spaces, which complements that of Proposition \ref{prop:carryover}.

\begin{corollary}\label{th:comparison2}
Let $\O \subset \Rn$ be open.
Let $\rho_1,\rho_2$ have compact support and satisfy \ref{itm:h0}--\ref{itm:h1} and \ref{itm:h4}.
Assume
\[
 0 < \liminf_{t \downarrow 0} \frac{\overline{\rho_2}(t)}{\overline{\rho_1}(t)} \leq \limsup_{t \downarrow 0} \frac{\overline{\rho_2}(t)}{\overline{\rho_1}(t)} < \infty
\]
and some of the following:
\begin{enumerate}[label=(\roman*)]
\item $p \in (1,\infty)$, $\rho_1,\rho_2$ satisfy \ref{itm:h2} and $\liminf_{t \downarrow0} t^{n-1}\overline{\rho_1}(t)>0.$

\item $p=2$ and $\rho_1, \rho_2$ are differentiable outside the origin.
\end{enumerate}
Then $H^{\rho_1,p}_0(\Omega) = H^{\rho_2,p}_0(\Omega)$ and there is a $C=C(n,p, \rho_1,\rho_2)>0$ such that
\[
  \frac{1}{C} \norm{D_{\rho_2} u}_{L^p(\R^n,\R^n)} \leq \norm{D_{\rho_1} u}_{L^p(\R^n,\R^n)} \leq C \norm{D_{\rho_2} u}_{L^p(\R^n,\R^n)} \quad \text{for all } u \in H^{\rho_1,p}_0(\Omega) .
\]
\end{corollary}

We finish this section by showing a partial converse of Theorem~\ref{th:l2bound}.

%by noting that if $\rho$ satisfies the assumptions of Lemma~\ref{le:qhatupper} and $p=2$, then the conditions in \ref{item:Poincare1} and \ref{item:Poincare2} in Theorem~\ref{th:l2bound} are also necessary to get a Poincar\'{e} inequality or compact embedding, respectively.
%This is stated in the following result.

\begin{proposition}\label{pr:Poincareconverse}
Let $\Omega \subset \R^n$ be open and bounded.
Let $\rho$ have compact support, be differentiable outside the origin and satisfy \ref{itm:h0}, \ref{itm:h1} and \ref{itm:h4}.
Then, the following two statements hold:
\begin{enumerate}[label = (\roman*)]
\item\label{item:Poincareconverse1} If there is a $C=C(\Omega,n,\rho)>0$ such that
\[
\norm{u}_{L^2(\Omega)} \leq C \norm{D_\rho u}_{L^2(\R^n,\R^n)} \quad \text{for all $u \in H^\rho_0(\Omega)$},
\]
then $\limsup_{t \downarrow0} t^{n-1}\overline{\rho}(t)>0$.

\item\label{item:Poincareconverse2} If $H^\rho_0(\Omega)$ is compactly embedded into $L^2(\R^n)$, then $\limsup_{t \downarrow 0}t^{n-1}\overline{\rho}(t)=\infty$.
\end{enumerate}
\end{proposition}
\begin{proof}
Part \ref{item:Poincareconverse1}. We assume to the contrary that $\lim_{t \downarrow 0}t^{n-1} \overline{\rho}(t) = 0$.
Then, we obtain thanks to Lemma \ref{le:qhatupper} that the function $\lambda_\rho(\xi) = 2\pi i \xi \widehat{Q}_{\rho}(\xi)$ (see Proposition \ref{pr:Qhat}) is a bounded function with $\abs{\lambda_\rho(\xi)} \to 0$ as $\abs{\xi} \to \infty$.
Therefore, in light of Lemma~\ref{le:lamda}, we find
\[
 \lim_{R \to \infty} \sup_{\substack{u \in C^{\infty}_c (\R^n) \\ \norm{u}_{L^2 (\Rn)} \leq 1}} \norm{\widehat{\Gcal_{\rho} u}}_{L^2(B_R^c)} \leq \lim_{R\to \infty} \sup_{\abs{\xi} \geq R} \abs{\lambda_\rho(\xi)} =0,
\]
which implies by the Fr\'{e}chet-Kolmogorov criterion in Fourier space (cf.~\cite[Theorem~3]{Peg85}) and the compact support of $\rho$, that $D_{\rho}$ is a compact operator from $L^2(\O)$ (extended as zero in $\O^c$) to $L^2(\R^n,\R^n)$.
Therefore, a Poincar\'{e} inequality is not possible. \smallskip

Part \ref{item:Poincareconverse2}. If $\limsup_{t \downarrow 0} t^{n-1}\overline{\rho}(t)<\infty$, then the function $\lambda_\rho$ is bounded.
We infer that $D_\rho$ is a bounded operator from $L^2(\Rn)$ to $L^2(\Rn,\Rn)$, as a composition of the following bounded operators in $L^2$, initially defined for $u \in C^{\infty}_c (\Rn)$:
\[
 u \mapsto \widehat{u} \mapsto \lambda_{\rho} \widehat{u} \mapsto (\lambda_{\rho} \widehat{u})^\vee = \Gcal_{\rho} u ;
\]
see Lemma \ref{le:lamda}.
Since $D_\rho$ is bounded from $L^2(\Rn)$ to $L^2 (\R^n, \Rn)$, a compact embedding is not possible.
\end{proof}

\section*{Acknowledgements} 

J.C.B. has been supported by the Agencia Estatal de Investigaci\'on of the Spanish Ministry of Research and Innovation, through project PID2020-116207GB-I00.
C.M.-C. has been supported by the Agencia Estatal de Investigaci\'on of the Spanish Ministry of Research and Innovation, through projects and PID2021-124195NB-C32 (also H.S.) and the Severo Ochoa Programme for Centres of Excellence in R\&D CEX2019-000904-S (also H.S.), by the Madrid Government (Comunidad de Madrid, Spain) under the multiannual Agreement with UAM in the line for the Excellence of the University Research Staff in the context of the V PRICIT (Regional Programme of Research and Technological Innovation), and by the ERC Advanced Grant 834728. H.S. acknowledges funding for his research visit to Madrid in 2023 by the Bayerische Forschungsallianz, through the project BayIntAn\_KUEI\_2023\_16, and also acknowledges the hospitality of ICMAT (Instituto de Ciencias Matem\'{a}ticas), which hosted him during this period. In addition, H.S. thanks Adolfo Arroyo-Rabasa for the helpful discussions that led to the inclusion of Remark~\ref{rem:QrhoBessel}.

\bibliographystyle{abbrv}
\bibliography{FTOCbib}
\end{document}